\documentclass[11pt]{amsart}

\usepackage{color}
\usepackage{amsmath}
\usepackage{amsfonts}
\usepackage{mathrsfs}
\usepackage{amsthm}
\usepackage{amssymb}
\usepackage{bbm}
\usepackage{amstext}
\usepackage[normalem]{ulem}
\usepackage{tabularx}
\usepackage{faktor}
\usepackage{array, ltablex, multirow}
\usepackage{makecell}
\setcellgapes{6pt}
\makegapedcells

\newcommand{\F}{\mathcal{F}}

\newcommand{\I}{\mathcal{I}}

\newcommand{\C}{\mathbb{C}}

\newcommand{\Z}{\mathbb{Z}}
\newcommand{\N}{\mathbb{N}}
\newcommand{\A}{\mathscr{A}}
\newcommand{\Aa}{\mathcal{A}}

\newcommand{\M}{\mathcal{M}}

\newcommand{\E}{\mathbb{E}}
\newcommand{\D}{\mathscr{D}}
\newcommand{\B}{\mathscr{B}}

\newcommand{\Orb}{\mathcal{O}}
\newcommand{\1}{\mathbbm{1}}

\newcommand{\Hilb}{\mathcal{H}}

\newcommand{\xmt}{(X,\mu,T)}
\newcommand{\yns}{(Y,\nu,S)}

\newcommand{\zmr}{(Z,m,R)}

\newcommand{\Erdos}{Erd\H{o}s}
\newcommand{\Folner}{F\o{}lner}

\newcommand{\gen}{\texttt{\textup{gen}}}
\newcommand{\supp}{\texttt{\textup{supp}}}
\newcommand{\Ball}{\texttt{\textup{B}}}
\newcommand{\Dist}{\texttt{\textup{D}}}

\newcommand{\ev}{\textup{ev}}

\usepackage{soul}

\numberwithin{equation}{section}

\usepackage[colorlinks=true,urlcolor=blue,linkcolor=black,citecolor=blue]{hyperref}
\usepackage[capitalize]{cleveref}
\newtheorem{theorem}{Theorem}[section]

\newtheorem{proposition}[theorem]{Proposition}

\newtheorem{lemma}[theorem]{Lemma}
\newtheorem{corollary}[theorem]{Corollary}

\theoremstyle{definition}
\newtheorem{defn}[theorem]{Definition}
\newtheorem{remark}[theorem]{Remark}
\newtheorem{example}[theorem]{Example}
\newtheorem{question}[theorem]{Question}

\theoremstyle{plain}
\newtheorem*{namedthm}{\namedthmname}
\newcounter{namedthm}
\makeatletter
\newenvironment{named}[2]
{\def\namedthmname{#1}
\refstepcounter{namedthm}
\namedthm[#2]\def\@currentlabel{#1}}
{\endnamedthm}
\makeatother

\usepackage{relsize}
\usepackage{mathtools}
\usepackage{scalerel}
\newcommand{\rtrdot}{\hspace*{1mm}\mathrlap{\mkern0.9mu{\mathlarger{\cdot}}}
\mathlarger{\triangleright}\hspace*{0.75mm}}
\newcommand{\ltrdot}{\hspace*
{0.75mm}\mathrlap{\mkern3.3mu{\mathlarger{\cdot}}}
\mathlarger{\triangleleft}\hspace*{1mm}}

\renewcommand{\d}{~\mathrm{d}}
\renewcommand{\epsilon}{\varepsilon}

\makeatletter
\renewcommand{\tocsection}[3]{%
  \indentlabel{\@ifnotempty{#2}{\bfseries\ignorespaces#1 #2\quad}}\bfseries#3}
\renewcommand{\tocsubsection}[3]{%
  \indentlabel{\@ifnotempty{#2}{\ignorespaces#1 #2\quad}}#3}

\newcommand\@dotsep{4.5}
\def\@tocline#1#2#3#4#5#6#7{\relax
  \ifnum #1>\c@tocdepth 
  \else
    \par \addpenalty\@secpenalty\addvspace{#2}%
    \begingroup \hyphenpenalty\@M
    \@ifempty{#4}{%
      \@tempdima\csname r@tocindent\number#1\endcsname\relax
    }{%
      \@tempdima#4\relax
    }%
    \parindent\z@ \leftskip#3\relax \advance\leftskip\@tempdima\relax
    \rightskip\@pnumwidth plus1em \parfillskip-\@pnumwidth
    #5\leavevmode\hskip-\@tempdima{#6}\nobreak
    \leaders\hbox{$\m@th\mkern \@dotsep mu\hbox{.}\mkern \@dotsep mu$}\hfill
    \nobreak
    \hbox to\@pnumwidth{\@tocpagenum{\ifnum#1=1\bfseries\fi#7}}\par
    \nobreak
    \endgroup
  \fi}
\AtBeginDocument{%
\expandafter\renewcommand\csname r@tocindent0\endcsname{0pt}
}
\def\l@subsection{\@tocline{2}{0pt}{2.5pc}{5pc}{}}
\makeatother

\usepackage{lipsum}

\usepackage{todonotes}
\usepackage{marginnote}

\usepackage{cancel}
\usepackage{tikz}

\title{Finding product sets in some classes of amenable groups}
\author{Dimitrios Charamaras and Andreas Mountakis}
\date{\today}


\begin{document}

\maketitle

\vspace*{-0.02cm}
\begin{abstract}
In \cite{kmrr2}, using methods from ergodic theory, a longstanding
conjecture of \Erdos{} 
(see \cite[Page 305]{erdos1973})
about sumsets in large subsets of the natural numbers was resolved.
In this paper, we extend this result to several important classes
of amenable groups, including  all finitely generated 
virtually nilpotent groups, and all abelian groups $(G,+)$
with the property that the subgroup $2G := \{g+g : g\in G\}$
has finite index. We prove that in any group $G$ from the above classes,
any $A\subset G$ with positive upper Banach density contains 
a shifted product set of the form
$\{tb_ib_j\colon i<j\}$, for some infinite sequence
$(b_n)_{n\in\N}$ and some $t\in G$.
In fact, we show this result for all amenable groups that posses
a property which we call square absolute continuity.
Our results provide answers to several questions and conjectures
posed in \cite{kmrr_survey}.
\end{abstract}

\tableofcontents

\section{Introduction}
In \cite{kmrr2}, Bryna Kra, Joel Moreira, Florian K. Richter 
and Donald Robertson, using methods from ergodic
theory, proved that every subset $A$ of the 
positive integers with positive upper Banach density
contains $\{b_1+b_2+t: b_1\neq b_2\in B\}$ for some infinite set
$B\subset A$ and some $t\in\N$. This resolved a longstanding
conjecture of \Erdos{} (see \cite[Page 305]{erdos1973}).
\par
A natural question to ask is whether this result generalizes to 
other countable groups, such as $\Z^d$ for $d\geq 2$ 
or the discrete Heisenberg group, for example. 
The purpose of this paper is to extend the result in \cite{kmrr2} 
to several important classes of amenable groups, 
including all finitely generated virtually nilpotent groups and all abelian groups $(G,+)$
with the property that the subgroup consisting of the elements $2g:=g+g$,
where $g\in G$, has finite index. 
To this end, we first extend the result to all amenable groups 
satisfying a property that we call square absolute continuity (see 
\cref{sac_def}). Then we show that our result 
applies to the aforementioned classes of groups,
by showing that they are (virtually) square absolutely continuous.
Our main results provide partial answers to some questions and conjectures
posed in \cite{kmrr_survey} regarding product sets in large subsets of amenable groups.
\par
Throughout, let $G$ denote a countable group.
Let us start with some basic definitions. 
\begin{defn}
Let $(G,\cdot)$ be a group. A sequence 
$\Phi=(\Phi_N)_{N\in\N}$ of finite subsets 
of $G$ is:
\begin{itemize}
    \item a \textit{left \Folner{} sequence}, if it satisfies
    $$\lim_{N\to\infty}\frac{|g\Phi_N\cap\Phi_N|}{|\Phi_N|}=1
    \text{\quad or equivalently, \quad}
    \lim_{N\to\infty}\frac{|g\Phi_N\triangle\Phi_N|}{|\Phi_N|}=0$$
    for any $g\in G$, and
    \item a \textit{right \Folner{} sequence}, if it satisfies
    $$\lim_{N\to\infty}\frac{|\Phi_N g\cap\Phi_N|}{|\Phi_N|}=1
    \text{\quad or equivalently, \quad}
    \lim_{N\to\infty}\frac{|\Phi_N g\triangle\Phi_N|}{|\Phi_N|}=0$$
    for any $g\in G$.
\end{itemize}
If both conditions are satisfied, then $\Phi$ is 
a \textit{two-sided \Folner{} sequence}.
\end{defn}

We remark that if a group admits a left (or right) \Folner{} sequence, 
then it admits a 
two-sided \Folner{} sequence.
Amenable groups, which are 
the central object of our study, are defined as follows:

\begin{defn}
A group $(G, \cdot)$ is called amenable if it admits a 
left \Folner{} sequence.
\end{defn}

The most common example of an amenable group is $\Z^d$, for any $d\in\N$.
Other examples of amenable groups are finite groups, abelian groups,
solvable groups and finitely generated groups of subexponential growth.
In addition, products of amenable groups, and virtually amenable groups are amenable.
\Folner{} sequences are useful to define notions of density in amenable groups.

\begin{defn}\label{upper B density def}
Let $(G,\cdot)$ be an amenable group, $\Phi$ a left (right) \Folner{}
sequence, and let $A\subset G$. Then the
\textit{left (right) upper density} of $A$ with respect to $\Phi$ 
is defined as
$$\overline{d}_{\Phi}(A):=\limsup_{N\to \infty} 
\frac{|A\cap \Phi_N|}{|\Phi_N|}.$$
We say that $A$ has \textit{positive left (right) upper Banach 
density} if it has positive left (right) upper density with
respect to some left (right) \Folner{} sequence.
We also say that $A$ has {\em positive upper Banach density}
if it has positive upper density with respect to some two-sided 
\Folner{} sequence.
\end{defn}
Note that if $G$ is an amenable group, and $A\subset G$ has positive
left upper Banach density, then this does not necessarily mean that $A$ has positive 
right upper Banach density.

Given a group $(G,\cdot)$ and a sequence
$B=(b_n)_{n\in\N}\subset G$, we define
$$ B\ltrdot B := \{b_ib_j : i<j\},$$
$$B \rtrdot B := \{b_ib_j : i>j\}$$
and
$$B \odot B := \{b_ib_j : i\neq j\}.$$
If $G$ is abelian then $B\ltrdot B=B\rtrdot B=B\odot B$, which we also denote by 
$B\oplus B$ if the group operation in $G$ is written using additive notation.
We refer to the map $s_G:G\to G, s_G (g)=g^2$, as the 
{\em squaring map on $G$}. The image of this map is the subset of $G$ consisting
of all the elements of the form $g^2$, where $g\in G$. We denote this by $G^2$,
i.e., $s_G(G)=G^2$, and we often refer to it as the {\em subset of squares}.

\begin{defn}\label{sac_def}
Let $G$ be an amenable group and $\phi:G\to G$ be a map.
We say that $G$ is {\em $\phi$-absolutely continuous}
if $G$ admits two left \Folner{} sequences
$\Phi=(\Phi_N)_{N\in\N}$ and 
$\Psi=(\Psi_N)_{N\in\N}$
satisfying the following:
for any $\epsilon>0$ there exists some $\delta>0$ 
such that 
for any $u:G\to [0,1]$ satisfying 
$$\limsup_{N\to\infty}\frac{1}{|\Phi_N|}\sum_{g\in\Phi_N} u(g) < \delta,$$ 
we have that
$$\limsup_{N\to\infty}\frac{1}{|\Psi_N|}\sum_{g\in\Psi_N} u(\phi(g)) < \epsilon.$$
If, in particular, $\phi=s_G$, then we say that $G$ is 
{\em square absolutely continuous}.
\end{defn}

\begin{remark}\label{sac_preserved}
For a set $C\subset G$, we denote by $C^{-1}$ the set 
$C^{-1}:=\{g^{-1} : g \in C \}$. Note that when $\phi$ is the squaring map $\phi=s_G$, the 
existence of two left \Folner{} sequences in 
\cref{sac_def} is equivalent to the existence of 
two right \Folner{} sequences: for any pair $\Phi$, $\Psi$ of 
left \Folner{} sequences, the pair $\Phi^{-1}=(\Phi_N ^{-1})_{N\in\N}$, 
$\Psi^{-1}=(\Psi_N ^{-1})_{N\in\N}$ is a pair of right \Folner{} sequences,
and if $\Phi$, $\Psi$ satisfy the conditions of \cref{sac_def}, then 
so do $\Phi^{-1}$, $\Psi^{-1}$. More precisely, given $\epsilon>0$, 
there is some $\delta>0$ so that the conditions of \cref{sac_def}
are satisfied for $\Phi, \Psi$. Let $u: G \to [0,1]$ satisfy 
$\limsup_{N\to\infty}\frac{1}{|\Phi_N^{-1}|}\sum_{g\in\Phi_N^{-1}} 
u(g) < \delta$ and consider the map $w: G \to [0,1],\: w(g):= u(g^{-1})$.
Then $\limsup_{N\to\infty}\frac{1}{|\Phi_N|}\sum_{g\in\Phi_N} w(g) < \delta$,
and therefore 
$\limsup_{N\to\infty}\frac{1}{|\Psi_N|}\sum_{g\in\Psi_N} w(g^2) < \epsilon$,
which in turn implies that 
$\limsup_{N\to\infty}\frac{1}{|\Psi_N^{-1}|}
\sum_{g\in\Psi_N^{-1}} u(g^2) < \epsilon$.
\end{remark}

\subsection{Main results}

Throughout, we say that a sequence $(b_n)_{n\in\N}$ in $G$ is {\em infinite} if the set $\{b_n\colon n\in\N\}$ is infinite.
The first main theorem of this paper is the following:
\begin{theorem}\label{mt1}
Let $G$ be a square absolutely continuous group
and $A\subset G$ with positive left upper Banach density. 
Then there exist an infinite sequence
$B=(b_n)_{n\in \N}\subset A$ and some $t\in G$ such that
$$B \ltrdot B \subset t^{-1}A.$$
\end{theorem}

\cref{mt1} provides a positive answer to \cite[Question 5.16]{kmrr_survey},
with the extra assumption that $G$ is square absolutely continuous, 
and under the weaker assumption that the set $A$ has positive 
left upper Banach density, instead of positive upper Banach density.

\begin{remark}
Note that \cref{mt1} immediately implies an analogous result
for right upper Banach density instead of left upper Banach density.
Indeed, since, by \cref{sac_preserved}, square absolute continuity is preserved through the map $g\mapsto g^{-1}$, \cref{mt1}
is equivalent to the assertion that for any $A\subset G$ 
with positive right upper Banach density, 
there exists an infinite sequence $C = (c_n)_{n\in\N}\subset A$ 
and some $r\in G$ such that 
$$C\rtrdot C\subset Ar^{-1}.$$
\end{remark}

Before continuing, let us recall the following definitions for a group $G$:
\begin{itemize}
    \item $G$ is {\em nilpotent} if its lower central series is finite, that is to say there is $n\in\N$ such that
    $$G = G_0\triangleright G_1 \triangleright \dots \triangleright G_n = \{e_G\},$$
    where $G_{i+1}:=[G_i,G]$ is the commutator group of $G_i$ and $G$, i.e., 
    the subgroup of $G$ generated by the
    elements of the form $hgh^{-1}g^{-1}$, where $h\in G_i,g\in G$. 
    \item $G$ is {\em finitely generated} if there exist $g_1,\dots,g_n\in G$
    such that any element of $G$ can be written as product of $g_1,\dots,g_n$.
    \item $G$ is {\em torsion-free} if it does not have any non-trivial
    element of finite order, that is to say, for any $g\in G$ with $g\neq e_G$ 
    and any $n\in\N$, we have $g^n\neq e_G$.
    \item If $P$ is a property of groups, then we say that a group $G$ is
    {\em virtually} $P$ if it has a finite-index subgroup that has the
    property $P$.
\end{itemize}

In any finitely generated nilpotent group $G$, there exist some $s\in\N$
(depending on the degree of nilpotency and the number of generators of $G$),
some $a_i\in G$ and some functions $t_i:G\to\Z$, for $1\leq i\leq s$, 
such that any $x\in G$ can be written as $x = a_1^{t_1(x)}\cdots a_s^{t_s(x)}$.
The $s$-tuple $(a_1,\dots,a_s)$ is a {\em Mal'cev basis} 
and the $s$-tuple $(t_1,\dots,t_s)$ is a 
{\em Mal'cev coordinate system} with respect to this Mal'cev basis.
If $G$ is also torsion-free, then the coordinate maps are injective and hence
we can identify $G$ with $\Z^s$ and
it is convenient to also identify any $x\in G$ with its coordinates 
$(t_1(x),\dots,t_s(x))\in\Z^s$.
The above facts about Mal'cev bases
can be found in \cite[Chapter 17.2]{Ka-Me-G}.

\begin{theorem}\label{tf fg nil are square absolutely continuous}
Every torsion-free finitely generated nilpotent group is 
square absolutely continuous.
\end{theorem}

Combining Theorems \ref{mt1} and \ref{tf fg nil are square absolutely continuous}
we have that every torsion-free finitely generated nilpotent group
satisfies the conclusion of \cref{mt1}. In fact, we prove the following slight strengthening:

\begin{corollary}\label{tf fg nil cor}
Let $G$ be a torsion-free finitely generated nilpotent group 
and $A\subset G$ with positive left upper Banach density. 
Then there exist an 
infinite sequence $B=(b_n)_{n\in\N}\subset A$ 
and some $t\in G$ such that
$$B \ltrdot B \subset t^{-1}A.$$
Moreover, given a Mal'cev coordinate system $(t_1,\dots, t_s)$
on $G$, we can choose $B$ so that the following holds:
for any finite set 
$C\subset \Z$ and any $1\leq i\leq s$, the set 
$\{b\in B \colon t_i(b) \in C\}$ is finite. 
\end{corollary}

Furthermore, we are able to extend the first statement of \cref{tf fg nil cor}
to all finitely generated virtually nilpotent groups.

\begin{corollary}\label{fg nil cor}
Let $G$ be a finitely generated virtually nilpotent group 
and $A\subset G$ with positive left upper Banach density. 
Then there exist $g\in G$, an infinite sequence 
$B=(b_n)_{n\in\N}\subset g^{-1}A$,
and some $t\in G$ such that
$$B \ltrdot B \subset t^{-1}A.$$
In particular, this holds for the group $UT(n,F)$\footnote{
$UT(n,F)$ is the unitriangular $n\times n$ matrix group
with entries from $F$.
Note that $U(3,\Z)$ is the well-known $3 \times 3$ Heisenberg group.},
where $n\in\N$ and $F =(F, +, \cdot)$ is any infinite commutative unital ring with the property that the 
additive group $(F, +)$ is finitely generated.
\end{corollary}

\begin{remark}
Note that the class of finitely generated virtually 
nilpotent groups coincides, in view of Gromov's theorem \cite{Gromov}, with
the class of finitely generated groups of polynomial growth.
\end{remark}

The final part of this subsection is concerned with sumsets in abelian groups.
Let $G=(G,+)$ be an abelian group. 
We write $2g$ to denote the element $g+g$, for any $g\in G$.
Moreover, we refer to the map $s_G:G\to G$ as the {\em doubling map}
and its image is now the subgroup of $G$ consisting of all elements
of the form $2g$, where $g\in G$. We denote this subgroup by $2G$,
i.e., $s_G(G)=2G$, and we often refer to it as the {\em doubling subgroup}.

In \cite[Conjecture 5.14]{kmrr_survey}, it is conjectured that in 
any countable abelian group $G$, every set of positive upper Banach density 
contains a set of the form $B\oplus B+t=\{b_1+b_2+t: b_1\neq b_2\in B\}$ 
for an infinite set $B\subset G$ and some $t\in G$. 
It follows from \cref{fg nil cor} that this conjecture holds
under the additional assumption that $G$ is finitely generated,
and moreover, it extends \cite[Theorem 1.2]{kmrr2} from $(\N,+)$
to all finitely generated abelian groups.
In fact, we extend \cite[Theorem 1.2]{kmrr2} to an even larger
collection of abelian groups, that contains all the finitely 
generated abelian groups along with some infinitely generated ones.
The following theorem allows us to do so:

\begin{theorem}\label{some abelian groups are sac}
Every abelian group whose doubling subgroup has finite index is 
square absolutely continuous.
\end{theorem}

The following corollary is an obvious consequence of Theorems \ref{mt1}
and \ref{some abelian groups are sac}:

\begin{corollary}\label{abelian groups with finite-index doubles}
Let $(G,+)$ be an abelian group such that $2G$ is a finite-index subgroup of $G$,
and let $A\subset G$ with positive upper Banach density.
Then there exist an infinite set $B\subset A$ and some $t\in G$
such that $$B \oplus B \subset A-t.$$
In particular, this holds for:
\begin{itemize}
    \item all finitely generated abelian groups,\footnote{
    It is not hard to check that in finitely generated abelian groups, 
    the doubling subgroup has finite-index.}
    and
    \item $(\mathbb{F}_p^\omega,+)$,\footnote{$\mathbb{F}_p^\omega$ is the direct product
    of infinitely many copies of $\mathbb{F}_p=\Z/p\Z$, and it is clearly infinitely
    generated abelian.}
    where $p$ is any odd prime.
\end{itemize}
\end{corollary}

\cref{abelian groups with finite-index doubles} is in fact optimal, in the sense
that $2G$ being a finite-index subgroup of $G$ is a necessary assumption. 
As shown in a recent paper of Ethan Ackelsberg \cite{ackelsberg2024counterexamples},
if $2G$ has infinite index in $G$, then
one can always find a set $A$ with upper Banach density arbitrarily close to $1$
which does not
contain any shifted sumset $t+ B \oplus B$ of some infinite set $B$.
Therefore, \cref{abelian groups with finite-index doubles} along with the work in \cite{ackelsberg2024counterexamples} fully resolve \cite[Conjecture 5.14]{kmrr_survey}.

We remark that \cref{abelian groups with finite-index doubles} can also be proved
independently of \cref{mt1}, meaning that, by slightly modifying the proof 
of \cref{mt1}, one can directly obtain the result 
for abelian groups with finite-index doubling subgroup 
without showing that such groups are square absolutely continuous.

\subsection{More product sets and open questions}
The following remark shows that in the formulation of \cref{mt1}
one can replace left shifts with right shifts and the statement remains true.

\begin{remark}\label{right shift}
Let $G$ and $A$ be as in \cref{mt1}. Then there exist some $t\in G$ 
and some $B=(b_n)_{n\in\N}\subset tAt^{-1}$ such that
$$B\ltrdot B\subset At^{-1}.$$
To see why, let $B'\subset A$ and $t\in G$ such that $B'\ltrdot B'\subset t^{-1}A$,
as guaranteed by \cref{mt1}, and then let $B = tB't^{-1}$.
\end{remark}

Aside from replacing left shifts with right shifts, it is also natural to
ask whether one can replace product sets of the form $B\ltrdot B$ with 
those of the form $B\rtrdot B$, and additionally when the restriction 
$B\subset A$ can be imposed. The following table addresses this question 
in the case when $G$ is a square absolutely continuous group. 

\begin{tabularx}{\textwidth}{|X|X|}
\caption{Product sets in sets of positive left upper Banach density}
\label{table}\\
 \hline
 $B\ltrdot B\subset t^{-1}A$, for $B\subset A$ & True (\cref{mt1};
 for $\overline{d}_\Phi(A)>0$ for some left \Folner{} $\Phi$) \\
 \hline
 $B\ltrdot B\subset At^{-1}$, for $B\subset G$ & True (\cref{right shift};
 for $\overline{d}_\Phi(A)>0$ for some left \Folner{} $\Phi$) \\
 \hline
 $B\ltrdot B\subset At^{-1}$, for $B\subset A$ & False (\cref{counterexample 1};
 with $\overline{d}_\Phi(A)>0$ for some left \Folner{} $\Phi$) \\
 \hline
 $B\rtrdot B\subset t^{-1}A$, for $B\subset G$ & False (\cref{counterexample 2};
 with $\overline{d}_\Phi(A)=1$ for some left \Folner{} $\Phi$) \\
 \hline
 $B\rtrdot B\subset At^{-1}$, for $B\subset G$ & False (\cref{counterexample 3};
 with $\overline{d}_\Phi(A)=1$ for some left \Folner{} $\Phi$) \\
 \hline
\end{tabularx}

The above table shows that \cref{mt1} is optimal for sets of positive 
left upper Banach density in non-commutative groups, in the sense that
it is not necessarily true that one can find a product set
of the form $B\odot B$ (or even $B\rtrdot B$) inside shifts of such sets.
In addition, we remark that the table above provides a partial answer to 
\cite[Question 5.19]{kmrr_survey}.

It remains interesting to ask whether product sets of the form $B\odot B$
can be found in sets with positive upper Banach density. Unfortunately, 
our methods here are insufficient to handle this case. 
In this spirit, we conclude this section with the two questions below. 
We remark that the second one is a special case of 
\cite[Question 5.17]{kmrr_survey}.

\begin{question}\label{Q1}
Let $G$ be a square absolutely continuous group
and $A\subset G$ be a set of positive upper Banach density. 
Is it true that there exists some infinite set $B\subset G$
such that
$$B\odot B\subset t^{-1}A\cup Ar^{-1}$$
for some $t,r\in G$?
\end{question}

\begin{question}\label{Q2}
Let $G$ and $A$ be as in \cref{Q1}. 
Is it true that there exists some infinite set $B\subset G$
such that 
$$B\odot B\subset t^{-1}Ar^{-1}$$
for some $t,r\in G$?
\end{question}

\subsection{Proof ideas}
To prove \cref{mt1}, we follow an ergodic-theoretic approach and 
we employ ideas similar to the ones used in \cite{kmrr2} in the setting 
of $(\N,+)$.
This approach is based on methods that were introduced in \cite{kmrr1}
to generalize another sumset conjecture of \Erdos{}, which was initially
proved in \cite{mrr19} by Moreira, Richter and Robertson.
However, the generality of the setting of amenable groups 
compared to $(\N,+)$ causes several issues and complications that we need 
to handle differently. These issues, along with the new ideas we develop 
to deal with them, are briefly discussed below.
\par
After translating \cref{mt1} into a dynamical statement (see \cref{dr3}),
we reduce the problem to finding certain dynamical configurations
given by limit points of orbits of ergodic measure-preserving $G$-actions,
called \textit{\Erdos{} progressions} (see \cref{Erdos progressions}).
The natural environment in which one can study such 
progressions is the Kronecker factor of a system, as \Erdos{}
progressions are simply $3$-term arithmetic progressions there. 

One of the main obstructions we had to overcome in our proof is 
the lack of commutativity of $G$. The most notable among the 
issues that this leads to is that the Kronecker factor  
does not have the structure of an abelian group, but instead, it 
is a homogeneous space $Z=K/H$, for some compact group $K$. This
makes the study of \Erdos{} progressions more technically challenging.
To be more precise, the abelian nature of the Kronecker factor in the setting of 
$(\N,+)$-actions is heavily used in \cite{kmrr2}. Consequently, 
due to the absence of commutativity in our case,
many of the techniques in \cite{kmrr2} do not generalize easily to our setting.
Another difficulty that arises in non-commutative groups is the erratic behavior of 
the set of squares $G^2$.
In particular, orbits of points along $G^2$ may be trapped in zero-measure regions,
which causes serious trouble in finding \Erdos{} progressions.
The assumption that $G$ is square absolutely continuous is critical 
in avoiding this scenario. In addition, we need an extension of a result of 
Host and Kra (\cite[Proposition 6.1]{HK09}) concerning actions of $(\N,+)$,
to the more general setting of amenable group actions
(see \cref{HK_generic}, proof in \cref{appendix.a}).

\vspace{2mm} 
\noindent
\textbf{Acknowledgements.}
We would like to thank Joel Moreira and Florian K. Richter for their
insightful suggestions, beneficial comments and constant support throughout 
the writing of this paper. We also want to thank Ethan Ackelsberg,
Felipe Hern\'andez Castro and the anonymous referee for their useful comments.

The first author gratefully acknowledges support from the 
Swiss National Science Foundation grant TMSGI2-211214.
The second author was supported by the Warwick Mathematics 
Institute Centre for Doctoral Training, and gratefully 
acknowledges funding by University of Warwick’s 
Chancellors' International Scholarship scheme.

\section{Preliminaries}\label{prelims section}
In this section, we state all the preliminaries that will be useful 
in the rest of the paper regarding classic notions and theorems of 
ergodic theory of actions of amenable groups. So, for the 
rest of the section, $G$ denotes an arbitrary countable and discrete 
amenable group. \\
\noindent
\textbf{Basics on $G$-systems:}
Given a compact metric space $X=(X,d_X)$, 
a \textit{continuous action} $T=(T_g)_{g\in G}$ of $G$ on $X$
is a collection of continuous functions $T_g:X\to X$ such that 
for any $g_1, g_2 \in G$, $T_{g_1} \circ T_{g_2} = T_{g_1g_2}$.
Given such an action, we call the pair 
$(X,T)$ a \textit{topological $G$-system.}
\par
Given a topological $G$-system $(X,T)$ and a point $x\in X$, 
we define its \textit{orbit} as $\Orb_T(x)=\{T_gx:g\in G\},$ and 
we say that the point is \textit{transitive} if $\Orb_T(x)$ is dense 
in $X$.
\par
Fix a topological $G$-system $(X,T)$. Let $M(X)$ denote the space 
of Borel probability measures on $X$, equipped with the weak$^\ast$ 
topology, which is compact and metrizable. A measure $\mu\in M(X)$ is 
said to be $T$-invariant, 
if it is invariant under $T_g$ for all $g\in G$. 
Amenability of $G$ implies 
that there are $T$-invariant measures in $M(X)$.
The subset of $M(X)$
consisting of $T$-invariant measures is denoted by $M^T(X)$,
and it is a non-empty closed and convex subset of $M(X)$.
The Borel $\sigma$-algebra on $X$ is denoted by $\B_X$ or just $\B$, 
if no confusion may arise.
\par
For $\mu\in M^T(X)$, the action $T$ on the Borel probability space 
$(X,\mu)$ is called a \textit{measure-preserving $G$-action} and 
$\xmt$ is called a \textit{measure-preserving $G$-system}.
Note that we omit writing the symbol for the $\sigma$-algebra, 
and from now on, whenever this happens, the implied $\sigma$-algebra will be the Borel $\sigma$-algebra.
For simplicity, we refer to the above as $G$-actions, and $G$-systems, respectively.
Recall that all $G$-actions considered throughout are continuous.
\par
Given a $G$-system $\xmt$, one can define an action, which by abuse of notation
will again be denoted by $T=(T_g)_{g\in G}$, of $G$ on $L^2(X)$ by
$T_g: L^2(X) \to L^2(X)$, $T_g f = f \circ T_g$. It is not hard see
that for all $g\in G$, $T_g$ is an isometry of $L^2(X)$. Note also that since
$G$ acts from the left on $X$, then $G$ acts from the right on $L^2(X)$.
\par
We remark that we are only considering $G$-systems where $G$ acts on
the prescribed space from the left, and then any associated \Folner{} sequence
will be considered left, without mentioning it, unless it is necessary. 

Note that we could define $G$-systems more 
generally as follows: a $G$-system is a quadruple $(X,\mathcal{A}, \mu,T)$, where $X$ is any set, $\mathcal{A}$ is a 
$\sigma$-algebra on $X$, $\mu$ is a probability measure 
on $(X,\mathcal{A})$ and $T$ is a left action of $G$ on $X$ 
which is measurable and preserves $\mu$. We chose to
not define systems in that generality, as for our purposes 
we will always work with the more specific $G$-systems defined 
above. The only occasion where we need this 
more general definition of $G$-systems is
when we define the Kronecker factor 
right after \cref{JdLG}, in which case $\mathcal{A}$ is a 
sub-$\sigma$-algebra of the Borel $\sigma$-algebra.\\ 
\noindent
\textbf{Product $G$-system:} Given two $G$-systems $\xmt$ 
and $\yns$, we define the product $G$-system 
$(X\times Y, \mu\times\nu, T\times S)$,
where the underlying $\sigma$-algebra is the product of the Borel $\sigma$-algebras 
on $X$ and $Y$, which coincides with the Borel $\sigma$-algebra on $X\times Y$, 
and the action is $T\times S = (T_g\times S_g)_{g\in G}$. \\
\noindent
\textbf{Factors of $G$-systems:}
Given two $G$-systems $\xmt$ and $\yns$, we say that $\yns$ is a
\textit{factor} of $\xmt$ if there exists a measurable map 
$\pi:X\to Y$, which we call a \textit{factor map}, satisfying: 
$\mu(\pi^{-1}E)=\nu(E)$ for any measurable $E\subset Y$ and for any $g\in G$,
$\pi\circ T_g = S_g\circ\pi$ $\mu$-almost everywhere on $X$. 
When the former is true, we say that $\nu$ is the push-forward of $\mu$
under $\pi$, and we write $\pi\mu=\nu$. When, additionally, the 
factor map $\pi$ is continuous and $\pi\circ T_g=S_g\circ\pi$ holds 
everywhere on $X$ for any $g\in G$, we say that $\pi$ is a \textit{continuous
factor map} and $\yns$ is a \textit{continuous factor} of $\xmt$. \\
\noindent
\textbf{Ergodicity and ergodic theorems for $G$-systems:}
A $G$-system $\xmt$ is called \textit{ergodic} 
if for any measurable set $A$
the following holds:
$$ T_g ^{-1} A =A \text{ for all } g\in G \implies \mu(A)=0 \text{ or }
\mu(A)=1.$$

Given a $G$-system $\xmt$, let $\A$ be a sub-$\sigma$-algebra of $\B$. 
For $f\in L^2(X,\mu)$, the \textit{conditional expectation of 
$f$ on $\A$}, denoted by $\E_\mu(f\:|\:\A)$, is defined as 
the orthogonal projection of $f$ on the closed subspace 
$L^2(X,\A,\mu)$ of $L^2(X,\mu)$. 
We also denote the
sub-$\sigma$-algebra of the $T$-invariant sets by $\I=\I(T)$; that is,
$$\I=\I(T) := \{ E\in \B: T_g ^{-1} E=E \text{ for all } g\in G\}.$$

\begin{theorem}[Mean Ergodic Theorem for $G$-systems, see {\cite[Theorem 3.33]{glasner}}]
\label{met}
Let $\xmt$ be a $G$-system, and let $\Phi$ be a \Folner{} sequence. 
Then, for any $f\in L^2(X)$,
$$ \frac{1}{|\Phi_N|} \sum_{g\in \Phi_N} T_g f \to \E_\mu(f\:|\:\I)$$
as $N\to \infty$ in $L^2(X)$. In addition, if the system is ergodic, 
the ergodic averages above converge to $\int_X f\d\mu$.
\end{theorem}
\noindent
\textbf{Measure disintegration and ergodic decomposition:}
When $X$ is a compact metric space, the space 
$M(X)$ of Borel probability measures on 
$X$ can be endowed with a $\sigma$-algebra $\M$ 
such that the space 
$(M(X),\M)$ is a standard Borel space.
The following theorem about disintegrations of 
measures is very useful.
\begin{theorem}[Disintegration of measures, 
see {\cite[Chapter 2, Section 2.5]{Host-Kra}}]
\label{dis-meas}
Let $X$ be a compact metric space, $\B$ the 
Borel $\sigma$-algebra on $X$ and $\mu$ a probability measure
on $(X, \B)$.
Let also $\D$ be a 
sub-$\sigma$-algebra of $\B$. 
Then there is a $(\D,\mu)$-almost everywhere 
defined and measurable map 
$(X,\D) \to (M(X),\M)$, $x\mapsto \mu_x$ with the 
following properties:
\begin{itemize}
    \item For every $f\in L^1(X,\mu)$, the function 
    $x \mapsto \int_X f \d \mu_x$ is in 
    $L^1 (X,\D,\mu)$, and for all 
    $D\in \D$, we have $\int_D f \d \mu=\int_D 
    \big(\int_X f \d \mu_x\big) \d \mu(x).$ 
    In particular, this implies that 
    $\int_X f \d \mu_x=\E_{\mu} (f\:|\:\D)(x)$ for 
    $(\D,\mu)$-almost every $x\in X$.
    \item For $(\D,\mu)$-almost every $x\in X$, $\mu_x([x]_{\D})=1$, where 
    $[x]_{\D}=\cap_{x\in D\in \D} D.$
\end{itemize}
The map satisfying the above properties is unique modulo
$(\D,\mu)$-null sets, and is called the 
\textbf{disintegration}
of the measure $\mu$ over the sub-$\sigma$-algebra $\D$. In that case, we 
write
$\mu=\int_X \mu_x \d \mu(x)$.
\end{theorem}

In this paper, we will extensively make use of
\textit{disintegrations over (continuous) factor maps}. 
Let $\pi:\xmt\to\yns$ be a factor map between two $G$-systems.
Given a function $f \in L^1(X, \mu)$, the conditional
expectation $\E_\mu(f\:|\:Y)$ of $f$ with respect to the factor $Y$
is the function in $L^1(Y, \nu)$ defined by 
\begin{equation*}
    \int_B \E_\mu(f\:|\:Y)(y) \d \nu (y) = 
    \int_{\pi ^{-1}(B)} f(x) \d \mu (x)
\end{equation*}
for every Borel measurable set $B \subset Y$.
Then \cref{dis-meas} gives a disintegration $y\mapsto \mu_y$
defined on $Y$, which is unique up to
$\nu$-null measure sets, 
and satisfies the following:
\begin{itemize}
    \item for every $f \in L^1(X, \mu)$,
    for $\nu$-almost every $y\in Y$,
    \begin{equation}\label{eq_disint}
    \E_\mu(f\:|\:Y)(y)
    =\int_X f\d\mu_y,
    \end{equation}
    \item for $\nu$-almost every $y\in Y$, 
    $\mu_y(\pi^{-1}(\{y\}))=1$, and finally, 
    \item for any $g\in G$ and for $\nu$-almost every $y\in Y$,
    $(T_g)\mu_y=\mu_{S_gy}$.
\end{itemize}

Let $\xmt$ be a $G$-system. Consider the (unique) 
disintegration of $\mu$
with respect to $\I=\I(T)$ given by \cref{dis-meas}. 
This disintegration is 
called the \textit{ergodic decomposition of $\mu$}. 
Equivalently, we say that the 
disintegration $x\mapsto\mu_x$ is the ergodic decomposition of $\mu$, if 
for any $f:X\to\C$ measurable and bounded,
\begin{equation}\label{eq.1}
\int_X f\d\mu_x=\E_\mu(f\:|\:\I)(x)
\end{equation}
holds for $(\I,\mu)$-almost every $x\in X$.
\begin{theorem}
[Ergodic decomposition of $G$-systems, see {\cite[Theorem 3.22]{glasner}}]\label{erg-dec}
If $\xmt$ is a $G$-system as above, then for 
$(\I, \mu)$-almost every $x \in X$, the measure 
$\mu_x$ is $T$-invariant and the system $(X,\mu_x, T)$ 
is ergodic.
\end{theorem}

\noindent
\textbf{Generic points and the support of a measure:}
In addition, we will need the notion of generic points:
\begin{defn}
Let $\xmt$ be a $G$-system and let $\Phi$ be a \Folner{} 
sequence. A point $a\in X$ is called 
\textit{generic for $\mu$ along $\Phi$} 
if for all $f\in C(X)$ we have 
$$ \lim_{N\to \infty} \frac{1}{|\Phi_N|} \sum_{g\in \Phi_N} f(T_g a)
= \int_{X} f \d \mu$$
or equivalently if 
$$ \lim_{N\to \infty} \frac{1}{| \Phi_N|} \sum_{g\in \Phi_N}
\delta_{T_g a}=\mu $$
where $\delta_x$ is the Dirac mass at $x\in X$ and the limit 
is in the weak$^\ast$ topology. If $a$ is generic 
for $\mu$ along $\Phi$, then we denote this by 
$a\in \gen(\mu,\Phi)$.
\end{defn}

Moreover, we will need the notion of the support of a measure. 
The \textit{support} of a Borel probability measure $\mu$ on a 
compact metric space $X$ is the smallest closed full-measure 
subset of $X$ and is denoted by $\supp(\mu)$.
We will need the following lemma, which says that generic 
points for a measure have dense orbit in the support of the measure. 
Its proof is quite standard, and we only include it for completeness.

\begin{lemma}\label{supp-gen lemma}
Let $\yns$ be a $G$-system and let $y,w\in Y$. If $y\in \gen(\nu,\Phi)$
for some \Folner{} sequence $\Phi$, and $w\in\supp(\nu)$,
then $S_{g_n}y\to w$, for some infinite sequence $(g_n)_{n\in\N}$ in $G$.
\end{lemma}

\begin{proof}
Fix a compatible metric on $X$ and let $\Ball(w,\epsilon)$
be the open ball centered at $w$ with radius $\epsilon>0$ with 
respect to this metric. 
By Urysohn's lemma, for every $\epsilon>0$ there exists a continuous 
function $f:X\to[0,1]$ with $f=1$ on $\Ball(w,\epsilon/2)$
and $f=0$ outside $\Ball(w,\epsilon)$.
Since $w\in\supp(\nu)$, it follows that $\int_Y f\d\nu>0$.
Now, using that $y\in\gen(\nu,\Phi)$, we have that
$$\lim_{N\to\infty}\frac{1}{|\Phi_N|}\sum_{g\in\Phi_N} 
f(S_g y)
= \int_X f\d\nu > 0,$$
which implies that $S_g y\in\Ball(w,\epsilon)$ for infinitely many 
$g\in G$. The result then follows.
\end{proof}

We will also make use of the following result of Lindenstrauss:

\begin{proposition}[see {\cite[Theorem 1.2 and Proposition 1.4]{lindenstrauss_pet}}]\label{Lind}
Let $\xmt$ be a $G$-system and $\Phi$ be a \Folner{} sequence in $G$. Then there
is a subsequence $\Psi$ of $\Phi$ such that for all $f\in L^1(\mu)$, 
$$\lim_{N\to \infty} \frac{1}{|\Psi_N|} \sum_{g\in \Psi_N}
T_g f (x) = \E_{\mu}(f\: | \: \I) (x)$$
for $\mu$-almost every $x\in X$.
\end{proposition}

The next two lemmas follow easily from \cref{Lind} and \eqref{eq.1}.

\begin{lemma}\label{m_gen}
Let $\xmt$ be an ergodic $G$-system. 
Then for any \Folner{} sequence $\Phi$
there exists some subsequence $\Psi$ such 
that $\mu$-almost every 
$x\in X$ is in $\gen(\mu,\Psi)$.
\end{lemma}

\begin{lemma}\label{mx_gen}
Let $\xmt$ be a $G$-system, let $\Phi$ be a \Folner{} sequence, and 
let 
$x\mapsto\mu_x$ be the ergodic decomposition of $\mu$. 
Then there exists some subsequence $\Psi$ of $\Phi$ such
that $\mu$-almost every $x\in X$ is in $\gen(\mu_x,\Psi)$.
\end{lemma}

Finally, it will be useful to have the following generalization of 
\cite[Proposition 3.9]{Fur1}, whose proof is again the same as for actions 
of $(\N,+)$, but we include it for the convenience of the reader. 

\begin{lemma}\label{quasi_gen}
Let $G$ be an amenable group, let $\xmt$ be an ergodic $G$-system, and let $a\in X$
be a point such that $\mu$ is supported on $\overline{\Orb_T(a)}$.
Then there exists some \Folner{} 
sequence $\Psi$ such that $a\in\gen(\mu,\Psi)$.
\end{lemma}
\begin{proof}
By \cref{m_gen}, there exists some $x_0\in\overline{\Orb_T(a)}$ that is 
generic 
for $\mu$ along some \Folner{} sequence $\Phi$. Let $\F=
(f_k)_{k\in\N}$ 
be a dense subset of $(C(X),\|\cdot\|_\infty)$ and let
$(\Phi_{N_n})_{n\in \N}$ 
be a subsequence of $\Phi$ such that for every $n\in\N$ 
and 
for every $j=1,2,\dots,n$,
$$\bigg|\frac{1}{|\Phi_{N_n}|}\sum_{g\in\Phi_{N_n}}f_j(T_gx_0)
-\int_X f_j\d\mu\bigg|<\frac{1}{n}.$$
Since $x_0\in\overline{\Orb_T(a)}$, there exists some 
$(g_n)_{n\in\N}\subset G$ such that 
$T_{g_n}a\to x_0$, so that we may assume that the equation above holds 
if we substitute $x_0$ with $T_{g_n}a$. Consider the \Folner{} 
sequence $\Psi=(\Psi_n)$ given by $\Psi_n=\Phi_{N_n}g_n$. Note that this is still a left \Folner{} sequence.
It follows that for every $n\in\N$ and any $j=1,2,\dots,n$,
$$\bigg|\frac{1}{|\Psi_n|}\sum_{g\in\Psi_n}f_j(T_ga)
-\int_X f_j\d\mu\bigg| < \frac{1}{n}.$$
Since $\F$ is dense in $C(X)$ the conclusion follows as before.
\end{proof}
\noindent
\textbf{Kronecker factor and the Jacobs-de Leeuw-Glicksberg decomposition:}
Let $\xmt$ be a $G$-system. A function $f\in L^2(X)$ is called:
\begin{itemize}
    \item \textit{compact},
    if $\overline{\{T_gf:g\in G\}}$ is compact with respect to the strong 
    topology on $L^2(X)$. 
    \item \textit{weak-mixing}, if for any \Folner{} sequence $\Phi$, and 
    any $f'\in L^2(X)$,
    $$\lim_{N\to\infty} \frac{1}{|\Phi_N|}\sum_{g\in\Phi_N}
    |\langle T_gf,f'\rangle| = 0.$$
\end{itemize}
We define the \textit{compact component} of $L^2(X)$ as 
$\Hilb_{\textup{c}}(T)=\overline{\textup{span}\{f\in L^2(X): f \text{ is compact}\}}$,
and the \textit{weak-mixing component} of $L^2(X)$ as 
$\Hilb_\textup{wm}(T)=\{f\in L^2(X): f \text{ is weak-mixing}\}$. When no confusion 
may arise, we simply write $\Hilb_{\textup{c}}$ and $\Hilb_\textup{wm}$ respectively.
\par
In case that $G$ is an amenable group, 
the Jacobs-de Leeuw-Glicksberg decomposition theorem applies, stating
that these two components give a decomposition of $L^2(X)$.

\begin{theorem}[Jacobs-de Leeuw-Glicksberg decomposition,
see {\cite[Theorem 2.24]{DKE}}]\label{JdLG}
If $\xmt$ is a $G$-system, then 
$$L^2(X)=\Hilb_{\textup{c}}\oplus\Hilb_{\textup{wm}}.$$
\end{theorem}

Now we will give a description of the factor of $\xmt$ 
corresponding to the subspace $\Hilb_{\textup{c}}$ of $L^2(X)$.
Note that if $f\in \Hilb_{\textup{c}}$, 
then for all $g \in G$, $T_g f\in 
\Hilb_{\textup{c}}$, so $\Hilb_{\textup{c}}
$ in invariant under the action of $T$ on 
$L^2(X)$. Let $\mathcal{A}$ be the smallest $\sigma$-algebra 
with respect to which all functions in $\Hilb_{\textup{c}}$ are 
measurable. Then $\mathcal{A}$ is a
$T$-invariant $\sigma$-algebra contained
in the Borel $\sigma$-algebra of $X$. Therefore, the system 
$(X,\mathcal{A},\mu,T)$ is a factor of the original system, with 
the factor map being the identity $\textup{id}: X \to X$. This factor 
is called the \textit{Kronecker factor} of $\xmt$.
\par
Our goal now is to give a nice algebraic description of the Kronecker 
factor when
the $G$-system $\xmt$ is ergodic, but first we need the following definition. 
\begin{defn}
Let $K$ be a compact group, $H$ be a closed subgroup of $K$, and $\alpha\colon G\to K$ be a group homomorphism. Consider the homogeneous space $Z=K/H$. Let also $m$ be the normalized Haar measure on $Z$, and $R=(R_g)_{g\in G}$, where for each $g\in G$, $R_g\colon Z\to Z$ is given by $R_g(z) = \alpha(g)z$. Then the $G$-system $\zmr$ is called {\em a rotation on the homogeneous space $Z$ by $\alpha$}.
\end{defn}

\begin{proposition}\cite[Theorem 1]{mackey}
\label{Kron_rot}
Let $\xmt$ be an ergodic $G$-system. Then its Kronecker factor is measurably isomorphic to a rotation on some homogeneous space $Z$ by some $\alpha$ with dense image.
\end{proposition}

From \cref{Kron_rot} we get that if $\mathcal{C}$ 
is the Borel $\sigma$-algebra on  
$Z$, then $\pi^{-1}(\mathcal{C})$ is equivalent to $\mathcal{A}$, 
i.e., they are 
equal modulo sets that have zero $\mu$ measure.
\cref{Kron_rot} allows us to identify the Kronecker factor with
a rotation on a homogeneous space, whenever $\xmt$ is ergodic.
\par
\noindent
\textbf{Characteristic factors for $G$-systems:} 
The notion of \textit{characteristic factors} 
will play a fundamental role later in one of our proofs.
Here we have the following theorem for the characteristic factors 
with respect to some double averages that will concern us.

\begin{theorem}\label{char}
Let $\xmt$ be an 
ergodic $G$-system, let $\zmr$ be its Kronecker factor and $\Phi$ be a 
\Folner{} sequence. Then for any $f_1,f_2 \in L^{\infty}(X)$, we have 
\begin{equation}\label{char_factor_eq}
\lim_{N\to \infty}\frac{1}{|\Phi_N|}\sum_{g\in \Phi_N}T_g f_1\otimes T_g 
f_2 
=
\lim_{N\to \infty}\frac{1}{|\Phi_N|}\sum_{g\in \Phi_N}T_g
\E_\mu(f_1\:|\:Z)\otimes T_g \E_\mu(f_2\:|\:Z)
\end{equation}
in $L^2(X\times X,\mu\times \mu)$.
\end{theorem}

\cref{char} says that the Kronecker factor is the characteristic factor for the averages in the left-hand side of \eqref{char_factor_eq}.
The proof of \cref{char} will follow easily from the next lemma.
\begin{lemma}\label{wmlemma}
Let $\xmt$ be a $G$-system. Then $$\Hilb_{\textup{wm}}(T)\otimes L^2(X)
\subset \Hilb_{\textup{wm}}(T\times T) \:\text{ and }\: 
L^2(X) \otimes \Hilb_{\textup{wm}}(T)
\subset \Hilb_{\textup{wm}}(T\times T).$$
\end{lemma}
\begin{proof}
We will only prove the first inclusion, as the second follows in 
an analogous way.
Let $\Phi$ be any \Folner{} sequence, and let 
$f_1\in\Hilb_{\textup{wm}}(T)$
and $f_2\in L^2(X)$. We may assume that $\|f_1\|_2,\|f_2\|_2\leq1$.
We want to show that
\begin{equation}\label{eq_ld}
\lim_{N\to\infty}\frac{1}{|\Phi_N|}\sum_{g\in\Phi_N}
\Big|\langle(T_g\times T_g)(f_1\otimes f_2),F\rangle_{L^2(\mu\times\mu)}\Big|^2 
= 0    
\end{equation}
for any $F\in L^2(X\times X,\mu\times\mu)$. 
\par
Let $F\in L^2(X\times X,\mu\times\mu)$ and $\epsilon>0$. We may assume that
$\|F\|_{L^2(\mu\times\mu)}=1$.
Now, since finite linear combinations of functions of the form
$f'_1\otimes f'_2$, where $f'_1,f'_2\in L^2(X)$, 
form a dense subset of $L^2(X\times X,\mu\times\mu)$, we can find 
$F' = \sum_{i=1}^k c_i (f_{1,i}'\otimes f_{2,i}')$ with 
$\|F'\|_{L^2(\mu\times\mu)}\leq1$ such that $\|F-F'\|_{L^2(\mu\times\mu)}<\epsilon/2$.
Then by the Cauchy-Schwarz inequality we have 
\begin{align*}
    & \big|\langle (T_g\times T_g)(f_1\otimes f_2), F\rangle_{L^2(\mu\times\mu)}\big|^2 
    = \langle (T_g\times T_g)(f_1\otimes f_2), F\rangle_{L^2(\mu\times\mu)} \langle F, (T_g\times T_g)(f_1\otimes f_2)\rangle_{L^2(\mu\times\mu)} \\
    & = \langle (T_g\times T_g)(f_1\otimes f_2), F'\rangle_{L^2(\mu\times\mu)}\langle F', (T_g\times T_g)(f_1\otimes f_2)\rangle_{L^2(\mu\times\mu)} \\
    & \hspace*{4.45cm} + \langle (T_g\times T_g)(f_1\otimes f_2), F'\rangle_{L^2(\mu\times\mu)}\langle F-F', (T_g\times T_g)(f_1\otimes f_2)\rangle_{L^2(\mu\times\mu)}\\
    & \hspace*{4.45cm} + \langle (T_g\times T_g)(f_1\otimes f_2), F-F'\rangle_{L^2(\mu\times\mu)} \langle F, (T_g\times T_g)(f_1\otimes f_2)\rangle_{L^2(\mu\times\mu)} \\
    & < \sum_{1\leq i,j\leq k} \overline{c_i}c_j
    \langle (T_g\times T_g)(f_1\otimes f_2), f_{1,i}'\otimes f_{2,i}'\rangle_{L^2(\mu\times\mu)} 
    \langle f_{1,j}'\otimes f_{2,j}', (T_g\times T_g)(f_1\otimes f_2)\rangle_{L^2(\mu\times\mu)}
    + \epsilon \\
    & \leq \sum_{1\leq i,j\leq k} \overline{c_i}c_j
    \|f_{2,i}'\|_2 \|f_{1,j}'\|_2\|f_{2,j}'\|_2|\langle T_gf_1,f_{1,i}'\rangle| + \epsilon.
\end{align*}
Therefore, using that $f_1$ is a weak-mixing function, we have that
\begin{multline*}
    \limsup_{N\to\infty}\frac{1}{|\Phi_N|}\sum_{g\in\Phi_N}
    \big|\langle(T_g\times T_g)(f_1\otimes f_2),F\rangle_{L^2(\mu\times\mu)}\big|^2 \\
    = \sum_{1\leq i,j\leq k} \overline{c_i}c_j
    \|f_{2,i}'\|_2 \|f_{1,j}'\|_2\|f_{2,j}'\|_2
    \limsup_{N\to\infty}
    \frac{1}{|\Phi_N|}\sum_{g\in\Phi_N} |\langle T_gf_1,f_{1,i}'\rangle|
    + \epsilon
    = \epsilon.
\end{multline*}
Since $\epsilon>0$ was arbitrary, then \eqref{eq_ld} follows.
The proof is complete.
\end{proof}

\begin{proof}[Proof of \cref{char}]
Let $\Phi$ be \Folner{} sequence, and let $f_1,f_2\in L^2(X)$.
Then we write
\begin{align}\label{av.split}
\frac{1}{|\Phi_N|}\sum_{g\in \Phi_N}T_g f_1\otimes T_g f_2
&=
\frac{1}{|\Phi_N|}\sum_{g\in \Phi_N}
T_g (f_1 - \mathbb{E}_{\mu}(f_1\:|\:Z) )\otimes T_g 
(f_2 - \mathbb{E}_{\mu}(f_2\:|\:Z) ) \nonumber \\
& \hspace*{0.04cm} +
\frac{1}{|\Phi_N|}\sum_{g\in \Phi_N}
T_g(f_1 - \mathbb{E}_{\mu}(f_1\:|\:Z) )\otimes T_g
\mathbb{E}_{\mu}(f_2\:|\:Z) 
\nonumber \\
& \hspace*{0.04cm} +
\frac{1}{|\Phi_N|}\sum_{g\in \Phi_N}
T_g \mathbb{E}_{\mu}(f_1\:|\:Z) \otimes T_g 
(f_2 - \mathbb{E}_{\mu}(f_2\:|\:Z)) \nonumber \\
& \hspace*{0.04cm} + 
\frac{1}{|\Phi_N|}\sum_{g\in \Phi_N}
T_g \mathbb{E}_{\mu}(f_1| Z) \otimes T_g \mathbb{E}_{\mu}(f_2\:|\:Z).
\end{align}
Note that the limits of all the terms above exist 
by the mean ergodic theorem (\cref{met})
applied to $T\times T$. By \cref{JdLG}, the functions $f_1-\E_\mu(f_1\:|\:Z),
f_2-\E_\mu(f_2\:|\:Z)$ are both weak-mixing. Then, 
by \cref{wmlemma}, the functions
$(f_1-\E_\mu(f_1\:|\:Z))\otimes(f_2-\E_\mu(f_2\:|\:Z)),
(f_1-\E_\mu(f_1\:|\:Z))\otimes \E_\mu(f_2\:|\:Z)$ and 
$\E_\mu(f_1\:|\:Z)\otimes(f_2-\E_\mu(f_2\:|\:Z))$ are weak-mixing 
with respect to $T\times T$. Hence, the limits of the first three terms in the 
right-hand side of \eqref{av.split} are $0$ in $L^2(X\times X,\mu\times\mu)$. 
Then the theorem follows.
\end{proof}

\noindent
\textbf{A correspondence principle:}
Finally, we need the following instance of Furstenberg's correspondence principle.

\begin{theorem}[Cf. {\cite[Theorem 2.8]{berg-f-m}}]\label{cp}
Let $G$ be an amenable group,
let $A\subset G$ and assume that there exists 
a left \Folner{} sequence $\Phi$ such that 
$d_\Phi(A)=\lim_{N\to\infty}\frac{|A\cap\Phi_N|}{|\Phi_N|}$
exists. Then there exist an ergodic $G$-system $(X,\mu,T)$, 
a clopen set $E\subset X$, a \Folner{} sequence $\Psi$, 
and a point $a\in \gen(\mu,\Psi)$ such that $\mu(E)\geq d_\Phi(A)$ 
and $A=\{h \in G : T_h a \in E\}$.
\end{theorem}
\begin{proof}
Consider the compact metric space 
$X:=\{0,1\}^G=\{x=(x_g)_{g\in G} : x_g \in \{0,1\} 
\: \forall \: g \in G\}$, equipped with 
the Borel $\sigma$-algebra.
We define the continuous action $T$ on $X$ by $T_h(x_g)_{g\in G}= 
(x_{gh})_{g\in G}$, for any $h\in G$ and $(x_g)_{g\in G} \in X$.
Now, consider the point $a=(\1_A (g))_{g\in G}\in X$ 
and the clopen set $E=\{x=(x_g)_{g\in G}\in X: x_{e_G}=1\}.$
By the choice of $a$, for $h\in G$ we have that
$ h\in A$ if and only if $T_h a \in E$, 
and therefore 
$A=\{ h\in G : T_h a \in E\}.$
Consider the sequence of Borel probability measures on $X$ defined by 
$$N\mapsto\mu_N=\frac{1}{|\Phi_N|}\sum_{h\in\Phi_N}\delta_{T_h a},$$
and let $\mu'$ be a weak$^{\ast}$ limit point of 
that sequence. 
Then $\mu ' (E) = d_{\Phi}(A)$, and $\mu '$ is 
$T$-invariant, but not necessarily ergodic.
Let $x \mapsto \mu_x '$ be the ergodic decomposition of $\mu '$.
Then $\mu'=\int_x \mu'_x \d \mu'(x)$, so there exists $x_0 \in X$ 
such that for the measure $\mu=\mu_{x_0} '$, we have that 
$(X,\mu,T)$ is ergodic and $\mu(E)\geq \d_{\Phi} (A)$. 
For all $N\in \N$, $\mu_N$ gives full measure to the orbit closure of $a$, 
and hence $\mu'$ also gives full measure to the orbit closure of $a$. 
Therefore, we may assume that $\mu$ is also supported
on the orbit closure of $a$ (as this is the case with $\mu'_x$ 
for $\mu'$-almost every $x\in X$). Then it follows by \cref{quasi_gen}, 
that $a\in\gen(\mu,\Psi)$ for some \Folner{} sequence $\Psi$.
\end{proof}

\section{Reduction of Theorem~\ref{mt1} to dynamical statements}
\label{reduction section}

In this section we translate our first main theorem, namely, 
\cref{mt1}, in a dynamical language. This will allow us 
to approach the problem through ergodic theoretic techniques.

\subsection{Dynamical reformulation via correspondence principle}
Usually in ergodic theory, correspondence principles serve 
as bridges between combinatorial and dynamical statements.
In our situation, we can use the correspondence principle (\cref{cp}) 
to show that Theorem ~\ref{mt1} follows from
Theorem \ref{dr1a} below, which is more dynamical in nature.

\begin{theorem}[First dynamical reformulation of \cref{mt1}]\label{dr1a}
Let $G$ be a square absolutely continuous group and
$\xmt$ be an ergodic $G$-system.
Let $a\in \gen(\mu,\Phi)$ 
for some \Folner{} sequence $\Phi$ and $E\subset G$ 
be clopen with $\mu(E)>0$. 
Then there exist an infinite sequence
$B=(b_n)_{n\in\N}\subset \{h\in G : T_h a\in E\}$ and some $t\in G$ 
such that
$$t \cdot B \ltrdot B \subset \{h\in G: T_h a\in E\}.$$
\end{theorem}

\begin{proof}[Proof that \cref{dr1a} implies \cref{mt1}]
Let $A\subset G$ have positive left upper Banach density, 
so that there exists some \Folner{} sequence $\Phi$ such that 
$d_\Phi(A)=\lim_{N\to\infty}\frac{|A\cap\Phi_N|}{|\Phi_N|}>0$, (where 
we have passed to a
subsequence). Then consider $(X,\mu,T)$, $E$, $\Psi$ and $a$, as 
ensured
by \cref{cp}, 
satisfying $\mu(E)\geq d_\Phi(A)>0$ and $\{h\in G: T_ha\in E\}=A$. 
It follows then by \cref{dr1a} that there exist an infinite sequence
$B=(b_n)_{n\in\N}\subset A$
and some $t\in G$ such that $t\cdot B\ltrdot B\subset A$.
\end{proof}

\subsection{\Erdos{} progressions}
The conclusion of Theorem \ref{dr1a} is still 
a rather combinatorial statement, so we need to reformulate it again 
into a dynamical statement. For this to be achieved, we will use the
notion of \Erdos{} progressions, as defined in \cite{kmrr2}, 
which in $\Z$ is a dynamical variant of $3$-term arithmetic progressions. In our case, \Erdos{} progressions are a dynamical variant of progressions of the form $(z,kz,k^2z)$, $k\in K$, $z\in Z$, which are the natural generalization of $3$-term arithmetic progressions in our setting.

\begin{defn}\label{Erdos progressions}
Given a topological $G$-system $(X,T)$, a point 
$(x_0,x_1,x_2)\in X^3$ is a \textit{$3$-term \Erdos{} progression}, 
if there exists an infinite sequence $(g_n)_{n\in\N}$ in $G$ 
such that
\begin{equation}\label{erdospr}
    (T_{g_n}\times T_{g_n})(x_0,x_1)\to (x_1,x_2) 
\end{equation}
\end{defn}

We refer to $3$-term \Erdos{} progressions simply as \Erdos{} progressions.
Through the notion of \Erdos{} progressions we are able to
reformulate Theorem \ref{dr1a} as follows: 

\begin{theorem}[Second dynamical reformulation]\label{dr2}
Let $G$ be a square absolutely continuous group, and let
$\xmt$ be an ergodic $G$-system and $a\in\gen(\mu,\Phi)$ 
for some \Folner{} sequence $\Phi$. If $E\subset X$ is a 
clopen set with $\mu(E)>0$, then there exist $t\in G$, $x_1\in E$ 
and $x_2\in {T_t}^{-1}E$ such that $(a,x_1,x_2)\in X^3$ forms 
an \Erdos{} progression.
\end{theorem}

For the reduction of Theorem \ref{dr1a} to \cref{dr2}
we provide the following lemma.
\begin{lemma}\label{l1}
Let $G$ be a group. Let $(X,T)$ be a topological $G$-system, and 
let $E,F\subset X$ be open sets. Assume that there 
exists an \Erdos{} progression 
$(a,x_1,x_2)\in X^3$ with $x_1\in E$ and $x_2\in F$. Then 
there exists an
infinite sequence $B=(b_n)_{n\in\N}\subset\{g\in G: T_ga\in E\}$ 
such that $B \ltrdot B\subset\{g\in G: T_ga\in F\}$.
\end{lemma}

To see how Theorem \ref{dr1a} follows from \cref{dr2}
just take $F={T_t}^{-1}E$ in the above lemma.

\begin{proof}[Proof of \cref{l1}]
By assumption there exists 
an infinite 
sequence $(g_n)_{n\in\N}$ in $G$ such that 
$(T_{g_n}\times T_{g_n}) (a,x_1)\to(x_1,x_2)$. 
Since $T_{g_n}a\to x_1\in E$ and $E$ is open, we get that 
$T_{g_n}a\in E$ for $n$ sufficiently large, so we may assume 
without loss of generality that $(g_n)_{n\in\N}\subset\{g\in G: T_ga\in E\}$.
Therefore we will construct the sequence $B$ to be a subset of $(g_n)_{n\in\N}$.
\par
We construct the sequence $B=(b_n)_{n\in\N}$ inductively.
\begin{itemize}
    \item Since $T_{g_n}x_1\to x_2\in F$ and $F$ is open, we can pick
    $b_1\in(g_n)_{n\in\N}$ such that $T_{b_1}x_1\in F$. Then $(x_1,x_2)\in
    (T^{-1}_{b_1}F)\times F$ which is open.
    \item Since $(T_{g_n}\times T_{g_n})(a,x_1)\to (x_1,x_2)$, we pick
    $b_2\in(g_n)_{n\in\N}$ such that $(T_{b_2}\times T_{b_2})(a,x_1)\in
    (T^{-1}_{b_1}F)\times F$ and $b_2\neq b_1$ (this is possible 
    since there are infinitely many choices). 
    Then $$a\in T^{-1}_{{b_1}{b_2}}F \text{\quad and \quad} 
    x_1\in T^{-1}_{b_2}F.$$
    \item Induction step: Assume we have found $b_1, b_2, \dots,
    b_n\in(g_n)_{n\in\N}$ all distinct to each other such that
    \begin{equation}\label{ind.hyp}
        a\in\bigcap_{1\leq i<j\leq n}T^{-1}_{{b_i}{b_j}}F \text{\quad 
        and \quad} x_1\in\bigcap_{1\leq m \leq n}T^{-1}_{b_m}F.
    \end{equation}
    Since $(T_{g_n}\times T_{g_n})(a,x_1)\to(x_1,x_2)\in 
    \Big(\bigcap_{1\leq m \leq n}T^{-1}_{b_m}F\Big)\times F$ 
    and this set is open, we can pick $b_{n+1}\in(g_n)_{n\in\N}$ 
    such that $$(T_{b_{n+1}}\times T_{b_{n+1}})(a,x_1)\in 
    \bigg(\bigcap_{1\leq m \leq n}T^{-1}_{b_m}F\bigg)\times F$$ 
    and $b_{n+1}\not\in \{b_m: 1\leq m\leq n\}$ 
    (since there are infinitely many choices). 
    Combining this with the inductive hypothesis, we obtain that 
    $$a\in\bigcap_{1\leq i<j\leq n+1}T^{-1}_{{b_i}{b_j}}F
    \text{\quad and \quad} 
    x_1\in\bigcap_{1\leq m \leq n+1}T^{-1}_{b_m}F.$$
\end{itemize}
Taking $B=(b_n)_{n\in\N}$ we clearly have an infinite subset of $(g_n)_{n\in\N}$
and since
\eqref{ind.hyp} holds for any $n\in\N$ by construction, we get that 
$B\ltrdot B\subset \{g\in G: T_ga\in F\}$ as desired.
\end{proof}

\subsection{Continuous factor maps to the Kronecker factor}
On our way to show \cref{dr2}, it will be useful to have 
the extra assumption of the  $G$-system having a continuous factor map to its
Kronecker factor. The reason for that will become clear towards the 
proof of our main theorem. As we will see below, it is possible to 
make such an assumption. \par
We begin by generalizing a result of Host and Kra in 
\cite{HK09} from actions of $(\N,+)$ to actions of amenable groups.

\begin{lemma}\cite[Proposition 6.1 for group actions]{HK09}
\label{HK_generic}
Let $G$ be an amenable group,
let $\xmt$ be an ergodic $G$-system, $\zmr$ be its Kronecker factor and 
$\rho:\xmt\to\zmr$ be a factor map. If $a\in X$ is a transitive
point, then there exists a point $z\in Z$ and a \Folner{} 
sequence $\Psi$ such that
\begin{equation}\label{eq_ld1}
    \lim_{N\to\infty}\frac{1}{|\Psi_N|}
\sum_{g\in\Psi_N}f_1(T_ga)\cdot f_2(R_gz)
=\int_{X}f_1\cdot(f_2\circ\rho)\d\mu
\end{equation}
holds for any $f_1\in C(X)$ and $f_2\in C(Z)$.
\par
We remark that the result still holds if we replace $\zmr$ by any factor of 
$\xmt$ that is distal as a topological system.
\end{lemma}

\begin{proof}
The proof of this lemma is given in the Appendix~\ref{appendix.a}.
\end{proof}

\begin{remark}\label{atrans}
Let $\xmt$ be an ergodic $G$-system, and let $a\in \gen(\mu,\Phi)$ for some
\Folner{} sequence $\Phi$. From \cref{supp-gen lemma} we have that every point in 
$\supp(\mu)$ belongs to the orbit closure of $a$, and therefore 
$\mu\big(\overline{\Orb_T(a)}\big)=1$. This implies that we can replace 
$X$ with $\overline{\Orb_T(a)}$ without affecting the ergodic theoretic 
properties of the system, and then the generic point $a$ is also transitive.
Therefore, whenever we have a generic point in a system, 
we may assume without loss 
of generality that it is also transitive.
\end{remark}

\begin{proposition}
Let $\xmt$ be an ergodic $G$-system, let $a\in \gen(\mu,\Phi)$ for 
some \Folner{} sequence $\Phi$, and $E\subset X$ be a set with $\mu(E)>0$. Then there exists an ergodic 
extension $(\widetilde{X},\widetilde{\mu},\widetilde{T})$ of $\xmt$, 
a \Folner{} sequence $\widetilde{\Phi}$ and a point 
$\widetilde{a}\in\gen(\widetilde{\mu},\widetilde{\Phi})$ such that:
\begin{enumerate}
    \item[(i)] There exists a continuous factor map
    $\widetilde{\pi}:\widetilde{X}\to
    X$ with $\widetilde{\pi}(\widetilde{a})=a$.
    \item[(ii)] $(\widetilde{X},\widetilde{\mu},\widetilde{T})$ 
    has continuous factor map
    to its Kronecker factor.
    \item[(iii)] The set $\widetilde{E} = \widetilde{\pi}^{-1}(E)$, has  $\widetilde{\mu}(\widetilde{E}) = \mu(E) >0$.
    \item[(iv)] If $(\widetilde{a},\widetilde{x}_1,\widetilde{x}_2)\in 
    {\widetilde{X}^3}$
    is an \Erdos{} progression, then $(a,x_1,x_2)\in X^3$ 
    is an \Erdos{} progression,
    where $x_i=\widetilde{\pi}(\widetilde{x}_i)$, for $i=1,2$.
\end{enumerate}
\end{proposition}

\begin{proof}
The proofs of (i) and (ii) are identical to those 
in the case $G$ is the semigroup $(\N,+)$, and they
can be found in \cite[Proposition 3.20]{kmrr1}. Therefore, here we only 
provide a sketch of the proof.

Let $\zmr$ be the Kronecker factor of $\xmt$, 
and let $\pi:X\to Z$ be a factor
map. Define $\widetilde{X}=X\times Z$ and $\widetilde{T}=T\times R$, 
consider the map
$\rho:X\to\widetilde{X}$, given by $\rho(x)=(x,\pi(x))$, and then define
$\widetilde{\mu}=\rho\mu$. Then the map
$\rho: X \to \widetilde{X}$ is an isomorphism of the $G$-systems
$(\widetilde{X},\widetilde{\mu},\widetilde{T})$ and $\xmt$, and therefore, 
since $\xmt$ is ergodic, we get that 
$(\widetilde{X},\widetilde{\mu},\widetilde{T})$ is also ergodic.
In addition, the projection on the first coordinate
$\widetilde{\pi}:\widetilde{X} \to X$ is a continuous
factor map of the systems.

By \cref{atrans}, we may assume that the point $a$ is transitive. 
Then we can use
\cref{HK_generic} to find a point $z\in Z$ and \Folner{} sequence 
$\widetilde{\Phi}$ such that \eqref{eq_ld1} holds for all $f_1\in C(X)$
and $f_2 \in C(Z)$. Using the definition of the measure 
$\widetilde{\mu}$, it is
not too difficult to see that for any continuous 
function $F\in C(\widetilde{X})$,
$$ \lim_{N\to \infty} \frac{1}{|\widetilde{\Phi}_N|}
\sum_{g\in \widetilde{\Phi}_N}
F(T_g \times R_g)(a,z) = \int_{\widetilde{X}} F \d \widetilde{\mu},$$
which means that the point $\widetilde{a}= (a,z)$ is in 
$\gen(\widetilde{\mu},\widetilde{\Phi})$. In addition, 
$\widetilde{\pi}(\widetilde{a})=a$.

Now, as the systems $\xmt$ and 
$(\widetilde{X},\widetilde{\mu},\widetilde{T})$ are
isomorphic, their Kronecker factors are also isomorphic, so we may assume that 
$\zmr$ is the Kronecker factor of 
$(\widetilde{X},\widetilde{\mu},\widetilde{T})$. Then 
the projection on the second coordinate 
$p: \widetilde{X} \to Z$ is a continuous factor map from 
$(\widetilde{X},\widetilde{\mu},\widetilde{T})$ to its Kronecker factor.
\par 
Note that (iii) is immediate from the definition of factors and factor maps. Hence, it remains to prove (iv). By assumption, 
$(\widetilde{T}_{g_n}\times\widetilde{T}_{g_n})
(\widetilde{a},\widetilde{x}_1)
\to(\widetilde{x}_1,\widetilde{x}_2)$ for some $(g_n)_{n\in\N}$
in $G$. We now notice that
$$T_{g_n}(a)=T_{g_n}(\widetilde{\pi}(\widetilde{a}))
=
\widetilde{\pi}(\widetilde{T}_{g_n})
\to\widetilde{\pi}(\widetilde{x}_1)=x_1,$$ 
since $\widetilde{\pi}$ is a continuous factor map. 
Similarly, we get $T_{g_n}(x_1)\to x_2$ and the result follows.
\end{proof}

This proposition allows us to reduce \cref{dr2}
to the case of ergodic $G$-systems with continuous factor maps to
their Kronecker factor as desired. 
Evidently, \cref{dr2} follows from the following:

\begin{theorem}[Reduction to $G$-systems with continuous factor 
maps to the Kronecker factor]\label{dr3}
Let $G$ be a square absolutely continuous group, 
let $\xmt$ be an ergodic $G$-system admitting a continuous 
factor map to its Kronecker factor, and
let $a\in \gen(\mu,\Phi)$ for some \Folner{} sequence
$\Phi$. If $E\subset X$ is clopen and $\mu(E)>0$, 
then there exist $t\in G$,
$x_1\in E$ and $x_2\in T^{-1}_tE$ such that 
$(a,x_1,x_2)\in X^3$ forms an \Erdos{} progression.
\end{theorem}

The proof of \cref{dr3} will be given in the next section.

\section{Measures on \Erdos{} progressions and the proof of \cref{dr3}}
\label{measures section}

In this section we prove \cref{dr3}, and consequently, Theorem 
\ref{mt1}.

\subsection{Measures on \Erdos{} progressions}
In what follows we fix a square absolutely continuous group $G$. 
We also fix an ergodic $G$-system $\xmt$. 
In addition, as per the assumptions of \cref{dr3}, we assume 
that $\xmt$ admits
a continuous factor map to its Kronecker factor. The Kronecker 
factor of $\xmt$ is denoted
by $\zmr$ and $\pi:X\to Z$ stands for the continuous factor map. 
According to \cref{Kron_rot},
$Z=K/H$ where $K$ is a compact group and $H$ is a closed subgroup of $K$. 
We denote by $p$ the natural projection $p: K \to K/H$, 
$p(k)=kH $. We also fix a bi-invariant metric $d_K$ on $K$, 
that is a metric on $K$ compatible 
with the topology on $K$ such that for all $u,v,w \in K$, 
$d_K(uv,uw)=d_K(v,w)=d_K(vu, wu)$.

Moreover, $m$ is the (left) Haar measure on $Z$, which is 
given as the push forward of the (left) Haar measure 
$m_K$ in $K$ by the natural projection $K\to Z=K/H$. We remark that since $K$ is compact,
it is unimodular, so $m_K$ is two-sided invariant. 
Finally, the action $R=(R_g)_{g\in G}$ is given by $R_g(z) = \alpha(g)z$, 
where $\alpha:G\to K$ is a group homomorphism with dense image.
Also, $\pi_i:X\times X\to X$ denotes the projection to the $i$-th coordinate,
for $i=1,2$. Moreover, 
$z\mapsto \eta_z$ is a fixed disintegration of $\mu$ 
over the continuous factor map $\pi$.

\begin{defn}
Consider the squaring map $s_K : K \to K, s_K(k)=k^2$, on $K$. We define the 
Borel probability measure $m_{K^2}$ on $K$ as the push-forward of 
the Haar measure $m_K$ under the map $s_K$, i.e., the measure on $K$, given by
$m_{K^2}(A) = m_K(s_K ^{-1}(A))$, for each Borel $A\subset K$.
\end{defn}

We will prove the following lemma, which is a key ingredient that will allow us 
to define the measures in order to study \Erdos{} progressions. 

\begin{lemma}\label{abs_c}
The measure $m_{K^2}$ is absolutely continuous with respect to $m_K$.
\end{lemma}

Before we prove \cref{abs_c}, let us state and prove an auxiliary lemma that 
will be used throughout this section:

\begin{lemma}\label{weak_ast}
Let $\Psi= (\Psi_{N})_{N\in \N}$ be any \Folner{} sequence in $G$. Then 
the sequence of measures $(\nu_N)_{N\in \N}$ defined as 
\begin{equation*}
\nu_N := \frac{1}{|\Psi_N|} \sum_{g\in \Psi_N}
\delta_{\alpha(g)}
\end{equation*}
converges in the weak$^{\ast}$ topology to the Haar measure $m_K$ on $K$.
\end{lemma}

\begin{proof}
The space of Borel probability measures on $K$ is weak$^{\ast}$ compact and 
metrizable, so in order to prove the result, it suffices to prove that 
if $(\nu_{N_j})_{j\in \N}$ is a convergent subsequence of $(\nu_N)_{N\in \N}$,
then it converges to the Haar measure $m_K$. 

Let $(\nu_{N_j})_{j\in \N}$ be a subsequence of $(\nu_N)_{N\in \N}$
which converges in the weak$^{\ast}$ topology to a 
measure $\nu$ on $K$. To prove 
that $\nu$ is the Haar measure on $K$, it suffices to prove that for 
all continuous functions $h$ on $K$ and all $k\in K$, we have that 
$\int_K h(k y) \d \nu(y) = \int_K h(y) \d \nu(y)$.

Let $d_K$ be a translation invariant metric on
$K$. Let also $h: K \to \C$ be continuous, $k \in K$ and
$\epsilon >0$. Since $K$ is compact, we have that $h$ is uniformly 
continuous, so there is $\delta >0$ such that if $d_K(y_1, y_2)< \delta$, then 
$| h(y_1) - h(y_2)| < \epsilon$. Recall that $(\alpha(g))_{g\in G}$ 
is dense in $K$, so there is $g_0$ such that $d_K(k, \alpha(g_0)) < \delta$.
We then have that for all $y \in K$, $d_K(ky, \alpha(g_0) y) < \delta$, so 
$|h(ky)-h(\alpha(g_0)y)| < \epsilon$, and we obtain that 
\begin{align*}
\bigg| \int_K h(y) \d  \nu(y) -  \int_K h(ky) \d  \nu(y) \bigg| 
& \leq 
\bigg| \int_K h(y) \d  \nu(y) -  \int_K h(\alpha (g_0)y) \d  \nu(y) \bigg|  \\
& \hspace*{0.04cm} + \bigg| \int_K h(\alpha (g_0)y) \d  \nu(y) -  
\int_K h(ky) \d  \nu(y) \bigg|
\\
& \leq \bigg| \int_K h(y) \d  \nu(y) -  \int_K h(\alpha (g_0)y) \d  \nu(y) 
\bigg| + \epsilon.
\end{align*}

From the continuity of $h$ and the definition of $ \nu$ we have that 
\begin{align*}
\int_K h(\alpha(g_0)y) \d  \nu (y) 
&=
\lim_{j\to \infty} \frac{1}{|\Psi_{N_j}|}
\sum_{g\in \Psi_{N_j}} h(\alpha(g_0) \alpha(g))
=
\lim_{j\to \infty} \frac{1}{|\Psi_{N_j}|}
\sum_{g\in \Psi_{N_j}} h(\alpha(g_0 g))  \\
&=
\lim_{j \to \infty} \frac{1}{|\Psi_{N_j}|}
\sum_{g\in g_0 \Psi_{N_j}} h(\alpha(g)) 
=
\lim_{j \to \infty} \frac{1}{|\Psi_{N_j}|}
\sum_{g\in \Psi_{N_j}} h(\alpha(g)) \\
&=
\int_K h(y) \d  \nu (y) ,
\end{align*}
so combining with the previous we get that 
$\big| \int_K h(y) \d  \nu(y) -  \int_K h(ky) \d  \nu(y) \big| 
\leq \epsilon$, and since $\epsilon$ was arbitrary, we obtain that 
$ \int_K h(y) \d  \nu(y) =  \int_K h(ky) \d  \nu(y) $, 
which proves that 
$ \nu = m_K$, and concludes the proof. 
\end{proof}

\begin{proof}[Proof of \cref{abs_c}]
As $K$ is compact and metrizable, the measures $m_K, m_{K^2}$ are regular. In
particular, for each Borel $A$ we have 
\begin{equation*}
m_K(A) = \sup_{\substack{C\subset A \\ C \: \text{compact}}} m_K(C) 
=
\inf_{\substack{O \supset A \\ O \: \text{open}}} m_K(O)
\:\: \text{ and } \:\:
m_{K^2}(A) = \sup_{\substack{C\subset A \\ C \: \text{compact}}} m_{K^2}(C)
=
\inf_{\substack{O \supset A \\ O \: \text{open}}} m_{K^2}(O).
\end{equation*}
Therefore, to prove that $m_{K^2}$ is absolutely continuous with respect 
to $m_K$, it suffices to prove that for each compact set $C \subset K$,
if $m_K (C) =0$, then $m_{K^2}(C)=0$. 

Let $C \subset K$ be a non-empty (for otherwise the result is trivial)
compact with $m_K(C) =0$ and let $\epsilon >0$. As $G$ is square 
absolutely continuous, we know that there are two
\Folner{} sequences $\Phi$ and $\Psi$ in $G$ and a $\delta >0$ such that for any 
$u:G\to [0,1]$ satisfying 
$\limsup_{N\to\infty}\frac{1}{|\Phi_N|}\sum_{g\in\Phi_N} u(g) < \delta,$
we have that
$\limsup_{N\to\infty}\frac{1}{|\Psi_N|}\sum_{g\in\Psi_N} u(g^2) < \epsilon.$
Since 
$$0=m_K(C)= \inf_{\substack{O \supset C \\ O \: \text{open}}} m_K(O),$$ we 
can pick an open set $O \supset C$ with $m_K(O) < \delta$.
By Urysohn's lemma, we know that there is a continuous function 
$f : K \to [0,1]$ such that $f=1$ on $C$ and $f=0$ outside $O$.
By \cref{weak_ast} we then have that 
\begin{equation*}
    \delta > m_K(O) \geq \int_K f(k) \d m_K(k) =
    \lim_{N\to \infty} \frac{1}{|\Phi_N|} \sum_{g\in \Phi_N} f(\alpha(g)),
\end{equation*}
and by the choice of $\delta$ we get that 
\begin{equation*}
\limsup_{N\to \infty} \frac{1}{|\Psi_N|} \sum_{g\in \Psi_N} f(\alpha(g^2)) 
<\epsilon.
\end{equation*}
From the definition of $m_{K^2}$ and the continuity of $k \mapsto f(k^2)$
we then obtain
\begin{equation*}
\int_K f(k) \d m_{K^2}(k) = \int_K f(k^2) \d m_K (k) =
\lim_{N\to \infty} \frac{1}{|\Psi_N|} \sum_{g\in \Psi_N} f(\alpha(g)^2) =
\lim_{N\to \infty} \frac{1}{|\Psi_N|} \sum_{g\in \Psi_N} f(\alpha(g^2))
< \epsilon,
\end{equation*}
where for the third equality we use that $\alpha$ is a group homomorphism.
So after all we have 
$m_{K^2}(C) \leq \int_K f(k) \d m_{K^2}(k) < \epsilon$, and as $\epsilon$ 
was arbitrary, we have that $m_{K^2}(C) =0$. This concludes the proof.
\end{proof}

Now, we define a measure $\sigma$ on $X\times X$, and we 
want $\sigma$ to be defined as a natural 
measure to study \Erdos{} progressions. 
It is not hard to see that \Erdos{} progressions on $K/H$ 
are exactly the triplets of the form
$(z,kz,k^2z)$ for some $k\in K$ and $z\in K/H$. Therefore, 
following the definition 
given in \cite{kmrr2}, we will define these measures as 
the natural measures on points 
$(x_1,x_2)\in X\times X$ to find \Erdos{} progressions on 
$K/H$ starting at $\pi(a)$, namely
$(\pi(a),k\pi(a),k^2\pi(a))$, for $k\in K$. 

\begin{defn}
We define the measure $\sigma$ on $X \times X$, given by 
 \begin{equation}\label{meas-s}
    \sigma := \int_{K}\eta_{k\pi(a)}
    \times\eta_{k^2\pi(a)}\d m_K(k).
\end{equation}
\end{defn}

Let us comment on why the measure $\sigma$ is well-defined. Let 
$k_0 \in K$ be such that $\pi(a)=k_0 H$. 
Since $\eta$ is $m$-almost everywhere defined on $Z$ and Borel
measurable, we can consider
a Borel measurable set $Z' \subset Z$ with $m(Z')=1$ such that 
$\eta_z$ is defined for all 
$z\in Z'$. Since $p$ is a Borel measurable map and $pm_K =m$, 
the set $K' = p^{-1}(Z')$ is 
Borel measurable and has $m_K(K')=1$.
Then also $m_K(K' k_0 ^{-1})=1$, and then it is not difficult to check that 
the map $K \to M(X)$, $k \mapsto \eta_{k \pi(a)}$ is defined on 
$K' k_0 ^{-1}$.
On the other hand, from \cref{abs_c}, we have that $m_{K^2}$ 
is absolutely continuous with 
respect to $m_K$, so $1=m_{K^2}(K' k_0 ^{-1})=
m_K(s_K^{-1} (K' k_0 ^{-1}))$, and again
it is not too difficult to check that 
the map $K \to M(X)$, $k \mapsto \eta_{k^2 \pi(a)}$ is defined on 
$s_K^{-1} (K' k_0 ^{-1})$. 
So, after all, the map 
$K \to M(X\times X)$, 
$k \mapsto \eta_{k \pi(a)} \times \eta_{k^2 \pi(a)} $ is defined 
on $K' k_0 ^{-1} \cap s_K^{-1} (K' k_0 ^{-1})$ and 
$m_K( K' k_0 ^{-1} \cap s_K^{-1} (K' k_0 ^{-1}) )=1 $, i.e.,
$k \mapsto \eta_{k \pi(a)} \times \eta_{k^2 \pi(a)}$ is 
$m_K$-almost everywhere defined. Also, since all the maps involved
in the definition of 
$k \mapsto \eta_{k \pi(a)} \times \eta_{k^2 \pi(a)}$ are Borel 
measurable, we have that 
$k \mapsto \eta_{k \pi(a)} \times \eta_{k^2 \pi(a)}$ is also 
Borel measurable. 
Therefore, $\sigma$ is indeed well-defined.

Using the invariance of $m_K$ we can express $\sigma$ as
\begin{equation}\label{alternative sigma_s}
    \sigma 
    = \int_K \eta_{kk_0H}\times\eta_{k^2k_0H} \d m_K(k)
    = \int_K \eta_{kH}\times \eta_{kk_0^{-1}kH} \d m_K(k)
\end{equation}

\begin{proposition}\label{s_props}
The measure $\sigma$ has the following properties:
\begin{enumerate}
    \item[(i)] $\pi_1\sigma=\mu$.
    \item[(ii)] $\pi_2 \sigma$ is absolutely continuous with
    respect to $\mu$.
\end{enumerate}
\end{proposition}

\begin{proof}[Proof of \cref{s_props}]
(i) Using \eqref{alternative sigma_s}, we have that
$$\pi_1\sigma = \int_K \eta_{kH} \d m_K(k) = \int_Z \eta_z \d m(z) = \mu.$$
(ii) From the definition of $\sigma$ we have that 
$\pi_2 \sigma = \int_{K} \eta_{k^2\pi(a)} \d m_K(k).$
Fix $k_0$ such that $\pi(a)=k_0 H$.

Let $A\subset X$ with $\mu(A)=0$. Then $\mu(A)=\int_Z \eta_z(A) \d m(z)$, 
so we get that there is a set $Z'\subset Z$ with $m(Z')=1$ such that 
for all $z\in Z'$, $\eta_z(A)=0$. As $pm_K =m$, we have $m_K(p^{-1}(Z'))=1$ and 
then also $m_K ((p^{-1}Z') k_0 ^{-1} )=1 $. Finally, using \cref{abs_c}
we get that $m_{K^2}((p^{-1}Z') k_0 ^{-1} )=
m_K (s_K ^{-1}((p^{-1}Z') k_0 ^{-1} ))=1 $. 
For each $k \in s_K ^{-1}((p^{-1}Z') k_0 ^{-1} )$ we have that 
$k^2 \pi(a) \in Z'$, so $\eta_{k^2 \pi(a)}(A)=0$, and therefore
$\pi_2 \sigma(A)=0$.
\end{proof}

\begin{theorem}\label{sigma.thm}
For any set $E\subset X$ with $\mu(E)>0$, we have that 
\begin{equation*}
    \sigma \bigg( E\times \bigcup_{t\in G}T^{-1}_tE \bigg) >0.
\end{equation*}
\end{theorem}

\begin{proof}
Let $E\subset X$ with $\mu(E)>0$ and recall that we want to show that the 
set $E\times \bigcup_{t\in G}T^{-1}_tE$ has positive measure 
with respect to $\sigma$. We begin by expressing this set as 
$$E\times \bigcup_{t\in G}T^{-1}_tE 
= (E\times X)\cap\bigg(X\times\bigcup_{t\in G}T^{-1}_tE\bigg).$$
By \cref{s_props} (i), we have that $\sigma(E\times X)=\mu(E)>0$.
Therefore, it is enough to show that 
\begin{equation}
    \sigma\bigg(X\times\bigcup_{t\in G}T^{-1}_tE\bigg)=1.
\end{equation}
Notice that the set $\bigcup_{t\in G} T_t^{-1}E$ is clearly $T$-invariant
and since $\mu$ is ergodic and $\mu(E)>0$, it follows that
$\mu\big(\bigcup_{t\in G} T_t^{-1}E\big)=1$.
By \cref{s_props} (ii),
$\pi_2 \sigma$ is absolutely continuous with respect to $\mu$, so
\begin{equation*}
    1 = \pi_2 \sigma\bigg(\bigcup_{t\in G} T_t^{-1}E\bigg) = 
    \sigma\bigg(X \times \bigcup_{t\in G} T_t^{-1}E\bigg),
\end{equation*}
which concludes the proof.
\end{proof}

\subsection{A continuous ergodic decomposition}
In this subsection we will define measures $\lambda_{(x_1,x_2)}$, 
for $(x_1,x_2)\in X\times X$, in a way
that $(x_1,x_2)\mapsto \lambda_{(x_1,x_2)}$ will be 
a continuous ergodic decomposition of $\mu\times\mu$, i.e.,
$(x_1,x_2)\mapsto \lambda_{(x_1,x_2)}$ will be both a continuous map
and an ergodic decomposition of $\mu \times \mu$.
We follow the definition given in \cite[Eq. (3.10)]{kmrr1}
and \cite[Eq. (3.1)]{kmrr2}.

\begin{defn}
For $(x_1,x_2)\in X\times X$ we define the measures 
$\lambda_{(x_1,x_2)}$ on $X\times X$ by
\begin{equation}\label{def.l}
    \lambda_{(x_1,x_2)}= \int_K \eta_{k\pi(x_1)}
    \times\eta_{k\pi(x_2)} \d m_K(k).
\end{equation}
\end{defn}

Given $x_1,x_2\in X$, we let $k_1,k_2\in K$ be such that $\pi(x_i)=k_iH$, for $i=1,2$.
Then, using the invariance of $m_K$, we can write
\begin{equation}\label{alternative lambda}
    \lambda_{(x_1,x_2)}
    = \int_K \eta_{kk_1H}\times\eta_{kk_2H}\d m_K(k)
    = \int_K \eta_{kH}\times\eta_{kk_1^{-1}k_2H} \d m_K(k).
\end{equation}

\begin{theorem}\label{c.e.d}
The map $(x_1,x_2)\mapsto\lambda_{(x_1,x_2)}$ is a continuous ergodic
decomposition of $\mu\times\mu$ in the following sense:
\begin{itemize}
    \item[(i)]It is a continuous map.
    \item[(ii)]It satisfies  
    $\int_{X \times X} \lambda_{(x_1,x_2)} \d (\mu \times \mu)(x_1,x_2)
    =\mu\times\mu$.
    \item[(iii)]The $G$-system $(X\times X, \lambda_{(x_1,x_2)}, 
    T\times T)$
    is ergodic for $\mu\times\mu$-almost every $(x_1,x_2)\in X\times X$.
\end{itemize}
In addition, for any $x_1,x_2\in X$ we have that 
\begin{equation}\label{l.inv}
\lambda_{(x_1,x_2)}=\lambda_{(T_gx_1,T_gx_2)}
\end{equation}
for any $g\in G$.
\end{theorem}

\begin{proof}
For the proof of (i) and (ii) we refer to 
\cite[Proposition 3.11]{kmrr1}, as the proof there can be directly 
adapted to our case.
We will now prove (iii). It is not too difficult to see that for all 
$(x_1, x_2) \in X\times X$ and for all
$g \in G$, $(T_g \times T_g) \lambda_{(x_1, x_2)} = 
\lambda_{(x_1, x_2)}$, i.e., $\lambda_{(x_1, x_2)}$ is 
$T\times T$-invariant. Therefore, to prove (iii),
it suffices to prove that there is some 
\Folner{} sequence $\Psi$ in $G$ such that for
$(\mu\times\mu)$-almost every
$(x_1,x_2)\in X\times X$ and all bounded and measurable 
functions $F$ on $X\times X$,
$$\lim_{N\to\infty}\frac{1}{|\Psi_N|}\sum_{g\in\Psi_N}(T_g\times T_g)F
= \int_{X\times X}F\d\lambda_{(x_1,x_2)},$$
in $L^2(X\times X,\lambda_{(x_1,x_2)})$. 

Now, since $X$ is a compact metric space,
there is a countable family 
of continuous functions $(f_k)_{k\in \N}$ which is dense 
in $L^p(\nu)$ for all $p \in [1, + \infty)$ and all 
Borel probability measures $\nu$ on $X$. Then, it is not too 
difficult to see that the set consisting of finite linear 
combinations of functions the form 
$(f_{j_1} \otimes f_{j_2})_{j_1, j_2 \in \N}$ is dense 
in $L^2(\rho)$ for all Borel probability measures $\rho$ on 
$X\times X$. 
Hence, using an approximation argument, it suffices to prove that 
there is a \Folner{} $\Psi$ in $G$
and a set $W\subset X \times X$ 
with $(\mu \times \mu)(W) =1 $ such that 
for all $(x_1, x_2) \in W$ and for all $j_1, j_2 \in \N$
\begin{equation*}
\lim_{N\to\infty}\frac{1}{|\Psi_N|}\sum_{g\in\Psi_N}(T_g\times T_g)
(f_{j_1} \otimes f_{j_2})
= \int_{X\times X} f_{j_1} \otimes f_{j_2} \d\lambda_{(x_1,x_2)},
\end{equation*}
in $L^2(X\times X,\lambda_{(x_1,x_2)})$.

\textbf{Step 1.} Let $\Phi$ be any \Folner{} sequence in $G$.
Then, using \cref{char}, we get that for each $j_1, j_2 \in \N$,
\begin{equation*}
\lim_{N\to\infty}\frac{1}{|\Phi_N|}\sum_{g\in\Phi_N}
(T_{g} \times T_{g}) (f_{j_1}\otimes f_{j_2})
=
\lim_{N\to\infty}\frac{1}{|\Phi_N|}\sum_{g\in\Phi_N}
(T_g \times T_g) (\E_{\mu}(f_{j_1}\: | \: Z)\otimes  
\E_{\mu}(f_{j_2}\: | \: Z))
\end{equation*}
in $L^2(\mu \times \mu)$. Combining this with (ii) yields
\begin{align*}
    & \lim_{N\to \infty}
    \int_{X\times X}
    \int_{X\times X}
    \bigg|\frac{1}{|\Phi_N|} \sum_{g\in \Phi_N} 
    (T_g \times T_g)(f_{j_1}\otimes f_{j_2})(y_1, y_2) \\
    &\:\:\:\:\:\:
    - 
    \frac{1}{|\Phi_N|} \sum_{g\in \Phi_N} 
    (T_g \times T_g)(\E_{\mu}(f_{j_1}\: | \: Z)\otimes  
    \E_{\mu}(f_{j_2}\: | \: Z))(y_1, y_2)\bigg| ^2
    \d \lambda_{(x_1, x_2)}(y_1, y_2)
    \d (\mu \times \mu)(x_1, x_2) 
    = 0.
\end{align*}
Then for each $j_1,j_2\in\N$ we can find a sub-\Folner{} sequence 
$\widetilde{\Phi}$ of $\Phi$, depending on $j_1,j_2$,
such that 
for $(\mu \times \mu)$-almost every $(x_1, x_2)\in X \times X$, we have 
\begin{align*}
    & \lim_{N\to\infty}\int_{X\times X}
    \bigg|\frac{1}{|\widetilde{\Phi}_N|} 
    \sum_{g\in\widetilde{\Phi}_N}
    (T_g \times T_g)(f_{j_1}\otimes f_{j_2}) \\
    & \hspace*{5.7cm}- \frac{1}{|\widetilde{\Phi}_N|} 
    \sum_{g\in\widetilde{\Phi}_N}
    (T_g \times T_g)(\E_{\mu}(f_{j_1}\: | \: Z)\otimes  
    \E_{\mu}(f_{j_2}\: | \: Z))\bigg|^2 \d\lambda_{(x_1,x_2)} 
    = 0,
\end{align*}
and since the limits of both averages above exist by \cref{met},
we have that, for $(\mu\times\mu)$-almost every $(x_1,x_2)\in X\times X$,
\begin{align*}
    \lim_{N\to\infty}\frac{1}{|\widetilde{\Phi}_N|} 
    \sum_{g\in\widetilde{\Phi}_N}
    (T_g \times T_g)(f_{j_1}\otimes f_{j_2})
    = 
    \lim_{N\to\infty}\frac{1}{|\widetilde{\Phi}_N|} 
    \sum_{g\in\widetilde{\Phi}_N}
    (T_g \times T_g)(\E_{\mu}(f_{j_1}\: | \: Z)\otimes  
    \E_{\mu}(f_{j_2}\: | \: Z))
\end{align*}
in $L^2(X\times X,\lambda_{(x_1,x_2)})$.
Since the family $(f_{j_1} \otimes f_{j_2})_{j_1, j_2 \in \N}$
is countable, then using a diagonal argument one can find a 
\Folner{} sequence 
$\Psi$ and a set $W_1 \subset X \times X$ with
$(\mu \times \mu)(W_1) =1$ such that 
for all $(x_1, x_2)\in W_1$ and $j_1, j_2 \in \N$, we have that 
\begin{equation}\label{eq_ld_ch}
\lim_{N\to\infty}\frac{1}{|\Psi_N|} \sum_{g\in \Psi_N} 
(T_g \times T_g)(f_{j_1}\otimes f_{j_2}) 
=
\lim_{N\to\infty}\frac{1}{|\Psi_N|} \sum_{g\in \Psi_N}
(T_g \times T_g)(\E_{\mu}(f_{j_1}\: | \: Z)\otimes  
    \E_{\mu}(f_{j_2}\: | \: Z))
\end{equation}
in $L^2(X\times X,\lambda_{(x_1,x_2)})$.

\textbf{Step 2.} 
Consider the sequence of probability measures 
on $K$ defined by $\nu_N := \frac{1}{|\Psi_N|} \sum_{g\in \Psi_N}
\delta_{\alpha(g)}$. From \cref{weak_ast} we know that 
$\nu_N \to m_K$ as $N \to \infty$ in the weak$^{\ast}$ topology.  

Let $\phi_1, \phi_2 : Z \to \C$ be continuous.
For each $z_1, z_2 \in Z$, consider the function 
$\phi_{z_1, z_2} : K \to \C$ defined by 
$\phi_{z_1, z_2}(k)= \phi_1(k z_1) \phi_2(k z_2)$. Then, 
$\phi_{z_1, z_2}$ is continuous, so we have that 
\begin{align*}
\int_K \phi_1(k z_1) \phi_2(k z_2) \d m_K(k) &=
\int_K \phi_{z_1, z_2}(k) \d m_K(k) 
=
\lim_{N\to \infty} \frac{1}{|\Psi_N|} \sum_{g\in \Psi_N} 
\phi_{z_1, z_2}(\alpha(g))\\
&=
\lim_{N\to \infty} \frac{1}{|\Psi_N|} \sum_{g\in \Psi_N} 
\phi_1(\alpha(g) z_1) \phi_2(\alpha(g) z_2)\\
&=
\lim_{N\to \infty} \frac{1}{|\Psi_N|} \sum_{g\in \Psi_N} 
(R_g \times R_g)(\phi_1\otimes \phi_2)( z_1, z_2).
\end{align*}
Since the previous holds for all $z_1, z_2 \in Z$, using the dominated 
convergence theorem, we get that 
\begin{equation*}
\lim_{N \to \infty} \int_{Z \times Z} 
\bigg| \frac{1}{|\Psi_N|} \sum_{g\in \Psi_N} 
(R_g \times R_g)(\phi_1\otimes \phi_2)(z_1, z_2)
-
\int_K \phi_1(k z_1) \phi_2(k z_2) \d m_K(k) \bigg|^2 \d (m\times m)
(z_1, z_2)=0,
\end{equation*}
i.e., the sequence 
$\frac{1}{|\Psi_N|} \sum_{g\in \Psi_N} 
(R_g \times R_g)(\phi_1\otimes \phi_2)$ converges in 
$L^2(Z \times Z, m\times m )$ as $N\to \infty$ to the function 
$(z_1, z_2) \mapsto \int_K \phi_1(k z_1) \phi_2(k z_2) \d m_K(k)$.

\textbf{Step 3.} 
Given two bounded measurable maps $h_1, h_2 : Z \to \C$,
approximating them in $L^2(Z,m)$ by two continuous functions $\phi_1, \phi_2$
and using Step 2,
one can prove that
$\frac{1}{|\Psi_N|}\sum_{g\in \Psi_N} 
(R_g \times R_g)(h_1\otimes h_2)$ converges in 
$L^2(Z \times Z, m\times m )$ as $N\to \infty$ to the function 
$(z_1, z_2) \mapsto \int_K h_1(k z_1) h_2(k z_2) \d m_K(k)$.
Since $\pi: \xmt \to (Z, m, R)$ is a factor map, it is not too difficult then to 
see that 
\begin{equation}\label{eq_ltwo}
\frac{1}{|\Psi_N|} \sum_{g\in \Psi_N} 
(T_g \times T_g) (h_1 \circ \pi \otimes h_2 \circ \pi) 
\to 
\Big[(x_1, x_2) \mapsto \int_K h_1(k \pi(x_1)) h_2(k \pi(x_2)) \d m_K(k)\Big]  
\end{equation}
as $N\to \infty$ in
$L^2(X \times X, \mu \times \mu )$.

For each $j \in \N$, $\E_{\mu}(f_j\:  | \: Z)$ can be viewed 
either as a function 
on $Z$ or as a function on $X$ measurable with respect to $\pi^{-1}(Z)$. In 
this proof, we always view $\E_{\mu}(f_j \: |\: Z)$ as a function on $X$ 
measurable with respect to $\pi^{-1}(Z)$. For each $j \in \N$, let 
$\psi_j $ be $\E_{\mu}(f_j \: | \: Z)$ when viewed as a function on
$Z$, so we have
that for $\mu$-almost every $x \in X$, 
$\psi_j \circ \pi (x) = \E_{\mu}(f_j \: | \: Z)(x)$. 
Then for each $j_1, j_2 \in \N$ and $(x_1, x_2) \in X\times X$, we have
\begin{align}\label{eq_disint_1}
    \int_{X \times X} f_{j_1} (y_1)
    f_{j_2}(y_2) \d \lambda_{(x_1, x_2)} (y_1, y_2)
    &= 
    \int_K 
    \int_{X \times X} f_{j_1} (y_1) f_{j_2}(y_2) 
    \d (\eta_{k\pi(x_1)}
    \times\eta_{k\pi(x_2)})(y_1,y_2) \d m_K(k) \nonumber \\
    &= \int_K 
    \int_{X}  f_{j_1} (y_1)
    \d \eta_{k\pi(x_1)} (y_1)
    \int_{X} f_{j_2}(y_2) 
    \d \eta_{k\pi(x_2)} (y_2)
    \d m_K(k) \nonumber \\
    & = \int_K 
    \psi_{j_1}(k\pi(x_1))
    \psi_{j_1}(k\pi(x_1))
    \d m_K(k) 
\end{align}
where the last equality follows using \eqref{eq_disint}. After all, combining 
\eqref{eq_ltwo} and \eqref{eq_disint_1} we get that for each 
$j_1, j_2 \in \N$, 
\begin{equation*}
\frac{1}{|\Psi_N|} \sum_{g\in \Psi_N} 
(T_g \times T_g) (\E_{\mu}(f_{j_1} \: | \: Z) \otimes 
\E_{\mu}(f_{j_2} \: | \: Z)) 
\to 
\Big[(x_1, x_2) \mapsto 
\int_{X \times X} f_{j_1} \otimes 
f_{j_2}\d \lambda_{(x_1, x_2)} \Big] 
\end{equation*}
as $N\to \infty$ in
$L^2(X \times X, \mu \times \mu )$. Now, since the family 
$f_{j_1}\otimes f_{j_2}$ is countable, using (ii) and a diagonal argument as in Step 1,
one can find a sub-\Folner{} sequence of $\Psi$, which by abuse of notation we again denote by $\Psi$,
and a set $W_2 \subset X \times X$ with $(\mu \times \mu)(W_2) =1$ such that 
for all $(x_1, x_2)\in W_2$ and $j_1, j_2 \in \N$,
\begin{equation}\label{eq_ld_ch_2}
\lim_{N\to\infty}\frac{1}{|\Psi_N|} \sum_{g\in \Psi_N} 
(T_g \times T_g)(\E_{\mu}(f_{j_1}\: | \: Z)\otimes  
\E_{\mu}(f_{j_2}\: | \: Z)) 
=
\int_{X \times X} f_{j_1} \otimes 
f_{j_2}\d \lambda_{(x_1, x_2)}
\end{equation}
in $L^2(X\times X, \lambda_{(x_1, x_2)})$.

Let $W = W_1 \cap W_2$. Then $(\mu \times \mu)(W) =1$,
and combining \eqref{eq_ld_ch} and \eqref{eq_ld_ch_2},
we get that 
for all $(x_1, x_2) \in W$ and all $j_1, j_2 \in \N$,
$$\lim_{N\to\infty}\frac{1}{|\Psi_N|}\sum_{g\in\Psi_N}
(T_g\times T_g)(f_{j_1} \otimes f_{j_2})
= \int_{X\times X} f_{j_1} \otimes f_{j_2} \d\lambda_{(x_1,x_2)}$$
in $L^2(X\times X,\lambda_{(x_1,x_2)})$, which was to be proved.

To conclude the proof we are left with showing \eqref{l.inv}. To this 
end, using 
the invariance of $m_K$, for any $g\in G$, we have that
\begin{align*}
    \lambda_{(T_gx_1,T_gx_2)}
    = \int_Z  \eta_{k\alpha(g)\pi(x_1)}\times \eta_{k\alpha(g)\pi(x_2)}\d m_K(k)
    & = \int_Z \eta_{k\pi(x_1)}\times \eta_{k\pi(x_2)}\d m_K(k\alpha(g)^{-1}) \\
    & = \int_Z \eta_{k\pi(x_1)}\times \eta_{k\pi(x_2)}\d m_K(k)
    = \lambda_{(x_1,x_2)}. 
\end{align*}
The proof of the theorem is now complete.
\end{proof}

\begin{theorem}\label{lambda.supp}
We have that 
$$\sigma(\{(x_1,x_2)\in X\times X:(x_1,x_2)\in\supp(\lambda_{(x_1,x_2)})\})=1.$$
\end{theorem}

\begin{theorem}\label{lambda.gen}
There exists some \Folner{} sequence $\Psi$ such that 
$$\sigma(\{(x_1,x_2)\in X\times
X:(a,x_1)\in\gen(\lambda_{(x_1,x_2)},\Psi)\})=1.$$
\end{theorem}

We first deal with \cref{lambda.supp}. Let us first
state and prove some results that will be useful in order to prove
\cref{lambda.supp}.

Let $\F(X)$ be the family of the closed, nonempty 
subsets of the compact metric space $(X,d_X)$.
We endow this family with the Hausdorff metric $\Dist$, defined by 
$$\Dist(A,B) = \max\bigg\{\sup_{x\in A}d_X(x,B),\sup_{y\in B}d_X(y,A)\bigg\},$$
for any $A,B\in\F(X)$.

We will need the following two lemmas, the proofs of which are omitted, 
as they can be found in \cite{kmrr2}:

\begin{lemma}\label{lem_supp}{\cite[Lemma 3.8]{kmrr2}}
Let $W$ be a compact metric space, $M(W)$ the space of Borel probability
measures on $W$ endowed with the weak$^{\ast}$ topology, and $\F(W)$ the 
space of closed, non-empty subsets of $W$ with the Hausdorff metric.
Then
\begin{itemize}
    \item The map $\nu \mapsto \supp(\nu)$ from $M(W)$ to 
    $\F(W)$ is Borel measurable. 
    \item If $x \mapsto \rho_x$ is a measurable map from $W$ to 
    $M(W)$, then $\{x \in X: x\in \supp(\rho_x) \}$ is a Borel set. 
\end{itemize}
\end{lemma}

\begin{lemma}\label{lem_supp_2}{\cite[Lemma 3.9]{kmrr2}}
The disintegration $z \mapsto \eta_z$ satisfies that the subset \\
$\{ x \in X : x\in \supp(\eta_x)\}$ of $X$ is Borel measurable and 
$$\mu(\{ x \in X : x\in \supp(\eta_x)\}) =1.$$
\end{lemma}

Using those, we can now prove the following proposition, which 
is a variant of \cite[Proposition 3.10]{kmrr2}.

\begin{proposition}\label{L}
There exists a sequence $\delta_j\to0$ such that for 
$\mu$-almost every $x\in X$ there exists 
$w\in K$ with $\pi(x) = w H$ such that for any open neighborhood $U$ of $x$, we have 
\begin{equation}\label{L eq}
    \lim_{j\to\infty}\frac{m_K(\{k\in K\colon \eta_{kH}(U)>0\}\cap \Ball_K(w,\delta_j))}
    {m_K(\Ball_K(w,\delta_j))}
    = 1,
\end{equation}
where $\Ball_K(w,\delta_j)$ denotes the ball centered at $w\in K$ 
and with radius $\delta_j$ in $K$ with respect to the fixed metric $d_K$.
\end{proposition}

\begin{proof}
Let $F: K\to \F(X)$ given by $F(k) = \supp(\eta_{kH})$. The natural
projection $p:K\to Z$
is continuous, hence Borel measurable, 
the map $z\mapsto \eta_z$ is Borel measurable, 
as $(\eta_z)_{z\in Z}$ is a 
disintegration and by \cref{lem_supp}, the map
$\nu\mapsto\supp(\nu)$ is also Borel measurable, thus, 
we obtain that their composition $F$ is also Borel measurable.
By Lusin's theorem \cite[Theorem 12.8]{Ali-Bor},
for any $j\in\N$ there exists a closed $K_j\subset K$ with 
\begin{equation}\label{K_j measure}
    m_K(K_j)>1-2^{-j}
\end{equation}
such that $F|_{K_j}$ is continuous. For $j\in\N$, using the fact that
$K$, and so $K_j$, is compact, we obtain that $F|_{K_j}$
is uniformly continuous. Therefore, 
for any $j\in\N$, there exists some $\delta_j>0$ such that for any 
$k_1,k_2\in K_j$ we have
$$d_K(k_1,k_2)\leq \delta_j \Longrightarrow \Dist(F(k_1),F(k_2))
<\frac{1}{j}.$$
\par
Fix $j \in \N$. Then by the invariance of $m_K$, there is 
$c_j>0$ such that 
for all $k \in K$, $m_K(\Ball_K(k, \delta_j)) = c_j$.
The regularity of the measure $m_K$ implies that
there is a compact set $C_j \subset \Ball_K(e_K, \delta_j)$ such that 
$m_K(C_j) \geq c_j - \frac{c_j}{2^j}$.
Now, by Urysohn's lemma, there is a continuous function 
$f_j: K \to [0, 1]$ such that 
$f_j = 1$ on $C_j$ and $f_j =0 $ outside $\Ball_K(e_K, \delta_j)$. 

Consider the set 
$$W_j = \bigg\{k \in K_j : \int_{K_j} f_j (w k^{-1}) \d m_K(w) 
\geq \bigg(1-\frac{1}{j} \bigg) \int_K f_j (w ) \d m_K(w)\bigg\}.$$
Note that 
$$\int_K f_j (w) \d m_K(w) \geq m_K(C_j) 
\geq c_j - \frac{c_j}{2^j} >0.$$
Then we can consider the function 
\begin{equation*}
\chi_j : K \to [0,1], \: \:  \chi_j (k) 
=
\frac{\int_{K_j} f_j (w k^{-1}) \d m_K(w)}{\int_K f_j (w) \d m_K(w)}
\end{equation*}
and moreover, we let 
$$A_j =\bigg \{ k \in K : \chi_j (k) \geq 1 - \frac{1}{j} \bigg\}.$$
Then we see that 
\begin{equation}\label{W_j prop}
    W_j = K_j\cap A_j.
\end{equation}
We will show that the set $W_j$ is closed.
Using the dominated convergence theorem and the fact that $f_j$ is continuous, 
one can show that if $(k_{\ell})_{\ell \in \N} $
is a sequence in $K$ and $k_{\ell} \to k$, then 
$\int_{K_j} f_j (w k_{\ell}^{-1}) \d m_K(w) 
\to \int_{K_j} f_j (w k^{-1}) \d m_K(w)$, and this proves the continuity of $\chi_j$. 
As a result, the set $A_j$ is a closed subset of $K$, and since $K_j$ is also closed, it follows 
that $W_j$ is closed.
\par
Now, using Fubini's theorem and the invariance of $m_K$ we deduce that 
\begin{align*}
    \int_K \chi_j(k) \d m_K(k) 
    & = \frac{1}{\int_K f_j (w) \d m_K(w)}\int_{K}\int_K 
    f_j (w k^{-1}) \1_{K_j}(w) \d m_K(k) \d m_K(w) \notag \\
    & = \frac{1}{\int_K f_j (w) \d m_K(w)}
    \int_{K}  \1_{K_j}(w) \int_K 
    f_j (k)  \d m_K(k) \d m_K(w) \notag \\
    & = m_K(K_j) > 1-\frac{1}{2^j} \notag
\end{align*}
and then we have that 
$$1-\frac{1}{2^j}< \int_{A_j} \chi_j(k) \d m_K(k) + \int_{A_j^\text{c}} \chi_j(k) \d m_K(k)
\leq m_K(A_j) + \bigg(1-\frac{1}{j}\bigg)m_K(A_j^\text{c})
= 1 - \frac{1}{j} + \frac{m_K(A_j)}{j},$$
which gives that
\begin{equation}\label{A_j measure}
    m_K(A_j)>1-\frac{j}{2^j}.
\end{equation}
Combining \eqref{K_j measure}, \eqref{W_j prop} and \eqref{A_j measure}, we obtain that
$$\sum_{j\in\N} m_K(K\setminus W_j)
= \sum_{j\in\N} m_K(K\setminus(K_j\cap A_j))
< \sum_{j\in\N} \frac{1+j}{2^j}
< \infty.$$
\par
Let $W = \bigcup_{J\in\N}\bigcap_{j\geq J} W_j$. It follows by the 
Borel-Cantelli lemma, using the last equation above, that $m_K(W) = 1$.
Now let $L:=\{x\in X\colon x\in\supp(\eta_{\pi(x)})\}\cap \pi^{-1}(p(W))$.
Observe that $p(W)= \bigcup_{J\in\N} p\big(\bigcap_{j\geq J} W_j\big)$.
For each $J \in \N$, $\bigcap_{j\geq J} W_j$ is a closed subset of $K$, 
thus it is compact and since $p$ is continuous, we get that
$p\big(\bigcap_{j\geq J} W_j\big)$ is also compact, thus it is Borel measurable.
As a result, we get that $p(W)$ is indeed a Borel subset of $Z$. In addition, 
$p^{-1}(p(W))\supset W$ and since $m_K(W)=1$ we have that 
$m_K(p^{-1}(p(W))=1$. Therefore, $m(p(W))=1$ and hence 
$\mu(\pi^{-1}(p(W)))=1$.
Then, in view of \cref{lem_supp_2}, it follows that 
$L$ is a Borel subset of $X$ and $\mu(L)=1$. 

We now show that elements of $L$ 
satisfy \eqref{L eq} and this will conclude the proof.
Let $x\in L= \{x\in X\colon x\in\supp(\eta_{\pi(x)})\}
\cap \pi^{-1}(p(W))$. 
Then $x\in \pi^{-1}(p(W))$ so there is $w\in W$ such that 
$\pi(x)=p(w)=wH$. 
Let $U$ be an open neighborhood of $x$. 
Then we have that there exists $J\in\N$ such that for any 
$j\geq J$, $w\in W_j$ and 
$\Ball(x,\tfrac{1}{j})\subset U$. 
\par
We now claim that 
$$\Ball_K(w,\delta_j)\cap K_j\subset \{k\in K\colon \eta_{kH}(U) >0\}.$$
To prove this, we let $w'\in\Ball_K(w,\delta_j)\cap K_j$. 
Then $d_K(w',w)<\delta_j$ and so, 
$\Dist(F(w'),F(w)) < \tfrac{1}{j}$. Notice now that 
$F(w) = \supp(\eta_{wH}) = \supp(\eta_{\pi(x)})$, and so, $x\in F(w)$,
as $x\in L$.
Then, by the definition of the Hausdorff metric,
there exists $x'\in F(w')$ with $d_X(x,x')<\tfrac{1}{j}$, 
and so, $x'\in U$, which, combined
with the fact that $x'\in F (w')$, yields 
$U\cap F(w')\neq\emptyset$. It follows that 
$\eta_{w'H}(U)>0$.
\par
It follows from the above claim that 
\begin{equation}\label{no_more_troubles_please}
\frac{m_K(\{k\in K\colon \eta_{kH}(U)>0\}\cap
\Ball_K(w,\delta_j))}{m_K(\Ball_K(w,\delta_j))}
\geq 
\frac{m_K (\Ball_K(w,\delta_j)\cap K_j)}{m_K(\Ball_K(w,\delta_j))}
=
\frac{\int_K \1_{\Ball_K(w,\delta_j)}(u) \1_{K_j}(u) \d m_K(u)}
{m_K(\Ball_K(e_K,\delta_j))} 
\end{equation}

The denominator in the right-most term in
\eqref{no_more_troubles_please} is 
smaller or equal to 
$\frac{2^j}{2^j -1} m_K(C_j)$, which then is smaller or equal to
$\frac{2^j}{2^j -1} \int_K f_j(u) \d m_K(u)$,
and therefore the expression in 
\eqref{no_more_troubles_please}
is greater or equal to 
\begin{equation}\label{no_more_troubles_please_1}
\frac{\int_K \1_{\Ball_K(w,\delta_j)}(u) \1_{K_j}(u) \d m_K(u)}
{\int_K f_j(u) \d m_K(u)}  \bigg(1-\frac{1}{2^j}\bigg).  
\end{equation}
Observe that for all $u \in K$,
$\1_{\Ball_K(w,\delta_j)}(uw)= \1_{\Ball_K(e_K,\delta_j)}(u)$, 
so we have that 
\begin{align*}
\int_K \1_{\Ball_K(w,\delta_j)}(u) \1_{K_j}(u) \d m_K(u) 
&= 
\int_K \1_{\Ball_K(w,\delta_j)}(uw) \1_{K_j}(uw) \d m_K(u) \\
&=
\int_K \1_{\Ball_K(e_K,\delta_j)}(u) \1_{K_j}(uw) \d m_K(u)
\geq 
\int_K f_j(u) \1_{K_j}(uw) \d m_K(u) \\
& = 
\int_K f_j(uw^{-1}) \1_{K_j}(u) \d m_K(u)
=
\int_{K_j} f_j(uw^{-1}) \d m_K(u).
\end{align*}
Combining the last equation with \eqref{no_more_troubles_please} 
and \eqref{no_more_troubles_please_1}, we get that 
\begin{align*}
\frac{m_K(\{k\in K\colon \eta_{kH}(U)>0\}\cap\Ball_K(w,\delta_j))}
{m_K(\Ball_K(w,\delta_j))}
&\geq    
\frac{\int_{K_j} f_j(uw^{-1}) \d m_K(u)}{\int_K f_j(u) \d m_K(u)} 
\bigg(1-\frac{1}{2^j}\bigg) \\
&\geq 
\bigg(1- \frac{1}{j}\bigg) 
\bigg(1-\frac{1}{2^j}\bigg),
\end{align*}
where the least inequality is due to the fact that 
$w \in W_j$. Then, taking the limit as $j \to \infty$
we obtain that 
$$\lim_{j\to\infty}\frac{m_K(\{k\in K\colon 
\eta_{kH}(U)>0\}\cap \Ball_K(w,\delta_j))}
    {m_K(\Ball_K(w,\delta_j))}
    = 1$$
and this concludes the proof.
\end{proof}

We are now ready to prove \cref{lambda.supp}.

\begin{proof}[Proof of \cref{lambda.supp}]
Let $S=\{(x_1,x_2)\in X\times X\colon 
(x_1,x_2)\in\supp(\lambda_{(x_1,x_2)})\}$.
By \cref{lem_supp}, $S$ is a Borel subset of 
$X\times X$.
Consider a sequence $\delta_j\to0$ such that 
\cref{L} is satisfied, and let 
$L\subset X$ be the set of $x\in X$ that satisfy \eqref{L eq}. 
By \cref{L}, $\mu(L)=1$.
Following the argument in \cite[Proposition 3.11]{kmrr2} 
we will show that $\sigma(L\times L)=1$ and $L\times L\subset S$. 
Consequently, we will have that $\sigma(S)=1$, 
concluding the proof.
\par
We start by showing that $\sigma(L\times L)=1$. 
We write $L\times L = (L\times X)\cap(X\times
L)$ and so it is enough to show that both sets in this 
intersection have full measure $\sigma$.
By \cref{s_props} (i), we have that 
$\sigma(L\times X)=\pi_1\sigma(L)=\mu(L)=1$.
By \cref{s_props} (ii), the measure $\pi_2\sigma$
is absolutely continuous with respect to $\mu$, and since $\mu(L)=1$,
it follows that $\sigma(X\times L)=\pi_2\sigma(L)=1$.
\par
To conclude the proof, we show that $L\times L\subset S$.
Let $(x_1,x_2)\in L\times L$.  
To show that $(x_1,x_2)\in S$, it is enough to show that
for all open neighborhoods $U_1, U_2$ of $x_1, x_2$
we have 
$\lambda_{(x_1,x_2)}(U_1\times U_2)>0$.

Let $U_1,U_2$ be open neighborhoods of
$x_1,x_2$ respectively. 
By writing $\lambda_{(x_1,x_2)}= \int_K 
\eta_{kH}\times\eta_{kk_1^{-1} k_2 H} \d m_K(k)$, 
where $k_1,k_2\in K$ are such that $\pi(x_1)=k_1H$, $\pi(x_2)=k_2H$,
we see that it suffices to show that the set 
$W=W(k_1,k_2):=\{k\in K: \eta_{kH}(U_1)>0 \text{ and } 
\eta_{kk_1^{-1}k_2H }(U_2)>0\}$
has positive measure $m_K$, for some choice of 
the $k_1,k_2$ as above.
By \cref{L}, we can choose the elements $k_1, k_2 \in K$ such that
\begin{equation}\label{eq_ball_1}
\frac{m_K(\{k\in K:\eta_{kH}(U_1)>0\}\cap\Ball_K(k_1,\delta))}
{m_K(\Ball_K(k_1,\delta))}
\geq \frac{3}{4}
\end{equation}
and 
\begin{equation}\label{eq_ball_2}
\frac{m_K(\{k\in K:\eta_{kH}(U_2)>0\}\cap\Ball_K(k_2,\delta))}
{m_K(\Ball_K(k_2,\delta))}
\geq \frac{3}{4},
\end{equation}
for some $\delta>0$.
Now using \eqref{eq_ball_2} along with the bi-invariance 
of both $d_K$ and $m_K$, we have that
\begin{align}\label{eq_ball_2'}
& \frac{m_K(\{k \in K:\eta_{kk_1^{-1}k_2 H}
(U_2)>0\}\cap\Ball_K(k_1,\delta))}{m_K(\Ball_K(k_1,\delta))}
\notag \\
& \hspace*{4.5cm} 
=
\frac{m_K(\{k \in K:\eta_{kH}(U_2)>0\}\cdot 
k_2^{-1}k_1\cap\Ball_K(k_2,\delta)\cdot k_2^{-1}k_1)}
{m_K(\Ball_K(k_2,\delta)\cdot k_2^{-1}k_1)} \notag \\
& \hspace*{4.5cm} 
= \frac{m_K(\{k \in K:\eta_{kH}(U_2)>0\}\cap\Ball_K(k_2,\delta))}
{m_K(\Ball_K(k_2,\delta))} \notag \\
& \hspace*{4.5cm} 
\geq \frac{3}{4}.
\end{align}
Combining \eqref{eq_ball_1} and \eqref{eq_ball_2'} yields 
$\frac{m_K(W)}{m_K(\Ball_K (k_1, \delta))} \geq\tfrac{1}{2}$. 
This implies that $m_K(W) >0$ and concludes the proof.
\end{proof}

It remains to show \cref{lambda.gen}. 
To this end, we need the following lemma,
which is the analog of \cite[Lemma 3.7]{kmrr2} in our setting.

\begin{lemma}\label{s.l}
For $\sigma$-almost every $(x_1,x_2)\in X\times X$,
we have $\lambda_{(a,x_1)}=\lambda_{(x_1,x_2)}$.
\end{lemma}

\begin{proof}
By the definition of $\sigma$ and
the fact that $z\mapsto\eta_z$ is 
a disintegration, it follows that for $\sigma$-almost 
every $(x_1,x_2)$, we have
$\pi(x_1)=w\pi(a)$ and $\pi(x_2)=w\pi(x_1)$, for some $w\in K$.
For any such $(x_1,x_2)$, using the right invariance of $m_K$, we have 
\begin{align*}
    \lambda_{(x_1,x_2)} 
    & = \int_{X\times X} \eta_{k\pi(x_1)}\times\eta_{k\pi(x_2)}\d m_K(k)
    = \int_{X\times X} \eta_{kw\pi(a)}\times\eta_{kw\pi(x_1)}\d m_K(k) \\
    & = \int_{X\times X} \eta_{k\pi(a)}\times\eta_{k\pi(x_1)}\d m_K(kw^{-1})
    = \int_{X\times X} \eta_{k\pi(a)}\times\eta_{k\pi(x_1)}\d m_K(k) 
    = \lambda_{(a,x_1)}.  
\end{align*}
This concludes the proof.
\end{proof}

We are now ready to prove \cref{lambda.gen}.

\begin{proof}[Proof of \cref{lambda.gen}]
Consider the measure $\nu_a := \delta_a \times \mu$, where $\delta_a$ 
denotes the Dirac mass at $a$. \\
\textbf{Claim.}
There exists some \Folner{} sequence $\Psi$ such that 
$$\nu_a(\{(x_0,x_1)\in X\times X:
(x_0,x_1)\in\gen(\lambda_{(x_0,x_1)},\Psi)\})=1.$$

From the definition of $\nu_a$, it is clear that for $\nu_a$-almost 
every $(x_0, x_1) \in X \times X$, $x_0=a$. Therefore, assuming the Claim
we have that 
$$\nu_a(\{(a,x_1)\in X\times X:
(a,x_1)\in\gen(\lambda_{(a,x_1)},\Psi)\})=1,$$
which, by the definition of $\nu_a$ implies that 
$$\mu(\{x_1\in X: (a,x_1)\in\gen(\lambda_{(a,x_1)},\Psi)\})=1.$$
Then, by \cref{s_props} (i), we have that $\pi_1 \sigma = \mu$, 
and therefore 
$$\sigma(\{(x_1,x_2)\in X\times X:
(a,x_1)\in\gen(\lambda_{(a,x_1)},\Psi)\})=1.$$ 
From \cref{s.l}, we know that $\lambda_{(a,x_1)} = \lambda_{(x_1,x_2)}$
for $\sigma$-almost every $(x_1,x_2)$, and consequently, 
$$\sigma(\{(x_1,x_2)\in X\times
X:(a,x_1)\in\gen(\lambda_{(x_1,x_2)},\Psi)\})=1,$$
which was to be proved. Now, to finish the proof of 
\cref{lambda.gen}, it only remains to prove the Claim.

\begin{proof}[Proof of Claim]
In this proof we follow the argument used in the proof of 
\cite[Theorem 7.6]{kmrr1}.
Apply \cref{mx_gen} for the ergodic decomposition
$(x_0,x_1)\mapsto\lambda_{(x_0,x_1)}$, to obtain a 
\Folner{} sequence $\Phi$ such that
\begin{equation}\label{gen_mm}
    (\mu\times\mu)(\{(x_0,x_1)\in X\times X:
    (x_0,x_1)\in\gen(\lambda_{(x_0,x_1)},\Phi)\})=1.
\end{equation}

Consider the map $X \to M(X)$, $s \mapsto \nu_s =\delta_s \times \mu$, where 
$\delta_s $ denotes the Dirac mass at $s$.
It is quite straightforward to see that 
$s\mapsto\nu_s$ is a continuous disintegration of $\mu\times\mu$, 
and moreover, satisfies 
\begin{equation}\label{nu_prop}
    (T_g\times T_g)\nu_s = \nu_{T_gs}
\end{equation}
for any $s\in X$.
By \eqref{gen_mm} and the fact that $s\mapsto\nu_s$ is a disintegration 
of $\mu\times\mu$, it follows that for $\mu$-almost every $s\in X$,
$\nu_s$-almost every $(x_0,x_1)\in X\times X$ is in
$\gen(\lambda_{(x_0,x_1)},\Phi)$. Fix $b\in\supp(\mu)$. Then 
\begin{equation}\label{gen1}
    \nu_b\text{-almost every } (x_0,x_1) \text{ is in } 
    \gen(\lambda_{(x_0,x_1)},\Phi).
\end{equation}
By \cref{supp-gen lemma}, 
there exists some sequence $(g_n)_{n\in\N}$ in $G$ such that
$T_{g_n}a\to b$, and now by continuity of the disintegration 
$s\mapsto\nu_s$, combined with \eqref{nu_prop}, it follows that
\begin{equation}\label{nu_conv}
    (T_{g_n}\times T_{g_n})\nu_a\to\nu_b.
\end{equation}
Now, let $\F=(F_k)_{k\in\N}$ be a dense subset of $C(X\times X)$
and for $k,N\in\N$, consider the sets 
$$A_{k,N}=\bigg\{(x_0,x_1)\in X\times X: 
\max_{1\leq j\leq k}\bigg|\frac{1}{|\Phi_N|}\sum_{g\in\Phi_N}
F_j((T_g\times T_g)(x_0,x_1)) - 
\int_{X\times X}F_j\d\lambda_{(x_0,x_1)}\bigg|\leq\frac{1}{k}\bigg\}.$$
Using \eqref{gen1} and the monotone convergence theorem, 
it follows that for any $k\in\N$ there exists some $N(k)\in\N$, 
such that
\begin{equation}\label{nu_ineq}
    \nu_b(A_{k,N(k)})\geq 1-2^{-k}.
\end{equation}
For $k,N\in\N$, we define 
$$B_{k,N}=\bigg\{(x_0,x_1)\in X\times X: 
d_{X\times X}((x_0,x_1),A_{k,N})<\frac{1}{k}\bigg\},$$
where $d_{X\times X}$ is the metric on $X\times X$. 
Then for $k$ sufficiently large, we have
\begin{equation}\label{AB}
    \max_{1\leq j\leq k}
    \bigg|\frac{1}{|\Phi_{N(k)}|}\sum_{g\in\Phi_{N(k)}}F_j((T_g\times
    T_g)(x_0,x_1))-\int_{X\times X}F_j\d\lambda_{(x_0,x_1)}\bigg|
    \leq \frac{2}{k},
\end{equation}
for any $(x_0,x_1)\in B_{k,N(k)}$.
The sets $A_{k,N}$ are open, while the sets $B_{k,N}$ are closed 
subsets of $X\times X$, and also $A_{k,N(k)}\subset B_{k,N(k)}$, 
so by Urysohn's lemma, we can find, 
for all $k\in\N$, continuous functions $f_k:X\times X\to[0,1]$
such that 
$$f_k|_{A_{k,N(k)}}=1 \text{\quad and\quad} 
f_k|_{(X\times X)\setminus B_{k,N(k)}}=0.$$
By \eqref{gen1}, for each $k\in\N$, there exists $n(k)\in\N$ 
such that
$$\bigg|\int_{X\times X}(T_{g_{n(k)}}\times T_{g_{n(k)}})f_k\d\nu_a
- \int_{X\times X}f_k\d\nu_b\bigg| \leq 2^{-k}.$$
Let $(h_k)_{k\in\N}$ be the subsequence of $(g_n)_{n\in\N}$
defined by $h_k=g_{n(k)}$, for $k\in\N$. Then, by the equation above, we have
\begin{align*}
    \nu_a\big((T_{h_k}\times T_{h_k})^{-1}B_{k,N(k)}\big) &\geq
    \int_{X\times X}(T_{h_k}\times T_{h_k})f_k\d\nu_a 
    \geq \int_{X\times X}f_k\d\nu_b - 2^{-k} \\
    &\geq \nu_b(A_{k,N(k)}) - 2^{-k} 
    \geq 1-2^{-k+1} \text{\quad (by \eqref{nu_ineq})},
\end{align*}
for any $k\in\N$.
Therefore, it holds that
$$\sum_{k\geq1}\nu_a\big((T_{h_k}\times T_{h_k})^{-1}B_{k,N(k)}\big)
=\infty,$$
and then, by the Borel-Cantelli lemma, it follows that 
$\nu_a$-almost every $(x_0,x_1)\in\supp(\nu_a)$ belong to all, 
but finitely many, sets
$(T_{h_k}\times T_{h_k})^{-1}B_{k,N(k)}$. Then by $\eqref{AB}$,
it follows that for $\nu_a$-almost every $(x_0,x_1)\in X\times X$
and $k$ sufficiently large, we have
$$\max_{1\leq j\leq k}\bigg|\frac{1}{|\Psi_k|}\sum_{g\in\Psi_k}
F_j((T_g\times T_g)(x_0,x_1))-
\int_{X\times X} F_j\d\lambda_{(T_{h_k}\times T_{h_k})(x_0,x_1)}\bigg|
\leq \frac{2}{k},$$
where $\Psi$ is the \Folner{} sequence defined by 
$\Psi_k = \Phi_{N(k)}h_k$, for $k\in\N$.
Using \eqref{l.inv}, the above equation becomes 
$$\max_{1\leq j\leq k}\bigg|\frac{1}{|\Psi_k|}\sum_{g\in\Psi_k}
F_j((T_g\times T_g)(x_0,x_1))-
\int_{X\times X} F_j\d\lambda_{(x_0,x_1)}\bigg| 
\leq \frac{2}{k}.$$
Sending $k\to\infty$, we have shown that
$$\lim_{k\to\infty}\frac{1}{|\Psi_k|}\sum_{g\in\Psi_k}
F((T_g\times T_g)(x_0,x_1))=
\int_{X\times X} F\d\lambda_{(x_0,x_1)}$$ 
holds for $\nu_a$-almost every $(x_0,x_1)\in X\times X$ and 
for any $F\in\F$. An approximation argument concludes the proof 
of the Claim.
\renewcommand\qedsymbol{$\triangle$}
\end{proof}
The proof of the theorem is complete.
\end{proof}

\subsection{The proof of \cref{dr3}}

We are now ready to prove \cref{dr3}.

Let $G$ be a square absolutely continuous group, 
let $\xmt$ be an ergodic $G$-system admitting a continuous 
factor map to its Kronecker factor, and
let $a\in \gen(\mu,\Phi)$ for some \Folner{} sequence
$\Phi$. 
Let also $E$ be a clopen subset of $X$ with $\mu(E)>0$.  
Consider the measure $\sigma$ given in \eqref{meas-s}.
By Theorems \ref{lambda.supp} and \ref{lambda.gen}, we have that 
there is a \Folner{} sequence $\Psi$ such that 
\begin{equation}\label{eq_1_1}
\sigma(\{ (x_1, x_2) : 
(a,x_1)\in\gen(\lambda_{(x_1,x_2)},\Psi)
\text{ and }
(x_1,x_2)\in\supp(\lambda_{(x_1,x_2)})\})
= 1.
\end{equation}
On the other hand, by 
\cref{sigma.thm} we have that 
\begin{equation}\label{eq_1_2}
    \sigma \bigg( E \times \bigcup_{t\in G} T_t ^{-1} E\bigg) >0.
\end{equation}
Combining \eqref{eq_1_1} and \eqref{eq_1_2} we get that there exists 
$(x_1, x_2) \in X \times X$ such that for the 
$T \times T$-invariant measure $\lambda := \lambda_{(x_1, x_2)}$ we 
have that $(a,x_1)\in\gen(\lambda, \Psi)$, 
$(x_1,x_2)\in\supp(\lambda)$ and also 
$(x_1, x_2) \in  E \times \big( \bigcup_{t\in G} T_t ^{-1} E\big)$.
Hence, there is $t\in G$ such that 
$x_1 \in E$ and $x_2 \in T_t ^{-1} E$. 
Finally, applying 
\cref{supp-gen lemma} for 
$Y=X\times X$, $S=T\times T$, $y=(a,x_1)$ 
and $w=(x_1,x_2)$, $\nu = \lambda$ and for the \Folner{} 
sequence $\Psi$ we get that there is an infinite sequence 
$(g_n)_{n\in N}$ in $G$ such that 
$(T_{g_n} \times T_{g_n})(a,x_1) \to (x_1,x_2)$ in $X\times X$.
Therefore, $(a,x_1,x_2)\in X^3$ forms an \Erdos{} progression, 
and this concludes the proof of \cref{dr3}.

\section{Proof of the corollaries of \cref{mt1}}\label{corollaries proof section}

We start by showing \cref{tf fg nil are square absolutely continuous},
and then we will prove Corollaries \ref{tf fg nil cor} and \ref{fg nil cor}.
We split the proof of \cref{tf fg nil are square absolutely continuous}
into the following three lemmas:

\begin{lemma}\label{tf fg nil lemma 1}
Let $G$ be an amenable group, let $M>0$ and let $\phi:G\to G$ be a map such
that for every $g\in G$, $|\phi^{-1}(\{g\})|\leq M$.
Suppose that there exist two \Folner{} sequences $\Phi$ and $\Psi$ in $G$ 
and some $\eta>0$ such that for any $N\in\N$, we have that
$\phi(\Psi_N) \subset \Phi_N$ and 
$$\frac{|\phi(\Psi_N)|}{|\Phi_N|}\geq\eta.$$
Then $G$ is $\phi$-absolutely continuous.
\end{lemma}

\begin{lemma}\label{tf fg nil lemma 2}
Let $G$ be a torsion-free finitely generated nilpotent group. 
Then the squaring map $s_G$ on $G$ is injective.
\end{lemma}

\begin{lemma}\label{tf fg nil lemma 3}
Let $G$ be a torsion-free finitely generated nilpotent group. Then there exists some 
\Folner{} sequence $\Psi$ in $G$ and some $\eta>0$ such that for any $N\in\N$,
we have that $s_G(\Psi_N)\subset\Psi_{N+1}$ and
\begin{equation}\label{folner equation 1}
    \frac{|s_G(\Psi_N)|}{|\Psi_{N+1}|}=\eta.
\end{equation}
\end{lemma}

It is clear that \cref{tf fg nil are square absolutely continuous}
follows immediately from Lemmas \ref{tf fg nil lemma 1}, \ref{tf fg nil lemma 2}
and \ref{tf fg nil lemma 3}, where \cref{tf fg nil lemma 1} is applied for $\phi=s_G$,
$\Psi$ the \Folner{} guaranteed by \cref{tf fg nil lemma 3} and $\Phi$ the \Folner{}
given by $\Phi_N=\Psi_{N+1}$.

\begin{proof}[Proof of \cref{tf fg nil lemma 1}]
Let $G, M, \phi, \Phi, \Psi$ and $\eta$ be as in the assumptions of \cref{tf fg nil lemma 1}. 
We will prove something stronger than we require, namely
that for any $u: G \to [0,1]$ we have
\begin{equation*}
    \limsup_{N\to \infty} \frac{1}{|\Psi_N|} \sum_{g\in \Psi_N} u(\phi(g)) 
    \leq 
    \frac{M}{\eta} \limsup_{N\to \infty} \frac{1}{|\Phi_N|} \sum_{g\in \Phi_N} u(g),
\end{equation*}
and then for any $\epsilon >0$, taking $\delta = \eta\epsilon/M>0$ yields that 
$G$ is $\phi$-absolutely continuous.
Let $u: G \to [0,1]$, and let $(\Psi_{N_k})_{k\in \N}$ satisfy
$$\lim_{k\to \infty} \frac{1}{|\Psi_{N_k}|} \sum_{g \in \Psi_{N_k}} u(\phi(g))
=\limsup_{N\to \infty} \frac{1}{|\Psi_N|} \sum_{g\in \Psi_N} u(\phi(g)).$$ 
For each $k \in \N$, using the assumption on $\Phi$ and $\Psi$, 
we have that
$\phi(\Psi_{N_k})\subset \Phi_{N_k} $
and $\frac{|\phi(\Psi_{N_k})|}{|\Phi_{N_k}|} \geq \eta$. 
Then for every $k \in \N$ we have 
\begin{equation*}
\frac{1}{|\Psi_{N_k}|} \sum_{g \in \Psi_{N_k}} u(\phi(g)) \leq
\frac{M}{|\phi(\Psi_{N_k})|} \sum_{g \in \phi(\Psi_{N_k})} u(g) \leq
\frac{M}{\eta} \frac{1}{|\Phi_{N_k}|} \sum_{g \in \Phi_{N_k}} u(g),
\end{equation*}
so letting $k \to \infty$ we obtain that 
\begin{equation*}
\lim_{k\to \infty}
\frac{1}{|\Psi_{N_k}|} \sum_{g \in \Psi_{N_k}} u(\phi(g))
\leq 
\frac{M}{\eta} \limsup_{k\to \infty} 
\frac{1}{|\Phi_{N_k}|} \sum_{g \in \Phi_{N_k}} u(g)
\leq 
\frac{M}{\eta} \limsup_{N \to \infty} 
\frac{1}{|\Phi_{N}|} \sum_{g \in \Phi_{N}} u(g),
\end{equation*}
which concludes the proof.
\end{proof}

Let $G$ be a torsion-free finitely generated nilpotent group and fix a
Mal'cev coordinate system $(t_1,\dots,t_s)$. By \cite[Theorem 17.2.5]{Ka-Me-G}, 
we have that for any $1\leq i\leq s$, 
there are polynomials $p_i$
in $i-1$ variables, 
with rational coefficients, satisfying 
$p_i(\Z^{i-1})\subset\Z$
such that for any $x\in G$, we have
\begin{equation}\label{Mal'cev mult 1}
    t_i(x^2) = 2t_i(x) +  p_i(t_1(x), \dots, t_{i-1}(x)).
\end{equation}
In the previous, $p_1$ is a polynomial in $0$ variables, 
which means that it is a constant polynomial. 

\begin{proof}[Proof of \cref{tf fg nil lemma 2}]
The proof becomes obvious by using \eqref{Mal'cev mult 1}
and the injectivity of the coordinate 
map $G\to \Z^{s}$, $x \mapsto (t_1(x), \dots, t_s(x))$, which follows from the fact that $G$ is torsion-free.
\end{proof}

\begin{proof}[Proof of \cref{tf fg nil lemma 3}]
Let $G$ be a torsion-free finitely generated nilpotent group and fix 
a Mal'cev coordinate system $(t_1,\dots,t_s)$. We identify the group $G$ with $\Z^s$
as implied by the Mal'cev coordinate system. Let $(p_i)_{1\leq i\leq s}$ be 
the sequence of polynomials satisfying \eqref{Mal'cev mult 1} and for each $i$ we denote
the number of terms of $p_i$ by $\gamma_i$. Now we will define
three integer-valued sequences $(b_i)_{1\leq i\leq s}, (c_i)_{1\leq i\leq s}$ 
and $(d_i)_{1\leq i\leq s}$ that will help us to construct the \Folner{} sequence $\Psi$
with the required properties. We start with the latter one and we define it recursively as follows:
Let $d_1=1$. Now let $1<i\leq s$ and suppose that $d_j$ has been defined for all $1\leq j<i$.
Given a monomial $m(x_1,\dots,x_{i-1}) = x_1^{e_1}\dots x_{i-1}^{e_{i-1}}$,
with $e_j\geq 0$ for any $1\leq j<i$,
we let $d(m) = \sum_{1\leq j<i} e_jd_j$,
and then we also let 
$d(p_i) = \max\{d(m) \colon \text{the monomial}~m~\text{appears in}~p_i\}$.
Then we define $d_i = \max\{d(p_i),1\}$.
The other two sequences are defined as follows: Let $b_1=0$. Now we fix $1<i\leq s$.
We denote by $m_i$ the monomial $m$ appearing in $p_i$ with maximal coefficient in absolute value 
such that $d(p_i)=d(m)$. Then we define $b_i$ to be ceiling of the absolute value 
of the coefficient of $m_i$.
Finally, for any $1\leq i\leq s$, we let $c_i = \gamma_ib_i+2$.
We also note that $d_i=d(m_i)$ for each $1\leq i\leq s$.
\par
Let us now define the \Folner{} sequence $\Psi$.
By the definition of $(b_i)_{1\leq i\leq s}$ and $(d_i)_{1\leq i\leq s}$, we have that
for any $1\leq i\leq s$, if $|x_j|\leq M^{d_j}$ for every $1\leq j<i$, for some $M>0$,
then $|m_i(x_1,\dots,x_{i-1})|\leq b_iM^{d_i}$, and then,
$|p_i(x_1,\dots,x_{i-1})|\leq \gamma_ib_iM^{d_i}$. It follows that for any $1\leq i\leq s$,
the following implication holds:
\begin{equation}\label{p_i ineq}
    |x_j|\leq c_j^{Nd_j} \quad\forall\:1\leq j<i 
    \Longrightarrow
    |p_i(x_1,\dots,x_{i-1})|\leq \gamma_ib_ic_i^{Nd_i}.
\end{equation}
Given $M>0$, we use the notation $[M]:=(-M,M]\cap\Z$.
Now, for any $N\in\N$, we define
$$\Psi_N
= 
[c_1^{(N-1)d_1}]\times[c_2^{(N-1)d_2}]
\times\dots\times[c_s^{(N-1)d_s}].$$
It is not hard to check that $\Psi=(\Psi_N)_{N\in\N}$ is \Folner{} sequence in $G$.
We show that $s_G(\Psi_N)\subset\Psi_{N+1}$ 
for every $N\in\N$. Let $N\in\N$ and then
$$s_G(\Psi_N) = \big\{(2t_i(x) + p_i((t_j(x))_{1\leq j<i}))_{1\leq i\leq s} 
\colon t_i(x)\in[c_i^{(N-1)d_i}]
\:\:\:\forall\: 1\leq i\leq s\big\}.$$
Let $x^2=(t_1(x^2),\dots,t_s(x^2))\in s_G (\Psi_N)$.
We claim that for each $1\leq i\leq s$, we have that 
\begin{equation}\label{folner inclusion}
    t_i(x^2)\in[c_i^{Nd_i}].
\end{equation}
Let $1\leq i\leq s$. For any $1\leq j<i$, we have that 
$t_j(x)\in[c_j^{(N-1)d_j}]$, and then we have that
\begin{align*}
    & t_i(x^2) 
    = 2t_i(x) + p_i((t_j(x)\colon j<i)) \\
    & \in (p_i((t_j(x))_{1\leq j<i}) - 2c_i^{(N-1)d_i},
    p_i((t_j(x))_{1\leq j<i}) + 2c_i^{(N-1)d_i}]
    \cap(2\Z + p_i((t_j(x))_{1\leq j<i})) \\
    & \subset [c_i^{Nd_i}],
\end{align*}
where the last one follows from 
$|p_i((t_j(x))_{1\leq j<i})| \leq \gamma_ib_ic_i^{(N-1)d_i}$, by 
\eqref{p_i ineq}, and from the definition of $c_i$ along with that $d_i\geq1$.
This shows \eqref{folner inclusion} and hence that 
${s_G}(\Psi_N)\subset\Psi_{N+1}$.
\par
It remains to prove \eqref{folner equation 1} and we show it for the \Folner{} 
sequence $\Psi$ and with $\eta:=\prod_{1\leq i\leq s} c_i^{-d_i}>0$.
In other words, we prove that for any $N\in\N$ we have
\begin{equation}\label{folner equation 2}
    \frac{|{s_G}(\Psi_N)|}{|\Psi_{N+1}|} 
    = \prod_{1\leq i\leq s} c_i^{-d_i}.
\end{equation}
Let $N\in\N$.
By \cref{tf fg nil lemma 2}, $s_G$ is injective, and since $\Psi_N$ is finite,
we have that
$$|s_G(\Psi_N)|
= |\Psi_N|
= 2^s\bigg(\prod_{1\leq i\leq s}c_i^{d_i}\bigg)^{N-1}.$$
Moreover, we clearly have 
$$|\Psi_{N+1}|
= 2^s\bigg(\prod_{1\leq i\leq s}c_i^{d_i}\bigg)^N,$$
and so \eqref{folner equation 2} follows from comparing the last two equations.
This concludes the proof.
\end{proof}

Having established \cref{tf fg nil are square absolutely continuous},
we are ready to prove \cref{tf fg nil cor}.
Before we prove it, let us make some remarks. 

Let $s\in \N$ and $1\leq i\leq s$. We say that a non-empty set 
$L \subset \Z^s$ is a {\em line} in the $i$-th coordinate if there is a
non-empty 
interval $I$ in $\Z$, i.e., 
a set of the form $\{m, \dots, m+r \}$ for some $m,r\in \Z$ 
with $r \geq 0$, and some integers $x_j$, 
$j\in \{1, \dots, s\}\setminus \{i\}$ such that 
$$L = \{x_1\} \times \dots \times \{x_{i-1}\} \times I \times 
\{x_{i+1}\} \times \dots \times \{x_s\}.$$ 
We refer to the cardinality of the interval 
$I$ as the {\em{length}} of the line $L$.

Let $G$ be a torsion-free finitely generated nilpotent group and 
$(t_1, \dots, t_s)$ a Mal'cev coordinate system. For each 
$1\leq i\leq s$, let 
$e_i$ be the element with coordinates $(e^{(i)}_1, \dots, e^{(i)}_s)$, where 
$e^{(i)}_i= 1 $ and $e^{(i)}_j=0$ for $j \neq i$.
Every 
$g\in G$ has a unique representation as 
$(t_1(g), \dots, t_s(g))\in \Z^s$, and this defines a bijective map
from $G$ to $\Z^s$. From now on, we identify each $g\in G$ with its 
coordinates $(t_1(g), \dots, t_s(g))\in \Z^s$, and every set
$E\subset G$ with the corresponding set of coordinates 
in $\Z^s$. 
We freely pass from viewing a set $E\subset G$ as a subset of $\Z^s$ 
and vice versa, without stating it, as it will be clear from the context.

If $E$ is a finite subset of $G$ and $1\leq i\leq s$, 
then $E$ can be written as a 
finite disjoint union of lines in the $i$-th coordinate,
and this can be done in many ways. 
We want to write $E$ as a union of lines which is going to be 
maximal in some sense that is going to be useful for us in our proof 
of \cref{tf fg nil cor}.

More precisely, if $E$ is a finite subset of $G$, 
and $1\leq i\leq s$,
then we can always write it as a disjoint union $E=\bigsqcup_{j=1}^{\ell} L^{(j)}$
such that the sets $L^{(j)}$ are lines in the $i$-th coordinate, and for each
$j$, the line $L^{(j)}$ is 
maximal within $E$,
meaning that for each $1\leq j\leq s$, 
$e_i L^{(j)}$ is not 
a subset
of $E$. Note that although there always exists such a choice of lines,
it may not be unique, but uniqueness is not necessary for our purposes. 
We are now ready to prove \cref{tf fg nil cor}.

\begin{proof}[Proof of \cref{tf fg nil cor}]
The first part of \cref{tf fg nil cor} follows immediately by combining 
Theorems \ref{mt1} and \ref{tf fg nil are square absolutely continuous}.
It remains to prove that if $G$ is a torsion-free finitely generated nilpotent 
group, then given a Mal'cev coordinate system $(t_1,\dots, t_s)$ on $G$, 
we can choose $B$ so that for any finite set 
$C\subset \Z$ and any $1\leq i\leq s$, the set 
$\{ b \in B : t_i(b) \in C\}$ is finite. 

Let $G$ be a torsion-free finitely generated nilpotent group, let 
$A$ have positive left 
upper Banach density, and $\Psi=(\Psi_N)_{N\in\N}$
a left \Folner{} sequence such that 
$d_{\Psi}(A)>0$, where the previous density exists.
In addition, let $(t_1, \dots, t_s)$ be any Mal'cev coordinate system 
on $G$.\\
\textbf{Claim.} There is a \Folner{} sequence 
$\Psi '$ such that 
$d_{\Psi '}(A)>0$ and
for
any finite set 
$C\subset \Z$ and any $1\leq i\leq s$, the set 
$\{N\in \N : t_i (\Psi '_N) \cap C \neq \emptyset \} $
is finite.

Assume that we have proved the Claim. Then we can consider the set 
$A ':= \bigcup_{N\in \N} A \cap \Psi '_N $.
Then we have that 
$d_{\Psi '} (A) >0$, so by \cref{tf fg nil cor}, 
there is an infinite sequence
$B\subset A '\subset A$ and some $t\in G$ such that
$t \cdot B \ltrdot B \subset A ' \subset A$. 
Let $C\subset \Z$ be finite and $1\leq i\leq s$. 
For each $N\in \N$, $\Psi '_N$ is finite and 
$\{N \in \N: t_i(\Psi_N') \cap C \neq \emptyset\}$ is also finite, 
so since $B\subset\bigcup_{N\in \N} \Psi '_N$ is infinite,
one easily sees that $\{ b \in B : t_i(b) \in C\}$ is finite. 
So it remains to prove the Claim.

\begin{proof}[Proof of Claim]
\renewcommand\qedsymbol{$\triangle$}
We know that 
$d_{\Psi}(A) = \lim_{N\to \infty} \frac{|A\cap \Psi_N|}{|\Psi_N|}>0$. 
For each $1\leq i\leq s$ and $N \in \N$, let 
$$\delta_{i,N} = \frac{|\Psi_N \triangle (e_i \Psi_N)|}{|\Psi_N|}.$$
Since $\Psi$ is a \Folner{} sequence, for all $i$, $\delta_{i,N} \to 0$ 
as $N\to \infty$. Then we can choose a sequence 
$Q_N$ of natural numbers such that $Q_N \to \infty$ as 
$N\to\infty$ and for 
all $i$, $Q_N \delta_{i,N} \to 0$ as $N \to \infty$. 

\underline{Step $1$}:
For each $N\in \N$, $\Psi_N$ is a finite subset of $G$, so 
we can write it as a disjoint union of lines in the $1$st coordinate,
which are maximal within $\Psi_N$. 
Let $\Psi_N ^{(1)}$ be the union of those lines whose 
length is greater that $Q_N$, and $m_N$ be the number 
of those lines whose length is less than or 
equal to $Q_N$. Then 
\begin{equation*}
    Q_N \delta_{1,N} = Q_N \frac{|\Psi_N \triangle (e_1 \Psi_N)|}{|\Psi_N|}
    \geq 
    \frac{Q_N m_N}{|\Psi_N|}
    \geq 
    \frac{|\Psi_N| - |\Psi_N ^{(1)}|}{|\Psi_N|}.
\end{equation*}
Therefore, $\frac{|\Psi_N ^{(1)}|}{|\Psi_N|} \to 1$ as 
$N\to \infty$, from which one gets that 
$\Psi^{(1)}=(\Psi_N ^{(1)})_{N\in \N}$ is a 
left \Folner{} sequence in $G$. 
In addition, it is not difficult to see that 
the density $d_{\Psi^{(1)}}(A)$ exists and 
$d_{\Psi^{(1)}}(A) = d_{\Psi}(A)$.

Recall that for each $N\in \N$, $\Psi_N ^{(1)}$ is a disjoint union of 
some lines in the $1$-st coordinate $L^{(1, N)}, \dots, L^{(\ell_N, N)}$ 
whose length is greater than $Q_N$.
Let $N_1=1$ and set $\widetilde{\Psi}_1 ^{(1)} = \Psi_1 ^{(1)}$.
Since $\Psi_1 ^{(1)}$ is finite, the projection $P_1$ of 
$\Psi_1 ^{(1)}$ in the first coordinate is also finite. 
For each $N \in \N$, let 
$$\Psi_N ^{(1,2)} = \{g\in \Psi_N ^{(1)} : t_1(g) \notin P_1 \}=
\bigsqcup_{j=1} ^{\ell_N} 
\{g\in L^{(j, N)} : t_1(g) \notin P_1\}. $$
Then $\Psi_N ^{(1,2)} \subset \Psi_N ^{(1)}$ and 
\begin{align*}
    \frac{|\Psi_N ^{(1,2)}|}{|\Psi_N ^{(1)}|}
    &=
    \frac{\sum_{j=1}^{\ell_N} | \{g\in L^{(j, N)} : t_1(g) \notin P_1\} |}
    {\sum_{j=1}^{\ell_N} | L^{(j, N)}|}
    \geq 
    \frac{\sum_{j=1}^{\ell_N} (| L^{(j, N)}| - | P_1 | )}
    {\sum_{j=1}^{\ell_N} | L ^{(j, N)}|}  \\
    & = 1 - \frac{\sum_{j=1}^{\ell_N} | P_1 | }
    {\sum_{j=1}^{\ell_N} | L^{(j, N)}| }
    \geq 
    1 - \frac{ \ell_N | P_1 | }
    {\ell_N Q_N} = 
    1- \frac{|P_1|}{Q_N}.
\end{align*}
Since $Q_N \to \infty$ as $N \to \infty$, we have that 
$\frac{|P_1|}{Q_N} \to 0$ as $N \to \infty$, so 
$\frac{|\Psi_N ^{(1,2)}|}{|\Psi_N ^{(1)}|} \to 1$ as 
$N \to \infty$. Hence, we can pick $N_2 \in N, N_2 > N_1$ 
such that 
$\frac{|\Psi_{N_2} ^{(1,2)}|}{|\Psi_{N_2} ^{(1)}|}>\frac{1}{2}$.
Set $\widetilde{\Psi}^{(1)}_2 = \Psi_{N_2} ^{(1,2)}. $

Now since $\widetilde{\Psi}^{(1)}_2$ is finite, 
the projection $P_2$ of 
$\widetilde{\Psi}_2 ^{(1)}$ in the first coordinate is also finite. 
For each $N \in \N$, let 
$$\Psi_N ^{(1,3)} = \{g\in \Psi_N ^{(1)} : t_1(g) \notin P_1\cup P_2 \}=
\bigsqcup_{j=1}^{\ell_N } 
\{g\in L^{(j, N)} : t_1(g) \notin P_1 \cup P_2\}. $$
Then $\Psi_N ^{(1,3)} \subset \Psi_N ^{(1)}$ and as before, we have that
\begin{equation*}
    \frac{|\Psi_N ^{(1,3)}|}{|\Psi_N ^{(1)}|}
    \geq 
    1- \frac{|P_1|+|P_2|}{Q_N}.
\end{equation*}
Again, since $Q_N \to \infty$ as $N \to \infty$, we have that 
$\frac{|P_1| + |P_2|}{Q_N} \to 0$ as $N \to \infty$, so 
$\frac{|\Psi_N ^{(1,3)}|}{|\Psi_N ^{(1)}|} \to 1$ as 
$N \to \infty$. Hence, we can pick $N_3 \in N, N_3 > N_2$ 
such that 
$\frac{|\Psi_{N_3} ^{(1,3)}|}{|\Psi_{N_3} ^{(1)}|}>\frac{2}{3}$.
Set $\widetilde{\Psi}^{(1)}_3 = \Psi_{N_3} ^{(1,3)}. $

Continuing inductively, we find a strictly increasing sequence 
of natural numbers
$(N_k)_{k\in \N}$ such that for all $k\in \N$, 
$\widetilde{\Psi}^{(1)}_k \subset {\Psi}^{(1)}_{N_k} \subset 
{\Psi}_{N_k}$
and 
$\frac{|\widetilde{\Psi}^{(1)}_k |}{|\Psi_{N_k} ^{(1)}|}> 
1 - \frac{1}{k}$. 
Then $\frac{|\widetilde{\Psi}^{(1)}_k |}{|\Psi_{N_k} ^{(1)}|}
\to 1 $ as $k \to \infty$,
from which one gets that 
$\widetilde{\Psi}^{(1)}=(\widetilde{\Psi}_k ^{(1)})_{k\in \N}$ is a 
left \Folner{} sequence in $G$. 
In addition, it is not difficult to see that 
the density $d_{\widetilde{\Psi}^{(1)}}(A)$ exists and 
$d_{\widetilde{\Psi}^{(1)}}(A) = d_{({\Psi}^{(1)}_{N_k})_{k\in \N}}(A)
=d_{\Psi^{(1)}}(A)= d_{\Psi}(A)$.
In addition, from the construction of $\widetilde{\Psi}^{(1)}$, 
we have that for
any finite set 
$C\subset \Z$, the set 
$\{k\in \N : t_1 (\widetilde{\Psi}^{(1)}_k) \cap C \neq \emptyset \} $
is finite.

\underline{Step $2$}: Repeat Step 1 with 
$\widetilde{\Psi}^{(1)}$ in place of $\Psi$, which we write as a disjoint union of lines 
in the $2$nd coordinate that are maximal within 
$\widetilde{\Psi}_N ^{(1)}$,
to obtain a
strictly increasing sequence of natural numbers
$(N_k)_{k\in \N}$ and 
a left \Folner{} sequence 
$\widetilde{\Psi}^{(2)}$
such that for all $k\in \N$, 
$\widetilde{\Psi}^{(2)}_k \subset \widetilde{\Psi}^{(1)}_{N_k}$,
$\frac{|\widetilde{\Psi}^{(2)}_k |}{|\widetilde{\Psi}^{(1)}_{N_k}|}
\to 1 $ as $k \to \infty$ and such that for 
any finite set 
$C\subset \Z$, the set 
$\{k\in \N : t_2 (\widetilde{\Psi}^{(2)}_k) \cap C \neq \emptyset \} $
is finite. Then we will also have that the density
$d_{\widetilde{\Psi}^{(2)}}(A)$ exists and 
$d_{\widetilde{\Psi}^{(2)}}(A) = 
d_{\widetilde{\Psi}^{(1)}}(A) = d_{\Psi}(A)$.

Recall that 
$\widetilde{\Psi}^{(1)}$ has the property that 
for
any finite set 
$C\subset \Z$, the set 
$\{N\in \N : t_1 (\widetilde{\Psi}^{(1)}_N) \cap C \neq \emptyset \} $
is finite. As  
$\widetilde{\Psi}^{(2)}_k \subset \widetilde{\Psi}^{(1)}_{N_k}$
for all $k\in \N$, we get that for 
any finite set 
$C\subset \Z$, the set 
$\{k\in \N : t_1 (\widetilde{\Psi}^{(2)}_k) \cap C \neq \emptyset \} $
is finite. Hence, after all, for any finite set 
$C\subset \Z$ and any $i \in \{1, 2\}$, the set 
$\{k\in \N : t_i (\widetilde{\Psi}^{(2)}_k) \cap C \neq \emptyset \} $
is finite. 

Repeating the same procedure, after $s$ steps, we find a 
left \Folner{} sequence  
$\widetilde{\Psi}^{(s)}$ such that 
the density
$d_{\widetilde{\Psi}^{(s)}}(A)$ exists, 
$d_{\widetilde{\Psi}^{(s)}}(A) =  d_{\Psi}(A) > 0 $
and 
for 
any finite set 
$C\subset \Z$ and any $i \in \{1, \dots, s\}$, the set 
$\{N\in \N : t_i (\widetilde{\Psi}^{(s)}_N) \cap C \neq \emptyset \} $
is finite. Taking
$\Psi ' := \widetilde{\Psi}^{(s)}$, we see that 
$\Psi '$ satisfies the Claim, thus concludes its proof.
\end{proof}
Since the Claim is established, the proof of the corollary is complete.
\end{proof}

Now, we have to show \cref{fg nil cor}, but this is not hard using the fact
that finitely generated nilpotent groups are virtually torsion-free.
Let us first show the following simple lemma:

\begin{lemma}\label{restriction of Folner}
Let $G$ be an amenable group, $H$ be a subgroup of $G$ with 
$[G:H]=r<\infty$ and $\Phi$ be a \Folner{} sequence in $G$.
Then the following hold:
\begin{itemize}
    \item[(i)] $d_\Phi(H) = \frac{1}{r}$.
    \item[(ii)] Letting $\Psi_N=\Phi_N\cap H$ for each $N\in\N$,
    the sequence $\Psi = (\Psi_N)_{N\in\N}$ is a \Folner{} sequence in $H$.
\end{itemize}
\end{lemma}

In the following proof and onwards, we use the symbol
$\sqcup$ to denote the disjoint union.

\begin{proof}
(i) This is quite easy to check. \\
(ii) Let $h\in H$ and suppose for sake of contradiction that 
$$\ell = \liminf_{N\to\infty}\frac{|h\Psi_N\cap\Psi_N|}{|\Psi_N|} < 1.$$
For any $N\in\N$, we have that 
$$h\Phi_N\cap\Phi_N
= (h\Psi_N\cap\Psi_N)\sqcup((h\Phi_N\cap\Phi_N)\setminus H),$$
hence,
$$\frac{|h\Phi_N\cap\Phi_N|}{|\Phi_N|} 
= \frac{|h\Psi_N\cap\Psi_N|}{|\Phi_N|} + \frac{|(h\Phi_N\cap\Phi_N)\setminus H|}{|\Phi_N|}
\leq \frac{|h\Psi_N\cap\Psi_N|}{|\Psi_N|}\cdot\frac{|\Psi_N|}{|\Phi_N|}
+ \frac{|\Phi_N\setminus H|}{|\Phi_N|}.$$
Using (i), it follows that
$$\lim_{N\to\infty}\frac{|h\Phi_N\cap\Phi_N|}{|\Phi_N|}
\leq \frac{\ell}{r} + 1 - \frac{1}{r}
<1,$$
which contradicts the fact that $\Phi$ is a \Folner{} sequence in $G$.
Therefore, $\Psi$ is indeed a \Folner{} sequence in $H$
and the proof of the lemma is complete.
\end{proof}

\begin{proof}[Proof of \cref{fg nil cor}]
Let $G$ be a finitely generated virtually nilpotent group, $\Phi=(\Phi_N)_{N\in\N}$ 
be a \Folner{} sequence in $G$ and $A$ be a subset of $G$ 
with $\overline{d}_\Phi(A)>0$. By passing to a subsequence for which the limit
exists, we may assume that $d_\Phi(A)>0$.
Let $G'$ be a nilpotent finite-index subgroup of $G$. Since $G$ is 
finitely generated, by writing it as a disjoint union of finitely many left
cosets of $G'$, it is easy to see that $G'$ is also finitely generated,
and it is also nilpotent. By \cite[Theorem 17.2.2]{Ka-Me-G}, 
there exists a normal subgroup $H$ of $G'$ with finite index, 
which is torsion-free. Hence, $H$ is a torsion-free
finitely generated nilpotent group, which has finite index in $G$.
By writing $G$ as finite disjoint union of left cosets of $H$, 
we can see that there exists some $g\in G$ such that
$\overline{d}_\Phi(g^{-1}A\cap H) = \overline{d}_\Phi(A\cap gH) > 0$.
Again, by passing to a subsequence, we may assume that the limit exists,
so $d_\Phi(g^{-1}A\cap H) = d_\Phi(A\cap gH) > 0$.
We let $\Psi_N = \Phi_N\cap H$ for each $N\in\N$ and
by \cref{restriction of Folner} (ii), $\Psi=(\Psi_N)_{N\in\N}$ is a \Folner{}
sequence on $H$.
Hence for every $N\in\N$ we have that
$$\frac{|g^{-1}A\cap \Psi_N|}{|\Psi_N|}
= \frac{|g^{-1}A\cap \Phi_N\cap H|}{|\Phi_N|}\cdot
\frac{|\Phi_N|}{|\Psi_N|},$$
thus, 
$$d_\Psi(g^{-1}A\cap H) = \frac{d_\Phi(g^{-1}A\cap H)}{d_\Phi(H)}>0,$$
using \cref{restriction of Folner} (i).
Therefore, recalling that $H$ is torsion-free finitely generated nilpotent
group, \cref{tf fg nil cor} yields an infinite sequence
$B\subset g^{-1}A\cap H\subset g^{-1}A$ and some $t_0\in H$
such that $B\ltrdot B \subset t_0^{-1}g^{-1}A\cap H \subset t^{-1}A$,
where we have set $t=gt_0\in G$. This concludes the proof.
\end{proof}

Having established the corollaries concerning finitely generated 
nilpotent groups, we move on to showing the results corresponding to 
abelian groups, namely \cref{some abelian groups are sac}
and \cref{abelian groups with finite-index doubles}. Obviously, the latter
is an immediate consequence of the former and of \cref{mt1}, so it suffices to
show \cref{some abelian groups are sac}.

\begin{proof}[Proof of \cref{some abelian groups are sac}]
Let $(G,+)$ be an abelian group such that $2G$ has finite index 
in $G$, let $r=[G:2G]$, and consider a collection $\beta_{i}, 1\leq i \leq r$ 
such that 
$\beta_1 = e_G \:(=$ the identity element in $G)$ and 
$G=\bigsqcup_{i=1}^{r} \beta_i + 2G$.
Let $s$ denote the doubling map 
$s:G\to 2G, g\mapsto 2g$.
From the first isomorphism theorem for groups we know that the map 
$\widetilde{s}: \faktor{G}{ker(s)} \to 2G, \widetilde{s} (g+ ker(s))=2g$
is an isomorphism. Since $2G$ has finite index in $G$ and $G$
is infinite, we know that $2G$ is also infinite so the quotient
$\faktor{G}{ker(s)}$ is infinite. Let $(g_n)_{n\in \N}$ be a sequence
in $G$ such that $\faktor{G}{ker(s)}=\{g_n + ker(s) : n\in \N\}$
and such that for any $n \neq m, g_n + ker(s) \neq g_m + ker(s)$.

Consider a \Folner{} sequence $F=(F_N)_{N\in \N}$ in 
$\faktor{G}{ker(s)}$, and for each $N$, take $x_{N,1}, \dots, x_{N, 
\ell(N)}$ from the sequence $(g_n)_{n\in \N}$ such that 
$F_N=\{ x_{N,1}+ ker(s), \dots, x_{N,\ell(N)}+ ker(s)\} $ and
such that $x_{N,i}+ ker(s) \neq x_{N,j}+ ker(s)$ whenever $i \neq j$.
For each $N$, let 
$\widetilde{\Phi}_N= \widetilde{s}(F_N)=\{2x_{N,1}, \dots, 2x_{N,\ell(N)}\}$. 
Since $\widetilde{s}_G$ is an isomorphism, we have that $\widetilde{\Phi}= 
(\widetilde{\Phi}_N)_{N\in \N}$ is a \Folner{} in $2G$ and 
$|\widetilde{\Phi}_N|=|F_N|=\ell(N)$.

For each $N$ we define $\Phi_N:= \bigsqcup_{i=1}^{r} \beta_i + \widetilde{\Phi}_N$
and we observe that $|\Phi_N|=r|\widetilde{\Phi}_N|$.

\underline{Claim $1$}: $\Phi=(\Phi_N)_{N\in \N}$ is a \Folner{} in $G$.

\begin{proof}[Proof of Claim $1$]
Let $y\in G$ and $\epsilon>0$.
Then there exist $1\leq i_0\leq r$ and $h\in G$ such that $y=\beta_{i_0}+2h$. Then 
\begin{equation}\label{folner eq1}
    (y+\Phi_N)\cap\Phi_N
    = \bigg(\bigsqcup_{i=1}^r (\beta_{i_0}+\beta_i+2h +\widetilde{\Phi}_N))\bigg)
    \cap\bigg(\bigsqcup_{j=1}^r (\beta_j+ \widetilde{\Phi}_N)\bigg).
\end{equation}
Since the cosets $\beta_j+2G$ are disjoint, it follows 
that for each $1\leq i\leq r$,
there exists a unique $1\leq j(i)\leq r$ such that
$\beta_{i_0} + \beta_i \in \beta_{j(i)} + 2G$, so there is
$h_i \in G$ such that $\beta_{i_0} + \beta_i = \beta_{j(i)} + 2 h_i$.
In addition, if we assume that for $i_1 \neq i_2$ we have 
$j(i_1) = j(i_2)$, then we have that $\beta_{i_1}- \beta_{i_2} = 
2h_{i_1} - 2h_{i_2} \in 2G$, so $\beta_{i_1} + 2G = \beta_{i_2} + 2G$, 
which is a contradiction. Therefore, the map
$j: \{1, \dots, r\} \to \{1, \dots, r\}, i \mapsto j(i)$ is a bijection,
and then \eqref{folner eq1} becomes
\begin{equation*}
    (y+\Phi_N)\cap\Phi_N
    = \bigsqcup_{i=1}^r\big((\beta_{j(i)}+2h_i+2h+\widetilde{\Phi}_N))
    \cap(\beta_{j(i)}+\widetilde{\Phi}_N)\big)
    = \bigsqcup_{i=1}^r \big(\beta_{j(i)} + ((2h_i+2h+\widetilde{\Phi}_N)
    \cap \widetilde{\Phi}_N)),
\end{equation*}
thus, we have that
\begin{equation}\label{folner eq2}
    |(y+\Phi_N)\cap\Phi_N|
    = \sum_{i=1}^r |(2h_i+2h+\widetilde{\Phi}_N)
    \cap \widetilde{\Phi}_N)|.
\end{equation}
Now since $\widetilde{\Phi}$ is a \Folner{} sequence in $2G$, for $N$ sufficiently 
large we have that for every $1\leq i\leq r$,
$|(2h_i+2h+\widetilde{\Phi}_N)\triangle\widetilde{\Phi}_N|\leq
\epsilon|\widetilde{\Phi}_N|$,
and then we have that
\begin{equation}\label{folner eq3}
\frac{|(2h_i+2h+\widetilde{\Phi}_N) \cap \widetilde{\Phi}_N)|}{|\widetilde{\Phi}_N|}
\geq 
1 - \frac{|(2h_i+2h+\widetilde{\Phi}_N)\triangle\widetilde{\Phi}_N|}
{|\widetilde{\Phi}_N|}
\geq 1 - \epsilon.
\end{equation}

Then, combining \eqref{folner eq2} and \eqref{folner eq3} we get 
that for $N$ sufficiently large
\begin{equation*}
    \frac{|(y+\Phi_N)\cap\Phi_N|}{|\Phi_N|} = 
    \frac{\sum_{i=1}^r |(2h_i+2h+\widetilde{\Phi}_N)
    \cap \widetilde{\Phi}_N)|}{r|\widetilde{\Phi}_N|}\geq 
    \frac{\sum_{i=1}^r (1-\epsilon)|\widetilde{\Phi}_N|}{r|\widetilde{\Phi}_N|}
    = 1-\epsilon.
\end{equation*}
Since $\epsilon>0$ was arbitrary, it follows that 
$\lim_{N\to\infty}\frac{|(y+\Phi_N)\cap\Phi_N|}{|\Phi_N|}=1$.
Thus, $\Phi$ is a \Folner{} sequence in $G$, and the proof of the claim
is complete.
\renewcommand\qedsymbol{$\triangle$}
\end{proof}

Assume that $ker(s)$ is infinite, and let $(h_n)_{n\in \N}$ be an 
enumeration of $ker(s)$ (the case when $ker(s)$ is finite is easier 
and can be treated with a similar argument).
Let also $(E_k)_{k\in \N}$ be a \Folner{} sequence in $ker(s)$
(in case $ker(s)$ is finite, one can take 
$E_k = ker(s)$ for all $k\in \N$). Fix $N \in \N$. Then for any $n,m$ with  
$1\leq n,m \leq N$ and any $i,j\in \{1, \dots, \ell(N)\}$, if 
$g_n+x_{N,i}-x_{N,j}+z_m \in ker(s)$, then
$$\lim_{k\to \infty} \frac{|(g_n+x_{N,i}-x_{N,j}+z_m+E_k) \triangle E_k|}{|E_k|}=0.$$ 
We can then pick $k_N$ large enough such that for all $n,m$ with 
$1\leq n,m \leq N$ and all $i,j\in \{1, \dots, \ell(N)\}$, if 
$g_n+x_{N,i}-x_{N,j}+z_m \in ker(s)$, then

\begin{equation}\label{useful}
\frac{|(g_n+x_{N,i}-x_{N,j}+z_m+H_N) \triangle H_N|}
{|H_N|}<\frac{1}{N},
\end{equation}
where we have set $H_N := E_{k_N}$.
For every $N\in \N$, let $\Psi_N=\bigsqcup_{i=1}^{\ell(N)}
x_{N,i} + H_N$.

\underline{Claim $2$}: $\Psi=(\Psi_N)_{N \in \N}$ is a \Folner{} in 
$G$.

\begin{proof}[Proof of Claim $2$]
Let $y\in G$, and take $n_y, m_y \in \N$ such that 
$y=g_{n_y} + z_{m_y}$. Let $\epsilon >0$. Since $F$ is a \Folner{} 
in $G$, there is $N_0\in \N$, $N_0 > n_y, m_y$, such that for every $N \geq N_0$, 
\begin{equation*}
    \frac{|(g_{n_y} + ker(s) +F_N) \triangle F_N |}{|F_N|} < \frac{\epsilon}{2}
\end{equation*}
and $\frac{1}{N}< \frac{\epsilon}{2}$. Let $N\geq N_0$. For each $i\in \{1, \dots, \ell(N)\}$, if there is $j\in \{1, \dots, \ell(N)\}$ such that 
$g_{n_y} + x_{N,i} \in x_{N,j} + ker(s)$, then this $j$ is unique by the choice of $(x_{N,i})_{i\in \{1, \dots, \ell(N)\}}$. Let 
$$A_N= \{ i \in \{1, \dots, \ell(N)\} : g_{n_y} + x_{N,i} \in x_{N,j} + ker(s) \text{ for some } j\in \{1, \dots, \ell(N)\}\},$$
and for each $i\in A_N$ denote the corresponding unique $j$ by $j(i)$. Let also
$$B_N:= \{ i \in \{1, \dots, \ell(N)\} : g_{n_y} + x_{N,i} \not \in x_{N,j} + ker(s) \text{ for any } j\in \{1, \dots, \ell(N)\}\}$$
and 
$$C_N:= \{ j \in \{1, \dots, \ell(N)\} : g_{n_y} + x_{N,i} \not \in x_{N,j} + ker(s) \text{ for any } i\in \{1, \dots, \ell(N)\}\},$$
and observe that $\{ g_{n_y} + x_{N,i} : i \in B_N\} = \{a\in G : a + ker(s) \in (g_{n_y} + ker(s) +F_N) \setminus F_N \}  $, 
$\{ x_{N,j} : j \in C_N\} = \{a\in G : a + ker(s) \in F_N \setminus (g_{n_y} + ker(s) +F_N)  \}  $, $|B_N|=|(g_{n_y} + ker(s) +F_N) \setminus F_N|$ and 
$|C_N|= |F_N \setminus (g_{n_y} + ker(s) +F_N)|$.
We then obtain that
\begin{align*}
    (y+\Psi_N) \triangle \Psi_N &= 
    \Big( \bigsqcup_{i=1}^{\ell(N)} g_{n_y} + x_{N,i} + z_{m_y} +H_N \Big) \triangle 
    \Big( \bigsqcup_{i=1}^{\ell(N)} x_{N,i} + H_N \Big)\\
     &\subset \Big( \bigcup_{i\in B_N} g_{n_y} + x_{N,i} + z_{m_y} +H_N \Big)
     \cup \Big( \bigcup_{i \in C_N} x_{N,i} + H_N \Big)
     \cup \\
     &\:\:\:\:\:\: \cup \Big( \bigcup_{i \in A_N} (g_{n_y} + x_{N,i} + z_{m_y} +H_N) \triangle (x_{N,j(i)}  +H_N)\Big).
\end{align*}
Therefore we have that 
\begin{align*}
    &|(y+\Psi_N) \triangle \Psi_N| \leq |B_N||H_N| + |C_N||H_N| + \sum_{i\in A_N} 
    |(g_{n_y} + x_{N,i} + z_{m_y} +H_N) \triangle (x_{N,j(i)}  +H_N)|\\
    &= |(g_{n_y} + ker(s) +F_N) \triangle F_N || |H_N| + \sum_{i\in A_N} 
    |(g_{n_y} + x_{N,i} + z_{m_y} +H_N) \triangle (x_{N,j(i)}  +H_N)|.
\end{align*}

Now for each $i\in A_N$, we have that $g_{n_y} + x_{N,i}  \in x_{N,j(i)} + ker(s)$, so also
$g_{n_y} + x_{N,i} + z_{m_y} \in x_{N,j(i)} + ker(s)$, and hence there is $w_y \in ker(s)$ such that 
$g_{n_y} + x_{N,i} + z_{m_y} = x_{N,j(i)} + w_y$. This implies 
$|(g_{n_y} + x_{N,i} + z_{m_y} +H_N) \triangle (x_{N,j(i)}  +H_N)| =
|x_{N,j(i)} + w_y +H_N) \triangle (x_{N,j(i)}  +H_N)| = \\|(w_y +H_N) \triangle H_N|.$ 
Since $w_y = g_{n_y} + x_{N,i} - x_{N,j(i)} + z_{m_y}  $, using \eqref{useful} and combining with the fact that 
$N> n_y, m_y$ we get that 
$\frac{|(w_y + H_N ) \triangle H_N|}{|H_N|}\leq \frac{1}{N}$.

Recall that $|\Psi_N|=\ell(N)|H_N|=|F_N||H_N|$ and $|A_N|\leq \ell(N)=|F_N|$, so  after all we have 
\begin{align*}
    \frac{|(y+ \Psi_N)\triangle \Psi_N|}{|\Psi_N|} &\leq 
    \frac{|(g_{n_y} + ker(s) +F_N) \triangle F_N | |H_N| + \sum_{i\in A_N} \frac{|H_N|}{N}}{|F_N| |H_N|} \\
    &\leq  \frac{|(g_{n_y} + ker(s) +F_N) \triangle F_N |}{|F_N|} +
    \frac{1}{N} <\epsilon.
\end{align*}
Since $\epsilon>0$ and $y\in G$ were arbitrary
this concludes the proof of the claim.
\renewcommand\qedsymbol{$\triangle$}
\end{proof}

We will now prove that $\Phi, \Psi$ satisfy \cref{sac_def} for the map $s$ and therefore $G$ 
is \textit{square absolutely continuous}.
Observe that $2\Psi_N = \{2g : g \in \Psi_N\} = 
\{2x_{N,i} : i \in \{1, \dots, \ell(N)\}\}= \widetilde{\Phi}_N $ and if 
$u:G \to [0,1]$, then for each $N\in \N$ we have 
$\sum_{g\in \Psi_N} u(2g) = \sum_{g\in \bigsqcup_{i=1}^{\ell(N)}
x_{N,i} + H_N } u(2g) = 
\sum_{i=1}^{\ell(N)}  |H_N| u(2x_{N,i})= |H_N| 
\sum_{g\in \widetilde{\Phi}_N} u(g) $. Therefore, for each $N\in \N$ we have 
\begin{align*}
    \frac{1}{|\Psi_N|} \sum_{g\in \Psi_N} u(2g) =
    \frac{1}{|\widetilde{\Phi}_N|} \sum_{g\in \widetilde{\Phi}_N} u(g)
    =\frac{r}{|\Phi_N|} \sum_{g\in \widetilde{\Phi}_N} u(g) 
    \leq \frac{r}{|\Phi_N|} \sum_{g\in \Phi_N} u(g).
\end{align*}
Therefore, given $\epsilon>0$ we can take $\delta=\frac{\epsilon}{r}$ and then
for any $u:G \to [0,1]$, if 
$$\limsup_{N\to \infty}  \frac{1}{|\Phi_N|} \sum_{g\in \Phi_N} u(g)<\delta = 
\frac{\epsilon}{r},$$
then 
$$\limsup_{N\to \infty}  \frac{1}{|\Psi_N|} \sum_{g\in \Psi_N} u(2g)< \epsilon.$$
This concludes the proof of \cref{some abelian groups are sac}.
\end{proof}

\section{Counterexamples on product sets}\label{counterexamples section}

To construct the counterexamples we introduce some convenient notation. 
We denote by $H_3$ the $3\times 3$ discrete Heisenberg group, that is
the group of $3\times 3$ upper triangular matrices with $1$ in the diagonal 
and integer entries. Using the obvious Mal'cev coordinate system in this group
we identify $H_3$ with $\Z^3$ by writing elements of $H_3$ as $a=(a_1,a_2,a_3)$
where the group operation is given by $ab = (a_1+b_1,a_2+b_2,a_3+b_3+a_1b_2)$.
For $N\in\N$, we denote $[N]:=[1,N]$, where all intervals are considered in $\Z$,
and $[N]':=[1,N]\cap(2\Z+1)$.
All the counterexamples below are constructed on the group $G=H_3$.
We may assume that any infinite sequence $B\subset H_3$ arising 
from \cref{tf fg nil cor} considered below is infinite an all coordinates.

\begin{example}\label{counterexample 1}
We construct a set $A\subset H_3$ and a \Folner{} sequence $\Phi$
with $\overline{d}_\Phi(A)>0$ such that there is no infinite sequence 
$B=(b(n))_{n\in\N}\subset A$ satisfying $B\ltrdot B\subset At^{-1}$ 
for some $t\in G$.
\par
Consider the \Folner{} sequence $\Phi=(\Phi_N)_{N\in\N}$ with 
$\Phi_N = (2^N + [N]) \times [N] \times [N^2]$.
It is not hard to see that $\Phi$ is a left \Folner{} sequence, 
but not a right one.
Let $\Phi_N' = (2^N + [N]')\times[N]'\times[N^2]'$ and consider the set 
$A = \bigcup_{N\in\N} \Phi_N'$. Clearly, $\overline{d}_\Phi(A)>0$. Suppose, 
for sake of 
contradiction, that there exist an infinite sequence 
$B = (b(n))_{n\in\N}\subset A$ 
and some $t\in H_3$ such that 
$\{b(i)b(j)\colon i<j\}\subset At^{-1}$. We denote $t^{-1}=(t_1,t_2,t_3)$.
We observe that $B\cap \Phi_N'\neq \emptyset$ for infinitely many $N\in\N$.
Let $b=b(i)=(b_1,b_2,b_3)$, for some $i\in\N$, such that $b\in\Phi_N'$ 
for some $N$ large 
(compared to the $t_i$'s). Then we can find 
some $j>i$ such that the element $c=b(j)=(c_1,c_2,c_3)$
belongs in some $\Phi_M'$ for some $M$ much larger than $N$, 
and then $bc\in At^{-1}$.
It follows that $bc\in \Phi_Q't^{-1}$, for some $Q\in\N$, where 
$$\Phi_Q't^{-1} = (2^Q+[Q]'+t_1)\times([Q]'+t_2)\times([Q^2]'+t_3+t_2(2^Q+[Q]')).$$
Then $2^Q - Q +t_1 \leq (bc)_1 \leq 2^Q + Q +t_1 $
and on the other hand, $(bc)_1 = b_1 + c_1 \in 2^N + 2^M + [N+M]'$, so  
$2^N + 2^M - N - M \leq (bc)_1 \leq 2^N + 2^M + N+M $. Combining those with 
the fact that $M$ is assumed to be sufficiently large (and much larger than
$N$) we obtain that $Q=M$.
We now want to show that $t_2\neq 0$. 
By the fact that $bc\in \Phi_Q't^{-1}$,
we have that 
$$bc \in (2\Z+1+t_1)\times (2\Z+1+t_2)\times (2\Z+1+t_3+t_2).$$
On the other hand, multiplying $b$ and $c$,
and using that $b_i$ and $c_i$ are odd for all $i$, gives that
$$bc = (b_1+c_1, b_2+c_2, b_3+c_3+b_1c_2) \in 2\Z \times 2\Z \times (2\Z+1).$$
It follows that all the $t_i$'s are odd, and in particular, $t_2\neq 0$.

Now, since $bc \in\Phi_Q't^{-1}$ and 
$t_2\neq 0$, we have that
$$(bc)_3 \gg 2^Q = 2^M.$$
On the other hand, for $M$ sufficiently large, we have that 
$$(bc)_3  = b_3+c_3+b_1c_2 \leq N^2 + M^2 + (2^N+N)M \ll M^2,$$
which yields a contradiction.
\end{example}

\begin{example}\label{counterexample 2}
We construct a set $A\subset H_3$ and a \Folner{} sequence $\Phi$ with $\overline{d}_\Phi(A)=1$
such that there is no infinite sequence $B=(b(n))_{n\in\N}\subset G$ satisfying $B\rtrdot B\subset t^{-1}A$ for some $t\in G$.
\par
Consider the same \Folner{} sequence $\Phi$ as in \cref{counterexample 1}. 
We observe that $\Phi_N\cap\Phi_M=\emptyset$ for any $N\neq M$,
and in particular, the projections of any two such sets in the first coordinate are
disjoint subsets of $\Z$.
We define the set $A = \bigcup_{N\in\N}\Phi_N$ and clearly we have $\overline{d}_\Phi(A)=1$.
Suppose, for sake of contradiction, that there exist an infinite sequence
$B=(b(n))_{n\in\N}\subset H_3$ and some $t\in H_3$ such that 
$\{b(i)b(j)\colon i>j\}\subset t^{-1}A$. We denote $t=(t_1,t_2,t_3)$ and 
$b(1)=b=(b_1,b_2,b_3)$. We may assume 
without loss of generality that $b_2\neq0$. We let $B'=(b(n))_{n\geq2}$.
Moreover, we denote $b^{-1}=y=(y_1,y_2,y_3)$ and then
we have that $B'\subset t^{-1}Ay$. It follows that
$B'\cap t^{-1}\Phi_Ny \neq \emptyset$ for infinitely many $N\in\N$.
Fix $c=b(i)$ for some $i>1$ such that $c\in B'\cap t^{-1}\Phi_My$ for some large $M\in\N$.
Then $cb\in t^{-1}A$, which implies that $cb$ belongs in exactly one set of the form $t^{-1}\Phi_N$.
We have that 
$$(cb)_1 = c_1+b_1 \in 2^M+[M]+t_1+y_1+b_1 = 2^M+[M]+t_1,$$
hence $cb\in t^{-1}\Phi_M$. Then we have that 
$(cb)_3 \in [M^2]+t_1[M]+t_3$, which imples that $(cb)_3\ll M^2$, for $M$ sufficiently large.
On the other hand, multiplying $c$ and $b$ gives that
$$(cb)_3 = c_3 + b_3 + c_1b_2 \gg c_1 \in 2^M + [M] + t_1 +y_1,$$
which implies that $(cb)_3 \gg 2^M = 2^M$, for $M$ sufficiently large, where the implied constant
is again absolute. This yields a contradiction.
\end{example}

\begin{example}\label{counterexample 3}
Consider the same $\Phi$ and $A\subset H_3$ as in \cref{counterexample 2}. Then we show that
there is no infinite sequence $B\subset H_3$ satisfying $B\rtrdot B\subset At^{-1}$.
\par
To see why, suppose, for sake of contradiction, that there exists such a 
sequence $B$ and,
as we did in \cref{right shift}, consider the sequence $B' = t^{-1}Bt$. Then
we have that $B'\rtrdot B'\subset t^{-1}A$, which cannot hold for this particular set $A$
as we saw in \cref{counterexample 2}. This yields a contradiction.
\end{example}

\appendix

\section{A result of Host and Kra for amenable groups}\label{appendix.a}

The purpose of this appendix is to prove \cref{HK_generic}.
The proof follows the ideas in the proof of 
\cite[Proof of Proposition 6.1]{HK09}, adapted in our setting.
We state the following classical result
(see for example \cite[Example 11.13 (a)]{rudin_functional}),
which we will need in the proof of \cref{HK_generic}:

\begin{lemma}
\label{lmf}
Let $X$ be a compact metric space. Then the only linear multiplicative functionals 
on the algebra $C(X)$ are the point evaluations,
i.e., $\ev_x(f)=f(x)$, for $x\in X$.
\end{lemma}

For convenience, we restate the lemma we want to prove:

\begin{named}{\cref{HK_generic}}{}\cite[Proposition 6.1 for
group actions]{HK09}
Let $G$ be an amenable group,
let $\xmt$ be an ergodic $G$-system, 
$\zmr$ be its Kronecker factor and 
$\rho:\xmt\to\zmr$ be a factor map. If $a\in X$ is a transitive
point, then there exists a point $z\in Z$ and a \Folner{} 
sequence $\Psi$ such that
\begin{equation}\label{eq_ld1_1}
    \lim_{N\to\infty}\frac{1}{|\Psi_N|}
\sum_{g\in\Psi_N}f_1(T_ga)\cdot f_2(R_gz)
=\int_{X}f_1\cdot(f_2\circ\rho)\d\mu
\end{equation}
holds for any $f_1\in C(X)$ and $f_2\in C(Z)$.
\par
We remark that the result still holds if we 
replace $\zmr$ by any factor of 
$\xmt$ that is distal as a topological system.
\end{named}

\begin{proof}As in \cite{HK09}, we split 
the proof into two parts.
\par
\underline{Construction of a common extension}. 
Let $\Aa\subset \{f:X\to\C\colon f~\text{is measurable and bounded}\}$ 
be the closed (in norm) subalgebra that is spanned by $C(X)$ and 
$\{f\circ\rho:f\in C(Z)\}$. This is a unital commutative separable algebra, 
which contains the constants and is invariant under both complex conjugation and $T$.
Consider the Gelfand spectrum of $\Aa$, which is defined as 
$$W = \{\chi:\Aa\to\C\colon \chi~\text{is linear and multiplicative}\}.$$

Note that $W$ is compact and metrizable, since $\Aa$ is separable. 
By Gelfand's theorem, there exists an isometric isomorphism $F:C(W)\to\Aa$, 
satisfying $F^{-1}(f)(\chi)= \chi(f)$ for all 
$f\in \Aa$ and all $\chi \in W$. Hence for all $\widetilde{f}\in C(W)$ and all 
$\chi\in W$, we have $\widetilde{f}(\chi) = \chi(F(\widetilde{f}))$.
For $g\in G$ and $\chi \in W$, we define $S_g(\chi): A \to \C$, 
$S_g(\chi)(f) = \chi(f\circ T_g)$. Then it is not too difficult to see that 
for each $g\in G$, $S_g: W \to W$ is a homeomorphism, and we also
let $S=(S_g)_{g\in G}$.
Then for every $\chi\in W$, we have that
$$
\chi(F(\widetilde{f} \circ S_g))= 
\widetilde{f}(S_g(\chi)) = 
S_g(\chi)(F(\widetilde{f}))=
\chi(F(\widetilde{f})\circ T_g)
$$
for $\widetilde{f}\in C(W)$ and any $g\in G$.
In particular, for every $x\in X$, by considering the evaluation functional
$\ev_x\in W$, it follows that
$$F(\widetilde{f}\circ S_g)(x) 
= \ev_x(F(\widetilde{f}\circ S_g))
= \ev_x(F(\widetilde{f})\circ T_g)
= (F(\widetilde{f})\circ T_g)(x)$$
holds for any $\widetilde{f}\in C(W)$ and any $g\in G$.
Thus, we have that
\begin{equation}\label{appb_eq.3}
    F(\widetilde{f}\circ S_g)=F(\widetilde{f})\circ T_g
\end{equation}
for every $\widetilde{f}\in C(W)$ and for every $g\in G$.

Now, we consider the embedding $F^{-1}|_{C(X)}:C(X)\hookrightarrow C(W)$. Given $w\in W$,
$\ev_w\circ F^{-1}|_{C(X)}$ is a linear multiplicative functional on $C(X)$, and by \cref{lmf}, 
there exists a unique $x\in X$ such that $\ev_w\circ F^{-1}|_{C(X)} = \ev_x$.
Thus, we can define $\pi_X:W\to X$ by $\pi_X(w) = x$ if and only if 
$\ev_w\circ F^{-1}|_{C(X)} = \ev_x$.
The last equation is equivalent to that for any $f\in C(X)$ and $w\in W$,
$$f\circ\pi_X(w) = \ev_{\pi_X(w)}(f) = \ev_w\circ F^{-1}|_{C(X)}(f)= F^{-1}(f)(w).$$
Hence $\pi_X$ is the unique map from $W$ to $X$ satisfying
\begin{equation}\label{appb_eq.4}
    f\circ\pi_X=F^{-1}(f)
\end{equation}
for any $f\in C(X)$.
We claim that $\pi_X$ is continuous and surjective.
\par
To show continuity, we let $(w_n)_{n\in\N}$ in $W$ such that
$w_n\to w\in W$. Then for any $f\in C(W)$, we have $\ev_{w_n}(f)\to \ev_w(f)$.
Hence, for any $f\in C(X)$, we have
$\ev_{\pi_X(w_n)}(f) \to \ev_{\pi_X(w)}(f)$, that is, 
$f(\pi_X(w_n)) \to f(\pi_X(w))$. Since $X$ is compact, $\pi_X(w_n)$ has 
a convergent subsequence, which by abuse of notation we denote by $\pi_X(w_n)$.
Suppose for sake of contradiction that $\pi_X(w_n)\to y\neq \pi_X(w)$.
Then by Urysohn's lemma, we can find some $f\in C(X)$ and some disjoint
open neighborhoods $U_1\ni \pi_X(w), U_2\ni y$, such that $f|_{U_1}=1$ and $f|_{U_2}=0$.
It follows that $\pi_X(w_n)\in U_2$ for large $n$, hence 
$f(\pi_X(w_n)) = 0$ for large $n$, while $f(\pi_X(w))=1$, but this contradicts 
the fact that $f(\pi_X(w_n))\to f(\pi_X(w))$. This shows that every convergent subsequence of
$\pi_X(w_n)$ converges to $\pi_X(w)$, and since $X$ is compact,
it follows that $\pi_X(w_n)\to \pi_X(w)$, showing the continuity of $\pi_X$.
\par
To show that $\pi_X$ is surjective, let $x\in X$ and consider the linear multiplicative functional
$\ev_x$ on $C(X)$. Since $F^{-1}|_{C(X)}$ is an embedding, it follows that
$F^{-1}|_{C(X)}\circ\ev_x$ is a linear multiplicative functional on $C(W)$. Then,
by \cref{lmf}, there exists some $w\in W$ such that $F^{-1}|_{C(X)}\circ\ev_x=\ev_w$.
Hence, $\pi_X(w)=x$, showing that $\pi_X$ is surjective.
\par
Moreover, for any $g\in G$ and any $f\in C(X)$, let
$\widetilde{f}=F^{-1}(f)\in C(W)$, and then
$F(\widetilde{f}) = f\in C(X)$ and by \eqref{appb_eq.3}, 
$F(\widetilde{f}\circ S_g) = f\circ T_g\in C(X)$. 
Then, using \eqref{appb_eq.3} and \eqref{appb_eq.4}, we have that
$$f\circ T_g \circ \pi_X 
= F(\widetilde{f}\circ S_g)\circ\pi_X
= \widetilde{f}\circ S_g
= F(\widetilde{f})\circ\pi_X\circ S_g
= f\circ \pi_X \circ S_g$$
for any $g\in G$ and any $f\in C(X)$.
It follows by Urysohn's lemma that
\begin{equation}\label{appb_eq.5}
    T_g\circ\pi_X=\pi_X\circ S_g
\end{equation}
for any $g\in G$. Therefore, we have proved that $W$ is an extension of $X$ 
with $\pi_X$ being a continuous topological factor map. 
\par
Similarly, by considering the
embedding $F^{-1}|_{C(Z)\circ\rho}: C(Z)\circ\rho \hookrightarrow C(W)$, 
there exists a unique surjective continuous map
$\pi_Z:W\to Z$ such that
\begin{equation}\label{appb_eq.6}
    f\circ\pi_Z=F^{-1}(f\circ\rho)
\end{equation}
for any $f\in C(Z)$, and
\begin{equation}\label{appb_eq.7}
    R_g\circ\pi_Z=\pi_Z\circ S_g
\end{equation}
for any $g\in G$. Hence, $W$ is also an extension of $Z$ with 
$\pi_Z$ being a continuous topological factor map.
\par
Now we will find a measure on $W$, with which $W$ will become a measurable
extension of $X$ and $Z$. Since $f\mapsto \int f\d\mu$ is a positive linear
functional on $\Aa$, there exists a unique probability measure $\nu$ on $W$
such that
$$\int_X f\d\mu = \int_W F^{-1}(f)\d\nu$$
for any $f\in\Aa$. By \eqref{appb_eq.3}, we have that $\nu$ is $S$-invariant, 
by \eqref{appb_eq.4}, we have $\pi_X\nu=\mu$ and by \eqref{appb_eq.6},
we have $\pi_Z\nu=m$. Consequently, $\pi_X$ and $\pi_Z$ are factor maps.
\par
The last thing in this first step is to show that $\pi_X$ is actually a
measurable isomorphism between $W$ and $X$ and thus, that the measure $\nu$
is ergodic. First, we want to extend \eqref{appb_eq.4} in $C(W)\simeq\Aa$.
For $f\in\Aa$, it holds $\int_W |F^{-1}(f)|^2\d\nu = \int_X |f|^2\d\mu$,
and $F^{-1}$ is an isometry from $\Aa$ (with the $L^2(X,\mu)$ norm) into $L^2(W,\nu)$.
Combining the facts that $C(X)$ is dense in $\Aa$ (with respect to the 
$L^2(X,\mu)$ norm) and that \eqref{appb_eq.4} holds for all $f\in C(X)$, we 
obtain that \eqref{appb_eq.4} holds for all $f\in\Aa$, $\nu$-almost always.
Then consider the map $H:L^2(X,\mu)\to L^2(W,\nu)$, such that 
$f\mapsto f\circ\pi_X$. Then $H(L^2(\mu))$ is closed in $L^2(W,\nu)$,
since the map is an isometry, and notice that it contains $F^{-1}(\Aa)=C(W)$.
Thus, $H(L^2(X,\mu))=L^2(W,\nu)$, showing that $\pi_X$ is a measurable 
isomorphism, and consequently, that $(W,\nu,S)$ is ergodic.
Finally, for any $f\in C(Z)$, using \eqref{appb_eq.4} and  
\eqref{appb_eq.6}, we see that
$f\circ\pi_Z=F^{-1}(f\circ\rho)=f\circ\rho\circ\pi_X$ holds 
$\nu$-almost always, and so, $\pi_Z=\rho\circ\pi_X$.
\par
\underline{Construction of the \Folner{} sequence}.
Since $(W,\nu,S)$ is ergodic, it follows that there exists $w_1\in\gen(\nu,\Phi)$
for some \Folner{} sequence $\Phi$.
\par
Set $x_1=\pi_X(w_1)$. Transitivity of $a$ implies that there exists a 
sequence $(h_N)_{N\in\N}\subset G$ such that 
\begin{equation}\label{ptw1}
    \lim_{N\to\infty}\sup_{g\in\Phi_N}d_X(T_gx_1,T_{gh_N}a)=0,
\end{equation}
where $d_X$ is the metric on the space $X$. Now set $z_1=\pi_Z(w_1)$. Let 
$E(Z,R)$ be the Ellis semigroup of $(Z,R)$, that is, the closure of $R$ as an
element of $Z^Z$, where this space is equipped with the pointwise convergence
topology. Let $R_0\in\overline{(R_{h_N})}_{N\in\N}\subset E(Z,R)$. 
By \cref{Kron_rot} $R$ is a rotation, which implies that is a bijection
from $Z$ to itself. In case that $\zmr$ is any distal system
(and not necessarily the Kronecker factor), then we also 
have that $R$ is a bijection
(see by \cite[Chapter 5]{auslander}). Therefore, there exists $z_0\in Z$ 
such that $R_0(z_0)=z_1$. Then there exists a subsequence of 
$(h_N)_{N\in\N}$, which, by abuse of notation, we denote as 
$(h_N)_{N\in\N}$, such that $\lim_{N\to\infty}R_{h_N}z_0=z_1$.
Therefore, there exists a further subsequence, which once again we denote
in the same way, for which it holds that
\begin{equation}\label{ptw2}     
\lim_{N\to\infty}\sup_{g\in\Phi_N}d_Z(R_gz_1,R_{gh_N}z_0)=0,
\end{equation}
where $d_Z$ is the metric on the space $Z$.
\par
Let $f_1\in C(X)$, $f_2\in C(Z)$. By \eqref{ptw1} and 
\eqref{ptw2}, we have that 
$$\lim_{N\to\infty}\sup_{g\in\Phi_N}|f_1(T_gx_1)-f_1(T_{gh_N}a)|=0
\text{\quad and \quad} 
\lim_{N\to\infty}\sup_{g\in\Phi_N}|f_2(R_gz_1)-f_2(R_{gh_N}z_0)|=0.$$
We define the \Folner{} sequence $\Psi=(\Psi_N)_{N\in\N}$,
by $\Psi_N = \Phi_Nh_N$ for any $N\in\N$.
It is easy to check that since $\Phi$ is a left \Folner{} sequence, 
then $\Psi$ is also a left \Folner{} sequence.
It follows from the above 
equations that 
\begin{align}\label{ptw3}
    & \lim_{N\to\infty}\bigg|\frac{1}{|\Psi_N|}\sum_{g\in\Psi_N}f_1(T_ga)f_2(R_gz_0)
    - \frac{1}{|\Phi_N|}\sum_{g\in\Phi_N}f_1(T_gx_1)f_2(R_gz_1)\bigg| \notag \\
    & = \lim_{N\to\infty}\bigg|\frac{1}{|\Phi_N|}\sum_{g\in\Phi_N}
    \big(f_1(T_{gh_N}a)f_2(R_{gh_N}z_0)-f_1(T_gx_1)f_2(R_gz_1)\big)\bigg| \notag \\
    & \leq \lim_{N\to\infty}\sup_{g\in\Phi_N}
    |f_1(T_{gh_N}a)f_2(R_{gh_N}z_0)-f_1(T_gx_1)f_2(R_gz_1)| \notag \\
    & \leq \lim_{N\to\infty}\sup_{g\in\Phi_N}
    \big(|f_1(T_{gh_N}a)f_2(R_{gh_N}z_0)-f_1(T_{gh_N}a)f_2(R_gz_1)| \notag \\
    &\hspace*{2.5cm}+|f_1(T_{gh_N}a)f_2(R_gz_1)-f_1(T_gx_1)f_2(R_gz_1)|\big) \notag \\
    & \leq \lim_{N\to\infty}\big(
    \|f_1\|_\infty\sup_{g\in\Phi_N}|f_2(R_{gh_N}z_0)-f_2(R_gz_1)| + 
    \|f_2\|_\infty\sup_{g\in\Phi_N}|f_1(T_{gh_N}a)-f_1(T_gx_1)|\big) \notag \\
    & = 0. 
\end{align}
Moreover, recalling that $w_1\in\gen(\nu,\Phi)$ and observing that 
$f_1\circ\pi_X\in C(W)$ and $f_2\circ\pi_Z\in C(W)$, we have that 
\begin{align}\label{ptw4}
    \lim_{N\to\infty}\frac{1}{|\Phi_N|} & \sum_{g\in\Phi_N}
    f_1(T_gx_1)f_2(R_gz_1)  =
    \lim_{N\to\infty}\frac{1}{|\Phi_N|}\sum_{g\in\Phi_N}
    f_1(T_g(\pi_X(w_1)))f_2(R_g(\pi_Z(w_1))) \notag \\
    & = \lim_{N\to\infty}\frac{1}{|\Phi_N|}\sum_{g\in\Phi_N}
    f_1(\pi_X(S_gw_1))f_2(\pi_Z(S_gw_1)) \text{\quad 
    (by \eqref{appb_eq.5},\eqref{appb_eq.7})} \notag \\
    & = \int_W (f_1\circ\pi_X)(f_2\circ\pi_Z)\d\nu \notag \\
    & = \int_W (f_1\circ\pi_X)(f_2\circ\rho\circ\pi_X)\d\nu
    \hspace*{2.23cm}\text{ (since } \pi_Z=\rho\circ\pi_X) \notag \\
    & = \int_W f_1\cdot(f_2\circ\rho)\d\mu
    \hspace*{3.87cm}\text{\quad (since } \pi_X\nu=\mu). 
\end{align}
Combining \eqref{ptw3} and \eqref{ptw4} yields the desired result.
The proof is complete.
\end{proof}

\bibliographystyle{abbrv}
\bibliography{refs}

\bigskip
\bigskip
\bigskip

\footnotesize
\noindent
Dimitrios Charamaras\\
\textsc{École Polytechnique Fédérale de Lausanne (EPFL)} \par\nopagebreak
\noindent
\href{mailto:dimitrios.charamaras@epfl.ch}
{\texttt{dimitrios.charamaras@epfl.ch}}

\bigskip 

\footnotesize
\noindent
Andreas Mountakis\\
\textsc{University of Crete} \par\nopagebreak
\noindent
\href{mailto:a.mountakis@uoc.gr}
{\texttt{a.mountakis@uoc.gr}}

\end{document}